\documentclass[11pt]{article}

\usepackage{amssymb,graphics,amsmath,amsthm,amsopn,amstext,amsfonts,color,fullpage,fancybox,hyperref,bm}
\usepackage{subfig,placeins} 
\usepackage{graphicx}
\usepackage{verbatim}
\usepackage{algorithm}
\usepackage{algpseudocode}
\usepackage{longtable}
\usepackage{multirow}

\newcommand{\gap}{\vspace{0.1in}}
\newcommand{\epc}{\hspace{1pc}}
\newcommand{\thalf}{{\textstyle{\frac{1}{2}}}}
\newcommand{\wh}{\widehat}
\newcommand{\wt}{\widetilde}

\newtheorem{theorem}{Theorem}[section]
\newtheorem{proposition}[theorem]{Proposition}
\newtheorem{lemma}[theorem]{Lemma}
\newtheorem{corollary}[theorem]{Corollary}
\newtheorem{definition}[theorem]{Definition}
\newtheorem{example}[theorem]{Example}
\newtheorem{remark}[theorem]{Remark}
\newtheorem{assumption}[theorem]{Assumption}

\newcommand{\E}{{\rm I\!E}}

\newcommand{\II}{{\rm I\!I}}

\newcommand{\zbld}{{\mbox{\boldmath $z$}}}

\newcommand{\IP}{{\rm I}\!{\rm P}}

\def\begar{$$\begin{array}{lll}}
\def\endar{\end{array}$$}
\def\begarlab{\begin{equation} \begin{array}{lll} \label}
\def\endarlab{\end{array} \end{equation}}

\def\ds1{{\mathrm{1 \hspace{-2.6pt} I}}}

\def\calD{{\cal D}}

\def\calF{{\cal F}}
\def\calG{{\cal G}}

\def\calL{{\cal L}}
\def\calM{{\cal M}}
\def\calN{{\cal N}}

\title{
	Statistical Analysis of Stationary Solutions of\\ Coupled Nonconvex Nonsmooth Empirical Risk Minimization}

\author{Zhengling Qi\thanks{Department of Decision Sciences, George Washington University, DC 20052.
		{\tt Email: qizhengling@gwu.edu}.}\and Ying Cui\thanks{The Daniel J.\ Epstein
		Department of Industrial and
		Systems Engineering, University of Southern California, Los Angeles, CA 90089.
		{\tt Emails: yingcui@usc.edu; jongship@usc.edu}.  The work of these two authors was based on research partially
		supported by the U.S.\ National Science Foundation grant IIS--1632971 and by the Air Force
		Office of Scientific Research under Grant Number FA9550-18-1-0382.} \and
	Yufeng Liu\thanks{Department of Statistics and Operations Research,
		University of North Carolina, Chapel Hill, NC 27599.
		{\tt Email: yfliu@email.unc.edu}.
		The work of this author was based on research partially supported by the U.S.\ National Science Foundation grant IIS-1632951
		and National Institute of Health grant R01GM126550.} \and Jong-Shi Pang\footnotemark[2]}

\begin{document}
	
\maketitle
	
\begin{abstract}
\noindent This paper has two main goals: (a) establish several statistical properties---consistency, asymptotic distributions,
and convergence rates---of stationary solutions and values of a class of coupled nonconvex and nonsmooth
empirical risk minimization problems, and (b) validate these properties by a noisy amplitude-based phase
retrieval problem, the latter being of much topical interest.
Derived from available data via sampling, these empirical risk minimization problems are 
the computational workhorse of a population risk model which involves the minimization of an expected value of a random functional.
When these minimization problems are nonconvex, the computation of their globally optimal solutions is elusive.
Together with the fact that the expectation operator cannot be evaluated for general probability distributions, it becomes
necessary to justify whether the stationary solutions of the empirical problems are practical approximations of the stationary 
solution of the population problem.  When these two features, general distribution and nonconvexity, are coupled with 
nondifferentiability that often renders the problems ``non-Clarke regular'', the task of the justification becomes challenging.
Our work aims to address such a challenge within an algorithm-free setting.
The resulting analysis is therefore different from the much of the analysis in the recent literature
that is based on local search algorithms.  Furthermore, supplementing the classical minimizer-centric analysis, our results
offer a first step to close the gap between computational optimization and asymptotic analysis of coupled nonconvex nonsmooth statistical
estimation problems, expanding the former with statistical properties of the practically obtained solution and providing the latter with
a more practical focus pertaining to computational tractability.

\end{abstract}
	
\medskip

\begin{center}
\parbox{0.95\hsize}{{\bf KEY WORDS:}\;  Statistical analysis; Consistency; Convergence rates; Directional stationarity; Asymptotic distribution;
Nonconvexity; Nonsmoothness; Phase retrieval problem.}
\end{center}

	
\section{Introduction}
Given a probability space $(\Omega, \calF, \IP)$, where $\Omega$ is the sample space, $\calF$ is the $\sigma$-field generated by $\Omega$,
and $\IP$ is the corresponding probability measure,
a parameterized random 
function ${\cal L} : \mathbb{R}^p \times \Omega \to \mathbb{R}$, and a compact convex set $X \subseteq \mathbb{R}^p$,
we consider the population risk minimization problem
\begin{equation}\label{eq:population}
\displaystyle\operatornamewithlimits{minimize}_{ x\in X}\;\, \mathcal{M}(x)\,  \triangleq \, \E_{\wt{\omega}}\, [\, \mathcal{L}(x;\wt{\omega})\,].
\end{equation}
In this setting, $\widetilde{\omega}$ is a random vector defined on the probability triple $(\Omega, \mathcal{A}, \mathbb{P})$;
the tilde on $\widetilde{\omega}$ signifies a random variable, whereas $\omega$ without the tilde will refer
to a realization of the random variable.  This convention of distinguishing a random variable and its realizations will be used throughout
the paper.
Subsequently, structure of ${\cal L}$ will be imposed for the purpose of analysis.
The expectation function in \eqref{eq:population} often does not have a closed form expression so that algorithms for solving  deterministic optimization
problems may not be directly applicable.  There are two classical Monte-Carlo sampling based approaches to solve the expected-value minimization
problem \eqref{eq:population}:
stochastic approximation (SA) and sample average approximation (SAA).  The SA proposed by Robbins and Monro~\cite{RobbinsMonro1951}
in the 1950s is a stochastic (sub)gradient method that updates each iterate along the opposite (sub)gradient direction estimated from
one or a small batch of samples.  It has attracted a great attention recently in machine leaning and stochastic programming communities,
partially due to its scalability and easy fitting to the online settings.  Interested readers are referred to~\cite{Chung54, Polyak90,Polyak92, robustSA09}
and the references therein for the development of the SA.
The SAA method, on the other hand,  takes $N$ independent and
identically distributed (i.i.d) random samples $\omega^1, \ldots, \omega^{\,N}$ with the same distribution
as $\omega$ and estimate the expectation function with the sample average approximation,
resulting in the empirical risk minimization or the M-estimation problem:
\begin{equation}\label{eq: saa}
\displaystyle\operatornamewithlimits{minimize}_{x \in X} \;\,  \mathcal{M}_{N}(x)\, \triangleq \,\displaystyle\frac{1}{N}\sum_{n = 1}^{N} \calL(x\,;\,\omega^{\,n}),
\end{equation}
	
	
There is a vast literature on the asymptotic analysis of the M-estimators/SAA solutions related to the optimal solution of the expectation
problem \eqref{eq:population} as the sample size $N$ goes to infinity.  The first celebrated consistency result  dates back to 1920s
by R.A.\ Fisher in \cite{ra1922mathematical,fisher1925theory} for the maximum likelihood estimation (MLE) problems.
An proof of the consistency of MLE is given by Wald in \cite{wald1949note}. Notice that the MLE is a special case of the
problem \eqref{eq:population} if we take the function $\mathcal{L}$ as the negative logarithm of probability density/mass functions.
Other important developments of the global optimal solutions of the M-estimation in  the statistical literature include
\cite{huber1967behavior,chernoff1954distribution,lecam1970assumptions,self1987asymptotic}.
The consistency and asymptotic distributions of the local optimal solutions for smooth optimization problems are studied by Geyer in \cite{geyer1994asymptotics}.
Most recently, Royset et al. \cite{Royset19,RoysetWets18} employed variational analysis to study statistical properties of M-estimators of
non-parametric problems.  In the field of  stochastic programming,
the study of the asymptotic behavior of the optimal solutions  begins with the work of Wets \cite{Wets79}, and is further developed in \cite{DupacovaWets88,Shapiro89}
with inequality constraints and nonsmooth objective functions using the tools from nonsmooth analysis.
Recently, the article \cite{DentchevaPenevRucz17} studies the statistical estimation of composite risk functionals and
risk optimization problems and establishes a central limit formula of their optimal values when an
estimator of the risk functional is used.
Interested readers are referred to the monographs \cite[Section~5.2]{van1998asymptotic} and \cite[Section~5]{ShapiroDR09} for comprehensive treatment of the
asymptotic analysis of the M-estimators/SAA solutions.  However, all these results pertain to the global or local minimizers
of the optimization problems or the (globally) optimal objective values, regardless of the possibility
that the latter problems may be nonconvex. 
Since in general one cannot find a global or local optimal solution to the
nonconvex optimization problems, any consistency results that are based on the global or local minimizers are at best ideal targets for such problems
and have little practical significance.   The situation becomes more serious
when nondifferentiability is coupled with nonconvexity because there is a host of stationary solutions of the resulting optimization problems.
Typically, the sharper the stationarity solution is (sharp in the sense of least relaxation in its definition), the more difficult it is to compute.
It is thus important to understand whether in practice, the focus should be placed on computing sharp stationary solutions (which
distinguish themselves as being the ones that must satisfy all other relaxed definitions of stationarity)
that potentially require higher computational costs versus computing some less demanding solutions.
Our derived results show that the sharpness of the stationarity at the empirical level is preserved at the population level,
thus favoring the former.  Furthermore, via a noisy amplitude-based phase retrieval problem that is of
much topical interest, we demonstrate that a stationary point of a relaxed kind can have no bearings to a minimizer, both in the
population and empirical problems.
In short, there is presently a gap in the literature between the asymptotic minimizer-centric analysis of statistical
estimation problems in the presence of (coupled) nonconvexity and nondifferentiability
and the computational tractability of the solutions being analyzed.  Our work offers a first step in closing this gap.
	
\gap
	
When the expected-value objective function ${\cal M}$ in \eqref{eq:population} is differentiable, the stationary points
of problem \eqref{eq:population} can be characterized by the solutions of the following stochastic generalized equation
\[
0 \, \in \, \nabla \, \E_{\wt{\omega}}\left[ \mathcal{L}(x;\wt{\omega}) \right] + \mathcal{N}(x;X),
\]
where $\mathcal{N}(x;X)$ denotes the normal cone of  $X$ at $x\in X$ as in convex analysis, see, e.g., \cite{RockafellarConvex}.
Similarly, a stationary point  of the empirical risk minimization \eqref{eq: saa} satisfies
\[
0 \, \in \,\displaystyle\frac{1}{N}\sum_{n = 1}^{N} \nabla_x \, \mathcal{L}(x;\omega^{\, n}) + \mathcal{N}(x;X).
\]
The consistency and asymptotic distributions of the solutions for such a stochastic generalized equation have been established
in the literature such as \cite{KingRockafellar93,GurkanRobinson99,Shapiro03}.  See also \cite{MeiMontanari17} for the correspondence of
stationary solutions between the empirical risk and the population risk when the sample size is sufficiently large.
	
\gap
	
While the consistency of the global optimal values and solutions is mainly due to the uniform law of large numbers for real-valued
random functions, the consistency of the stationary solutions of nonconvex nonsmooth problems needs the uniform law of large numbers
for set-valued subdifferentiable mappings.
It is well known that Attouch's celebrated theorem  on the equivalence of the epiconvergence of a sequence of convex functions
and the graphical convergence of the subdifferential \cite{Attouch84} fails for general nonconvex functions, which makes the
asymptotic analysis for the SAA a challenging task when applied to a nonconvex problem.  For a special case where the function
$\mathcal{L}(\,\bullet\,,\omega)$ is Clarke regular  \cite[Section 2]{Clarke83} for almost all $\omega\in \Omega$, the uniform
law of large numbers for random set-valued Clarke regular mappings is established in \cite{shapiro2007uniform}
and the consistency of Clarke stationary points is also provided therein.
	
	\gap
	
	Many modern statistical and machine learning problems consist of inherently coupled nonconvex and nonsmooth objective functions.
	More specifically, the objective functions therein cannot be decomposed into either the sum of a smooth nonconvex function and a nonsmooth convex function,
	or the composition of a convex function and a smooth function; see the examples in Section~\ref{sec: example}.
	Such functions often fail to satisfy the Clarke regularity so that the results in \cite{shapiro2007uniform} are no longer valid.
	In particular, the inclusions \eqref{eq:inclusion1} and \eqref{eq:inclusion2} can be strict.
	Furthermore, the classical (let alone uniform) law of large numbers of random variables cannot be easily extended to such random functions.
	Adding to this difficulty, the discontinuity of $\partial \calM$ results in the possible failure of the continuous convergence of the sample average functions.
	Back to the optimization problem in \eqref{eq: saa}, a natural way to tackle the nondifferentiable objective function seems to be the smoothing approach.
	Xu and Zhang \cite{XuZhang09} show that the stationary point of the smoothed problem converges to a so-called weak (Clarke) stationary point of the
	original expectation problem.  This is a very nice theoretical result.  However, the Lipschitz constant of the gradient of the smoothed problem
	goes to infinity as the smoothing parameter goes to zero.  This fact makes it difficult for the smoothed version of \eqref{eq: saa} to be solved
	efficiently by either gradient-type or Newton-type methods, thus weakening the practical significance of the mentioned convergence result.
	
	\gap
There is an increasing literature that are focused on studying the convergence of a particular algorithm for nonconvex M-estimation problems
with the guarantee of statistical accuracy.  For example, relying on the restricted strong convexity, the references
\cite{loh2011high,loh2015regularized,loh2017statistical} show that gradient decent method with a proper initialization converges to the
statistical ``truth" for different regression models with nonconvex objective functions.  Adding to these references, the
paper \cite{MeiMontanari17} recently establishes a one-to-one correspondence of stationary solutions of non-convex M-estimation problems
by analyzing the landscape of the empirical problem.  However, existing literature relies heavily on the smoothness of M-estimation problems
and their special structure such as restricted strong convexity, which limit their applications on analyzing a broad class of modern statistical
and machine learning problems, such as the examples in Section~\ref{sec: example}.
	
\gap
	
In this work, we are taking a first step to establish the consistency  of the stationary point for a  class of coupled nonconvex and nonsmooth
empirical risk minimization problems.  Our focus is placed on the
asymptotic behavior of the directional stationary points of problem \eqref{eq: saa}, which distinguish themselves as being the sharpest kind among	
all stationary solutions of such objectives,
such as the Clarke stationarity that defined in \eqref{Clarke:saa}.  We consider a class of composite functions $\calL$ that covers a wide range
of practical applications spanning modern statistical estimation and machine learning.  For problems in this class, it has been shown
in \cite{CuiPangSen18} that their empirical directional stationary points are computationally tractable by iteratively solving 
convex subprograms.
Our results demonstrate that the additional efforts as required by the algorithm in the latter reference
for computing the empirical directional stationary point of a sharp kind
pay off not only at the empirical level, but also at the population level.  It should be noted that
our general analysis is independent of particular algorithms and thus is broadly applicable.
Finally, we apply our developed theory to the noisy amplitude-based phase retrieval problem
and show that every empirical directional stationary point, which can be computed by an algorithm described in \cite{CuiPangSen18},
is $\sqrt{N}$-consistent to a global minimizer of the corresponding population problem.
As our approach is algorithm-free, the analysis is different from much of the existing literature
such as \cite{ma2019optimization} that requires algorithm-based local search.


\gap

To summarize, the contributions of this paper are as follows:

\gap

$\bullet $ we directly address the asymptotic convergence of the SAA stationary solutions
for nonconvex nondifferentiable problems without Clarke regularity of the objective function, and establish results that are
not linked to particular algorithms;

\gap

$\bullet $ we establish the consistency and derive the convergence rate of empirical local minimizers to population local minimizers
for a class of composite nonconvex, nonsmooth, and non-Clarke regular functions;

\gap

$\bullet $ we apply our derived results to a topical problem to support the value of this kind of algorithm-free statistical
analysis which can be validated by a rigorous algorithm if needed.
	
\section{Problem Structures and Examples} \label{sec: example}
	
Many practical statistical estimation and machine learning problems, even though with nonconvex and nondifferentiable objective functions,
often have special structures. Supervised learning is a class of machine learning problems that infers a function to map inputs $\xi :\Xi \to \mathbb{R}^{d}$
to the outputs $\zbld:\mathcal{Z} \to \mathbb{R}$,  jointly defined on the probability  space $(\Omega, \calF, \IP)$, where $\Omega = \Xi \times \mathcal{Z}$.
The objective function of the supervised learning  takes the form of
\begin{equation}\label{eq: cvx composite dc}
\mathcal{L}(x;\xi,\zbld) \,\triangleq\,  h\circ \left(m(x;\xi)\, ;\, \zbld \right),
\end{equation}
where $h(\,\bullet\,;\,\zbld):\mathbb{R}\to\mathbb{R}$ is a univariate loss function measuring the error between a possibly
nonconvex nondifferentiable statistical model $m(\,\bullet\,; \,\xi):\mathbb{R}^p\to\mathbb{R}$ with the input feature $\xi$ and the output response $\zbld$.
In fact, the above function  can also be interpreted as an unsupervised learning model  when  the random variable $\zbld$ is absent.  In the notation of
(\ref{eq:population}), the pair $(\xi,\zbld)$ constitutes the random variable $\omega$.  At this juncture, we should clarify our convention of the probability
triple $(\Omega, \calF, \IP)$ projected onto the input and output spaces $\Xi$ and ${\cal Z}$, especially when we want to discuss about properties of the
function $m(x;\xi)$ which involves the input variable $\xi \in \Xi$ only.  Letting $P_{\Xi} : \Omega \to \Xi$ be the natural projection of the Cartesian product
$\Omega = \Xi \times {\cal Z}$ onto $\Xi$, an arbitrary subset $S \subseteq \Xi$ can be associated with its inverse image in $\Omega$ under $P_{\Xi}$; a statement
such as ``$S$ has measure one'' then means that $P_{\Xi}^{-1}(S)$, which is a subset of $\Omega$, has measure one.  A similar meaning holds if $S$ is a subset of
${\cal Z}$.  In the rest of this paper, this convention is applied to almost sure events in the spaces $\Xi$ and ${\cal Z}$.  We say that a subset
(of either $\Xi$, ${\cal Z}$, or $\Omega$) is a
{\sl probability-one set} if its probability measure is one.

	\gap
	
	We are particularly interested in  a class of difference-of-max-convex parametric model $m(\,\bullet\,; \,\xi)$ with the form of
	\begin{equation} \label{eq:dcp}
	m(x ; \xi) \, \triangleq \, \underbrace{\displaystyle{
			\max_{1 \leq j \leq k_f}
		} \, f_{j}(x ; \xi)}_{\mbox{denoted $f(x ; \xi)$}} \, - \, \underbrace{\displaystyle{
			\max_{1 \leq j \leq k_g}
		} \, g_{j}(x ;\xi)}_{\mbox{denoted $g(x ; \xi)$}},
	\end{equation}
	where each  $f_j(\,\bullet ; \xi)$ and $g_j(\,\bullet ; \xi)$ are convex differentiable functions from $\mathbb{R}^p$ to $\mathbb{R}$.
	This model is pervasive in the contemporary fields of data science. Below we list two such applications.
	
	\begin{example}[Piecewise affine regression]\rm
		Linear regression is perhaps the  simplest parametric model to estimate the relationship between  the response variable $\zbld$ and the covariate information $\xi$.
		Piecewise linear regression is a generalization of the classical linear regression to  enhance the model flexibility.
		It is known that every piecewise affine function can be written in the form of
		\[
		m(x\,; \,\xi) \, = \, \max_{1\leq j\leq k_f} \,\left( (a^{\,j})^{\top} \xi +\alpha_j\,\right) \,  - \, \max_{1\leq j\leq k_g} \,\left( (b^{\,j})^{\top} \xi + \beta_j\,\right)
		\]
		with the parameter $x \,\triangleq \,\left\{ \left( a^{\,j}, \alpha_j \right)_{j=1}^{k_f}, \left( b^{\,j}, \beta_j \right)_{j=1}^{k_g} \right\}
		\in \mathbb{R}^{(k_f + k_g)(d + 1)}$ \cite{Scholtes02}.
		Obviously, this piecewise affine model is a special case of the model \eqref{eq:dcp}.
		Taking the quadratic function $h(\,\bullet\,; \zbld) = (\zbld - \bullet)^2$ as the loss measure to estimate the parameter $x$,
		we obtain the following optimization problem
		\[
		\begin{array}{lll}
		\hspace{0.3in} & \displaystyle{
			\operatornamewithlimits{\mbox{minimize}}_x
		} \quad &  \E_{\wt{\omega}} \, \left[ \, \wt{\zbld} - \displaystyle\max_{1\leq j\leq k_f} \,\left((a^{\,j})^{\top} \wt{\xi} + \alpha_j\right) +
		\displaystyle\max_{1\leq j\leq k_g} \,\left((b^{\,j})^{\top} \wt{\xi} + \beta_j\,\right)  \,\right]^2 \\ [0.25in]
		& \mbox{subject to} \quad & x = \left\{(a^j, \alpha_j)_{j=1}^{k_f}, (b^j, \beta_j)_{j=1}^{k_g}\right\} \, \in \, X \, \subseteq \, \mathbb{R}^{(k_f + k_g)(d+1)}.
		\end{array}
		\]
		Notice that the overall objective function in the above optimization problem is nonconvex.  More seriously, the nonconvexity and nondifferentiability
		within the square bracket are coupled.  In the special case of the ReLu function,
		which is basically the plus function (see Example~\ref{example:ReLu} below), it was shown in \cite[Lemma~57 and below]{JinJordan18}
		the expected-value function is not differentiable at the point $x = 0$.
		
		\gap
		
		Alternatively, we may take the least absolute deviation as the loss function $h(\,\bullet\,;\zbld)=  |\zbld - \bullet|$
		and consider the robust piecewise affine regression problem
		\[
		\begin{array}{lll}
		\hspace{0.3in} & \displaystyle{
			\operatornamewithlimits{\mbox{minimize}}_{x }
		} \quad &  \E_{\wt{\omega}} \, \left[ \,\left|\, \wt{\zbld} - \displaystyle\max_{1\leq j\leq k_f} \,\left((a^{\,j})^{\top} \wt{\xi} + \alpha_j\right) +
		\displaystyle\max_{1\leq j\leq k_g} \,\left((b^{\,j})^{\top} \wt{\xi} + \beta_j\,\right) \,\right| \,\right] \\ [0.25in]
		& \mbox{subject to} \quad & x = \left\{(a^j, \alpha_j)_{j=1}^{k_f}, (b^j, \beta_j)_{j=1}^{k_g}\right\} \, \in \, X \, \subseteq \, \mathbb{R}^{(k_f + k_g)(d+1)},
		\end{array}
		\]
		which is again a nonconvex and nonsmooth stochastic optimization problem.
	\end{example}
	
	\begin{example}[2-layer neural network model with the ReLu activation function] \label{example:ReLu}
		\rm
		
		Consider a 2-layer neural network model with the rectified linear unit (ReLU) activation function that takes the form of
		\begin{equation}\label{defn: 2-layer ReLu}
		m(x;\xi) \, \triangleq \, \max\left( \, b^{\top} \max\left( A\xi + a, 0 \right) + \beta, \, 0 \right),
		\end{equation}
		where  $x$ consists of the two vectors $b$ and $a$ each in $\mathbb{R}^k$, the matrix $A \in \mathbb{R}^{k \times d}$,
		and scalar $\beta \in \mathbb{R}$.  The two occurrences of the max ReLu functions indicate the action of 2 hidden layers,
		where the ``max'' operation of $A\zbld + a$ and $0$ is taken componentwise.  Variation of the model where only the first layer
		is subject to the ReLu activation and extensions to more than 2 layers can be similarly treated, although the latter leads to
		much more complicated formulations.  No matter what loss function $h(\,\bullet\,;\zbld)$ takes,
		the square loss, the cross entropy function or the huber loss, the overall objective of $\mathcal{L}$ admits a coupled nonconvex and
		nonsmooth structure that is challenging to handle.  Nevertheless, we show below that the function \eqref{defn: 2-layer ReLu} can be
		expressed as the difference of two convex piecewise continuously differentiable functions, thus reducing it to a special case of
		the model \eqref{eq:dcp}.  For notational simplicity, we omit the vector $a$ since it can be absorbed in $A$ as an extra
		column with $\xi$ redefined by $(\xi,1)\in \mathbb{R}^{d+1}$.  With this simplification, we derive
		$$
		\begin{array}{rl}
		m(x\,;\,\xi) \, = & \max\big( \, \max(b^{\top},\,0)\,\max(A\xi,\,0) - \max(-b^{\top}, \,0)\,\max(A\xi,\,0) + \beta\,,\, 0 \, \big) \\ [0.15in]
		= &\displaystyle\thalf \, \displaystyle\max\big( \, \| \, \max(b\,,\,0)+ \max(A\xi,\,0) \, \|^2+ \| \, \max(-b\,,\, 0) \, \|^2 + 2\beta \,- \\[ 0.1in]
		&\qquad\qquad\qquad\qquad  \| \, \max(-b\,,\, 0) + \max(A\xi,\,0) \, \|^2 - \| \, \max(b\,,\,0) \, \|^2, \, 0 \, \big) \\ [0.1in]
		=  & \thalf \,\max\left( \begin{array}{l}
		\| \, \max(b,\, A\xi,\,b+A\xi,\,0 ) \, \|^2 + \| \, \max(-b,\,0) \, \|^2 + 2\beta, \\ [0.1in]
		\| \, \max(-b\,,A\xi,\, -b+A\xi,\, 0) \, \|^2 + \| \, \max(b, \, 0) \, \|^2
		\end{array} \right) - \\ [0.25in]
		& \thalf \, \left( \, \| \, \max(-b\,,A\xi,\, -b+A\xi,\, 0) \, \|^2 + \| \, \max(b\,,\,0)\|^2 \, \right).
		\end{array}
		$$
		Although the terms $\| \max( \pm b,\, A\zbld, \, \pm b+A\zbld, \,0 ) \|^2$ are not differentiable,
		they can each be represented as the pointwise maximum of finitely many convex differentiable functions.  In fact, we have,
		with $A_{i \bullet}$ denoting the $i$-th row of the matrix $A$,
		\[ \begin{array}{l}
		\left\| \, \max( \, \pm b\,,A\xi,\, \pm b+A\xi,\, 0 \, ) \, \right\|^2 \\ [5pt]
		\epc = \, \displaystyle{
			\sum_{i=1}^k
		} \, \max\left( \, \max( \, \pm b_i, \, 0 \, )^2, \,
		\max( \, A_{i \bullet}\xi, \, 0 \, )^2, \, \max( \, \pm b_i + A_{i \bullet}\xi, \, 0 \, )^2 \, \right)
		\\ [5pt]
		\epc = \, \displaystyle{
			\max_{(\lambda_{1}, \lambda_2, \lambda_3) \in \Delta}
		} \, \left\{ \,\displaystyle{
			\sum_{i=1}^k
		} \, \left[ \, \underbrace{\lambda_{1,i} \, \max( \pm b_i,0 )^2
			+ \lambda_{2,i} \, \max( A_{i \bullet}\xi, 0 )^2 + \lambda_{3,i} \, \max( \pm b_i + A_{i \bullet}\xi,0 )^2}_{\mbox{nonnegative, convex, differentiable}}
		\, \right] \, \right\},
		\end{array} \]
		where $\Delta \triangleq \left\{(\lambda_1, \lambda_2, \lambda_3)\in \{0,1\}^{3k}\,\bigg|\,
		\displaystyle\sum_{j=1}^3\lambda_{j,i} = 1, \; \forall\; i = 1,\ldots, k\,\right\}$ is a finite set of binary indicators.
		Substituting the above expression into the function
		$m(x\,;\,\xi)$, we see that this function can be written in the form of \eqref{eq:dcp}
		for some positive integers $k_f$ and $k_g$ and convex functions $f_j(x;\xi)$ and $g_j(x;\xi)$ that involve the squared
		plus function: $t_+^2 \triangleq \max(t,0)^2$ for $t \in \mathbb{R}$; it is easy to check that the latter univariate
		function is convex, once but not twice continuously differentiable.
	\end{example}
	
	\section{Concepts of Stationarity} \label{sec:phase}
	Our primary focus in this paper is on the consistency of a sharp kind of stationary solutions
	of the M-estimation problem \eqref{eq: saa}, which we term a directional stationary point.
	Let $\varphi$ be a locally Lipschitz continuous function defined on an open set $S \subseteq \mathbb{R}^p$.
	The one-sided directional derivative of $\varphi$ at the vector $x \in \mathbb{R}^p$ along
	the direction $v \in \mathbb{R}^p$ is defined as
	\[
	\varphi^{\, \prime}(x;v) \, \triangleq \, \lim_{\tau \downarrow 0} \, \frac{\varphi(x + \tau \, v) - \varphi(x)}{\tau}
	\]
	if the limit exists; $\varphi$ is said to be directionally differentiable at $x \in S$ if $\varphi^{\, \prime}(x\, ;\, v)$ exists for all $v \in \mathbb{R}^p$.
	Recalling that the set $X$ is assumed convex, we say $\bar{x}\in X$ is a  {\sl d(irectional)-stationary}
	point of the program $\displaystyle\operatornamewithlimits{minimize}_{x\in X} \, \varphi(x)$  if
	\[
	\varphi^{\, \prime}(\bar{x}\, ;\, x - \bar{x}) \, \geq \, 0, \epc \forall \ x \, \in \, X.  
	\]
	The d(irectional)-stationary point, in its dual form, satisfies
	\[
	0 \, \in \, \wh{\partial} \left(\,\varphi(\bar{x}) + \delta_X(\bar{x})\,\right),
	\]
	where $\delta_X(\bar{x})$ is the indicator function of the set $X$; i.e., $\delta_X(x) \triangleq \left\{ \begin{array}{ll}
	0 & \mbox{if $x \in X$} \\
	\infty & \mbox{otherwise}
	\end{array} \right.$ and
	\[
	\wh{\partial} \, \phi(\bar{x}) \, \triangleq \, \left\{ \, v \, \in \mathbb{R}^p \; \bigg| \;
	\liminf_{\bar{x} \neq x\to\bar{x}} \; \frac{\phi(x) - \phi(\bar{x})-v^{\top}(x-\bar{x})}{\| \, x-\bar{x} \, \|} \, \geq \, 0 \, \right\}
	\]
	is the {\sl regular subdifferential} of an extended-value function $\phi:\mathbb{R}^p\to (\infty, +\infty]$ \cite[Section 8.B]{RockafellarRWets98}.
	A d-stationary point is in contrast to a C(larke)-stationary point \cite{Clarke83} which by definition satisfies
	\[
	0 \, \in \, \partial_C \left(\,\varphi(\bar{x}) + \delta_X(\bar{x})\,\right),
	\]
	where the Clarke subdifferential is:
	\[
	\partial_C \, \phi(\bar{x})  \,= \, \left\{ \, v \, \in \, \mathbb{R}^p \; \bigg| \;
	\limsup_{x\to\bar{x},\; t\downarrow 0}  \; \frac{\phi(x+tw) - \phi(x)-t\,v^{\top}w}{t} \, \geq \, 0, \quad \forall \; w \, \in \, \mathbb{R}^p \, \right\}.
	\]
	Unlike the Clarke subdifferential $\partial_C  \,\phi$ which is outer semicontinuous \cite[Proposition 2.1.5]{Clarke83}; the regular subdifferential
	mapping is not ``robust''.
	This is one source of difficulty for analyzing the consistency of the d-stationarity for problem \eqref{eq: saa} in its general form.
	Yet, as we will demonstrate in Section~\ref{sec:phase} via a practical example, analyzing the consistency of a C-stationary point could be meaningless
	as far as a (local) minimizer is concerned.  For evaluation purposes, we note that
	\begin{equation} \label{eq:Clarke subdiff}
	\begin{array}{ll}
	\partial_C \, \phi(\bar{x}) \, = & \mbox{convex hull of} \\ [5pt]
	& \left\{ \, \displaystyle{
		\lim_{k \to \infty}
	} \, \nabla \phi(x^k) \, \mid \,
	\mbox{each $x^k$ is a differentiable point of $\phi$ and $\displaystyle{
			\lim_{k \to \infty}
		} \, x^k \, = \, x$} \, \right\}.
	\end{array} \end{equation}
	In the context of \eqref{eq:population} with a convex $X$, $\bar{x} \in X$ is a C-stationary point if
	\begin{equation}\label{Clarke:saa}
	0 \, \in \,   \partial_C \, \calM(\bar{x}) + \mathcal{N}(\bar{x};X) ,
	\end{equation}
	where, as in standard convex analysis, ${\cal N}(\bar{x};X)$ is the normal cone of $X$ at $\bar{x}$.
	Similarly, we say $\bar{x}\in X$ is  a C-stationary point of \eqref{eq: saa} if
	\[
	0 \, \in \, \partial_C  \left(\displaystyle\frac{1}{N}\sum_{n = 1}^{N} \, \mathcal{L}(\bar{x}\,;\,\omega^{\, n})\right) + \mathcal{N}(\bar{x};X),
	\]
	where the Clarke subdifferential is taken with respect to the variable $x$.
	Notice that in general we have the inclusions
	\begin{equation}\label{eq:inclusion1}
	\partial_C \, \calM(x) \, \subseteq \, \E_{\wt{\omega}}\left[ \, \partial_C \, \calL(x;\wt{\omega}) \, \right],
	\end{equation}
	where $\E$ is taking as the Aumann integration (also called the selection expectation) \cite[Definition~1.12]{Molchanov05}, and
	\begin{equation}\label{eq:inclusion2}
	\partial_C  \left(\displaystyle\frac{1}{N}\sum_{n = 1}^{N} \,\mathcal{L}(x;\omega^{\, n})\right) \, \subseteq \,
	\displaystyle\frac{1}{N}\,\sum_{n = 1}^{N} \,\partial_C \, \mathcal{L}(x;\omega^{\, n}).
	\end{equation}
	When both of the functions $\calM$ and $\calL$ are Clarke regular, the above two inclusions become equality.
	The consistency of C-stationary points under Clarke regularity is  established in \cite{shapiro2007uniform}.

	\section{The Composite Difference-max Estimation Problem}
	
	In the rest of this paper, we focus on the coupled nonconvex nonsmooth program \eqref{eq: saa} with the loss function ${\cal L}$ given by
	the composite function (\ref{eq: cvx composite dc}) where $h(\,\bullet\,;\zbld)$ is a nonnegative convex function and the model
	$m(\,\bullet; \xi)$ is given by \eqref{eq:dcp}.  The nonnegativity condition of $h$ is satisfied by practically all the interesting
	applications in machine learning and statistical estimation.  The special form of the statistical model $m$ can be exploited to
	characterize d-stationarity in terms of certain convex programs.  Specifically, we consider the empirical risk minimization problem:
	\begin{equation}\label{eq:empirical composite saa}
	\displaystyle{
		\operatornamewithlimits{minimize}_{x \in X}
	} \  \mathcal{M}_{N}(x)\, \triangleq \, \displaystyle{
		\frac{1}{N}
	} \, \displaystyle{
		\sum_{n = 1}^{N}
	} \, \calL(x;\xi^{\, n};\zbld^{\, n}), \epc \mbox{with }
	\calL(x;\xi^{\, n};\zbld^{\, n}) \, \triangleq \, h( m(x;\xi^{\, n});\zbld^{\, n} ),
	\end{equation}
	where $m(x\,;\xi^{\, n})$ is given by (\ref{eq:dcp}), as a sample average approximation of the population model
	\begin{equation}\label{eq:population composite saa}
	\displaystyle{
		\operatornamewithlimits{minimize}_{x \in X}
	} \ \mathcal{M}(x) \,\triangleq \, \E_{\wt{\omega}}\, \left[\, \mathcal{L}(x;\wt{\xi};\wt{\zbld}) \, \right].
	\end{equation}
	Before proceeding to the mathematical analysis, we should highlight the main technical challenges associated with the
	above problems.  Foremost among these is a workable understanding and characterization of d-stationarity to facilitate the
	analysis.  It turns out that such a characterization (see Lemma~\ref{lemma:d-stat}) is available that involves (a) linearizations
	of the functions $f_j(\bullet;\xi)$ and $g_j(\bullet;\xi)$, and (b) the maximizing index sets of the functions $f(\bullet;\xi)$ and $g(\bullet;\xi)$
	(see below), both varying randomly due to the variable $\xi$.   When embedded in the expectation, such random variations, especially the index sets
	over which the linearizations are to be chosen, are not easy to treat.  Our approach is to employ a notion of stationarity
	(see Subsection~\ref{subsec:composite stationarity})
	that on one hand is computationally tractable and on the other
	hand is not overly relaxed as Clarke stationarity, which as illustrated by the phase retrieval problem, can be practically meaningless.  This constitutes
	the main contribution of our work.
	
	\gap
	
	Throughout, several assumptions will be imposed; the first of which is the following finite mean assumption: for every $x \in X$,
	\[
	\E_{\wt{\omega}}\left[\, \mathcal{L}(x;\wt{\xi};\wt{\zbld})  \, \right] \, < \, +\infty.
	\]
	For any $\xi \in \Xi$ and any nonnegative scalar $\varepsilon$, we consider the ``$\varepsilon$-argmax'' indices of the pointwise max functions
	$f$ and $g$ in \eqref{eq:dcp} as elements of the following two sets:
	\[
	\left\{\begin{array}{ll}
	{\cal A}_{f;\varepsilon}(x;\xi) \,\triangleq\, \left\{\, 1\leq \bar{j}\leq k_f \, | \;
	f_{\, \bar{j}\, }(x; \xi) \geq \underset{1 \leq j \leq k_f}{\max} \; f_j(x; \xi) - \varepsilon \right\} \\[0.2in]
	{\cal A}_{g;\varepsilon}(x;\xi) \,\triangleq\, \left\{\, 1\leq \bar{j} \leq k_g \; | \;
	g_{\, \bar{j}\, }(x; \xi) \geq \, \underset{1 \leq j\leq k_g}{\max} \; g_j(x; \xi) - \varepsilon \right\},
	\end{array}\right.
	\]
	respectively.  If $\varepsilon = 0$, the above sets reduce to the ``argmax'' indices of $f$ and $g$, for which we omit the subscript $\varepsilon$ and write them as
	\begin{equation}\label{def:active}
	\left\{\begin{array}{ll}
	{\cal A}_{f}(x;\xi) \,\triangleq\, \left\{\, 1\leq \bar{j}\leq k_f \; | \; f_{\,\bar{j}\,}(x; \xi)  = \underset{1 \leq j \leq k_f}{\max} \;
	f_j(x; \xi) \,\right\} \\ [0.2in]
	{\cal A}_g(x;\xi) \,\triangleq\, \left\{\, 1\leq \bar{j}\leq k_g \; | \; g_{\,\bar{j}\,}(x; \xi) \,= \underset{1 \leq j \leq k_g}{\max} \; g_j(x; \xi) \;\right\}.
	\end{array}\right.
	\end{equation}
	Notice the if $f_{\,\bar{j}\,}(x;\xi) = \underset{1 \leq j \leq k_f}{\max} \; f_j(x;\xi)$ for all $\bar{j} \in \{ 1, \cdots, k_f \}$, then
	${\cal A}_{f;\varepsilon}(x \, ;\, \xi) = {\cal A}_f(x \, ;\, \xi) = \{ 1, \cdots, k_f \}$ for all $\varepsilon \geq 0$.  A similar remark
	applies to the family of $g$-functions.  In general, the above-defined
	index sets have the inclusion property stated in the lemma below wherein $\mathbb{B}_\delta(\bar{x})$ denotes the (closed) Euclidean
	ball with center at $\bar{x}$ and radius $\delta > 0$.
	
	\begin{lemma}\label{lm:set inclusion}\rm
		Suppose that there exist positive constants $\mbox{Lip}_f(\xi)$, $\mbox{Lip}_g(\xi)$ and $c_0$ and a probability-one subset $\Xi^{\, 1}$
		of $\Xi$ such that for all $\xi \in \Xi^{\, 1}$,
		$\max\left( \, \mbox{Lip}_f(\xi), \, \mbox{Lip}_g(\xi) \, \right) \leq c_0$, and for all $x^1$ and $x^2$ in $X$,
		\begin{equation} \label{eq:Lipschitz fg}
		\begin{array}{lll}
		| \, f_j(x^1;\xi) - f_j(x^2;\xi) \, | & \leq & \mbox{Lip}_f(\xi) \, \| \, x^1 - x^2 \, \|_2, \epc \forall \ j \, = \, 1, \cdots, k_f, \\ [0.1in]
		| \, g_j(x^1;\xi) - g_j(x^2;\xi) \, | & \leq & \mbox{Lip}_g(\xi) \, \| \, x^1 - x^2 \, \|_2, \epc \forall \ j \, = \, 1, \cdots, k_g.
		\end{array} \end{equation}
		Then, for every scalar $\varepsilon > 0$,
		a scalar $\delta > 0$ exists such that for all
		$\varepsilon^{\, \prime} \in [ 0, \varepsilon ]$, all $\xi \in \Xi^{\, 1}$, and all pairs $x^1$ and $x^2$ in $X$
		satisfying $\| x^1 - x^2 \|_2 \leq \delta$,
		it holds that
		${\cal A}_{f;\varepsilon^{\, \prime}}(x^1; \xi) \subseteq {\cal A}_{f;2\varepsilon}(x^2;\xi)$ and
		${\cal A}_{g;\varepsilon^{\, \prime}}(x^1; \xi) \subseteq {\cal A}_{g;2\varepsilon}(x^2; \xi)$.
	\end{lemma}
	
	\begin{proof}  In what follows, the random realization $\xi$ is restricted to be in the set $\Xi^{\, 1}$.
		For any index $j = 1, \cdots, k_f$, we have
		\[
		f_j(x^1;\xi) \, = \, f_j(x^2;\xi) + \left[ \, f_j(x^1;\xi) - f_{\bar{j}}(x^2;\xi) \, \right]
		\, \leq \, f_j(x^2;\xi) + \mbox{Lip}_f(\xi) \, \| \, x^1 - x^2 \, \|_2;
		\]
		similarly,
		\[
		\underset{1 \leq j \leq k_f}{\max} \; f_j(x^1; \xi) \, \geq \,
		\underset{1 \leq j \leq k_f}{\max} \; f_j(x^2; \xi) - \mbox{Lip}_f(\xi) \, \| \, x^1 - x^2 \, \|_2.
		\]
		Thus, for $\bar{j} \in {\cal A}_{f;\varepsilon^{\, \prime}}(x^1; \xi)$, since
		$f_{\bar{j}}(x^1;\xi) \, \geq \, \underset{1 \leq j \leq k_f}{\max} \; f_j(x^1; \xi) - \varepsilon^{\, \prime}$, we deduce,
		for any positive $\delta \leq \displaystyle{
			\frac{\varepsilon}{2 \, c_0}
		}$ and provided that $\| x^1 - x^2 \|_2 \leq \delta$,
		\[ \begin{array}{lll}
		f_{\bar{j}}(x^2;\xi) & \geq & \underset{1 \leq j \leq k_f}{\max} \; f_j(x^2; \xi) - 2 \, \mbox{Lip}_f(\xi) \, \| \, x^1 - x^2 \, \|_2 - \varepsilon^{\, \prime}
		\\ [0.15in]
		& \geq & \underset{1 \leq j \leq k_f}{\max} \; f_j(x^2; \xi) - 2 \, \delta \, \mbox{Lip}_f(\xi) - \varepsilon^{\, \prime}
		\, \geq \,  \underset{1 \leq j \leq k_f}{\max} \; f_j(x^2; \xi) - 2 \, \varepsilon.
		\end{array}
		\]
		Hence $\bar{j} \in {\cal A}_{f;2\varepsilon}(x^2; \xi)$; thus ${\cal A}_{f;\varepsilon^{\, \prime}}(x^1; \xi) \subseteq {\cal A}_{f;2\varepsilon}(x^2; \xi)$.
		Similarly, we can establish the same inclusion for $g$.
	\end{proof}
	
	Since
	\[ \begin{array}{l}
	m(x^1 ; \xi) - m(x^2 ; \xi) \\ [0.1in]
	\epc = \, \left[ \, \displaystyle{
		\max_{1 \leq j \leq k_f}
	} \, f_j(x^1 ; \xi) - \displaystyle{
		\max_{1 \leq j \leq k_f}
	} \, f_j(x^2 ; \xi) \, \right] -  \left[ \, \displaystyle{
		\max_{1 \leq j \leq k_g}
	} \, g_j(x^1 ; \xi) - \displaystyle{
		\max_{1 \leq j \leq k_f}
	} \, g_j(x^2 ; \xi) \, \right],
	\end{array} \]
	the inequalities (\ref{eq:Lipschitz fg}) imply for all $x^1$ and $x^2$ in $X$ and almost all $\xi \in \Xi$,
	\begin{equation} \label{eq:Lipschitz m}
	| \, m(x^1 ; \xi) - m(x^2 ; \xi) \, | \leq \, \left( \, \mbox{Lip}_f(\xi) + \mbox{Lip}_g(\xi) \, \right) \, \| \, x^1 - x^2 \, \|_2.
	\end{equation}
	
	\subsection{Composite $\varepsilon$-strong d-stationarity} \label{subsec:composite stationarity}
	
	To facilitate the consistency analysis in the next section, we need to introduce a restriction of d-stationarity for the empirical problem
	(\ref{eq: saa}) known as
	$\varepsilon$-strong d-stationarity that corresponds to a given scalar $\varepsilon > 0$.
	The latter restricted concept of stationarity is more stable at the nondifferentiable points of the empirical risk objective.
	
	\gap
	
	Given convex functions $f$ and $\{g_j\}_{j=1}^k$ on $\mathbb{R}^n$ and a convex set $X\subseteq \mathbb{R}^n$,
	one can equivalently define $\bar{x}\in X$ to be a d-stationary point of the difference-of-convex programming
	\begin{equation}\label{opt:dc}
	\displaystyle\operatornamewithlimits{minimize}_{x\in X} \; \theta(x) \,\triangleq \, f(x) - \max_{1\leq j\leq k} g_j(x)
	\end{equation}
	if  for all $j$ satisfying $g_j(\bar{x}) =g(\bar{x})$,
	\[
	\theta(\bar{x}) \,\leq \, f(x) - \left[\, g_j(\bar{x}) + \nabla g_j(\bar{x})^\top (x-\bar{x})\,\right] + \frac{c}{2}\|x - \bar{x}\|^2, \quad \forall \; x\in X,
	\]
	for an constant $c \geq 0$;
	see, for example, \cite[Proposition 5]{PangRazaviyaynAlvarado16}.
	In a recent paper \cite{LuZhouSun18},  the authors introduce  a concept called $\varepsilon$-strong d-stationary solution,
	which pertains to a point $\bar{x}\in X$ satisfying the above inequality  for all $j$ such that $g_j(\bar{x}) \geq  g(\bar{x}) - \varepsilon$.
	Since our problem \eqref{eq:empirical composite saa} does not have the dc decomposition as in \eqref{opt:dc} due to the composition of
	a convex function $h(\,\bullet\,;\zbld)$ and a difference-of-convex function $m(\,\bullet\,;\xi)$, we are led to the extended
	$\varepsilon$-strong d-stationarity concept that is the subject of this subsection.
	
	\gap
	
	We start from the following lemma that allows us to characterize a d-stationary
	point of (\ref{eq:empirical composite saa}) as an optimal solution of a (nonconvex) optimization problem; see
	Lemma~\ref{lemma:d-stat}.
	
	\begin{lemma}\rm\label{lemma:cvx decomp} (\cite[Lemma 3]{CuiPangSen18})
		Any univariate convex function can be represented as the sum of a convex non-decreasing function and a  convex non-increasing function.
		Moreover, if the given function is Lipschitz continuous, then so are the two decomposed functions with the same Lipschitz constant.
		\hfill $\Box$
	\end{lemma}
	
	Applying the above lemma to the function $h(\,\bullet\,;\zbld)$, it follows that
	there exist a univariate convex non-decreasing function $h^\uparrow(\,\bullet\,;\zbld)$ and a univariate convex non-increasing
	function $h^\downarrow(\,\bullet\,;\zbld)$, both of which are easy to construct from $h(\,\bullet\,;\zbld)$,
	such that the convex loss function $h(\,\bullet\,;\zbld)$ in \eqref{eq: cvx composite dc} can be decomposed as
	\[
	h(t;\zbld) \, = \, h^\uparrow(t; \zbld) + h^\downarrow(t; \zbld),\quad \forall \ t \, \in \, \mathbb{R}.
	\]
	Moreover, if $h(\,\bullet\,;\zbld)$ is Lipschitz continuous (see Assumption~4.1(b)), then so are $h^{\updownarrow}(\,\bullet\,;\zbld)$
	with the same Lipschitz constant.
	Based on the above decomposition of the latter function, we introduce the following notation for any
	given $\bar{x}$, $x$, and $y$ in $\mathbb{R}^p$ and a nonnegative scalar $\varepsilon$:
	\[\left\{\begin{array}{ll}
	r_{\bar{x};\varepsilon}^{\uparrow}(y,x;\omega) \,\triangleq \, h^\uparrow \left (\, f(y;\xi) - \displaystyle\max_{j \in {\cal A}_{g;\varepsilon}(\bar{x};\xi)}
	\left[\, \underbrace{g_j(x;\xi)+(y-x)^{\top} \nabla_x g_j(x;\xi)}_{\mbox{\small linearization of $g_j$ at $x$ evaluated at $y$}}\,\right];\,\zbld \,\right),\\[0.35in]
	r_{\bar{x};\varepsilon}^{\downarrow}(y, x;\omega) \,\triangleq \, h^\downarrow \left(\, \displaystyle\max_{j \in {\cal A}_f(\bar{x};\xi)}
	\left[ \, \underbrace{f_j(x;\xi) +(y-x) ^{\top} \nabla_x f_j(x; \xi)}_{\mbox{\small linearization of $f_j$ at $x$ evaluated at $y$}}\,\right]- g(y;\xi);\,\zbld \, \right) .
	\end{array}\right.
	\]
	We further denote
	\begin{equation}\label{def:R}
	{R}_{\bar{x};\varepsilon}^{\,\updownarrow}(y,x) \,\triangleq\, \E_{\wt{\omega}} \left[\, r_{\bar{x};\varepsilon}^{\updownarrow}(y, x;\wt{\omega})\,\right]
	\epc \mbox{and}\epc
	{R}_{N;\bar{x};\varepsilon}^{\,\updownarrow}(y,x)\triangleq \displaystyle{
		\frac{1}{N}
	} \displaystyle{
		\sum_{n=1}^N
	} r_{\bar{x};\varepsilon}^{\updownarrow}(y, x;\omega^n),
	\end{equation}
	and their corresponding sum as
	\begin{equation}\label{def: R sum}
	R_{\bar{x};\varepsilon}(y,x) \, \triangleq \, R_{\bar{x};\varepsilon}^\uparrow(y,x) + R_{\bar{x};\varepsilon}^\downarrow(y,x), \epc
	{R}_{N;\bar{x};\varepsilon}(y,x) \, \triangleq \, {R}_{N;\bar{x};\varepsilon}^\uparrow(y,x) + R_{N;\bar{x};\varepsilon}^\downarrow(y,x),
	\end{equation}
	where we assume all the expectations are finite.
	When $\varepsilon = 0$, we will write $r_{\bar{x}}^{\updownarrow}(y, x;\omega)$, $R_{\bar{x}}^{\,\updownarrow}(y,x)$ and
	$R_{N;\bar{x}}^{\,\updownarrow}(y,x)$ for $r_{\bar{x};\varepsilon}^{\uparrow}(y, x;\omega)$, $R_{\bar{x};0}^{\,\updownarrow}(y,x)$
	and $R_{N;\bar{x};0}^{\,\updownarrow}(y,x)$, respectively.  Notice that $R_x(x,x) = {\cal M}(x)$ and $R_{N;x}(x,x) = {\cal M}_N(x)$ for all $x \in X$.
	Furthermore, for a piecewise affine $m(\bullet;\xi)$ given by (\ref{eq:dcp}) where each $f_j(\bullet;\xi)$ and $g_j(\bullet;\xi)$ are affine as
	in the piecewise affine regression problem, we have
	\[ \left\{ \begin{array}{lll}
	r_{\bar{x};\varepsilon}^{\uparrow}(y,x;\omega) & = & h^{\uparrow} \left(\, f(y;\xi) - \displaystyle{
		\max_{j \in {\cal A}_{g;\varepsilon}(\bar{x};\xi)}
	} \, g_j(y;\xi) \, \right) \\ [0.2in]
	r_{\bar{x};\varepsilon}^{\downarrow}(y,x;\omega) & = & h^{\downarrow} \left(\, \displaystyle{
		\max_{j \in {\cal A}_{f;\varepsilon}(\bar{x};\xi)}
	} \, f_j(y;\xi) - g(y;\xi) \, \right)
	\end{array} \right\} \epc \forall \, x,
	\]
	so that $R_y(y,x) = {\cal M}(y)$ and $R_{N;y}(y,x) = {\cal M}_N(y)$ for all $x$ and $y$ in $X$.  Some of the technical challenges mentioned
	before in the analysis of the problems (\ref{eq:population composite saa}) and (\ref{eq:empirical composite saa}) are embodied in the expect-value
	function  $R_{\bar{x};\varepsilon}(y,x)$ and its sampled approximation $R_{N;\bar{x};\varepsilon}(y,x)$, which are the main conduits employed
	in the analysis.  Namely, the index sets $A_{f/g;\varepsilon}(x;\xi)$
	are varying with the random realization $\xi$ that affects the pointwise maximum selection of the linearizations of $f_j$ and $g_j$; upon taking expectations of
	the random functionals $r_{\bar{x};\varepsilon}^{\updownarrow}(y,x;\omega)$, the behavior of $R_{\bar{x};\varepsilon}(y,x)$ is difficult to pinpoint, which
	relies on a good understanding of the variations of these random index sets; see Lemmas~\ref{lm:set inclusion} and \ref{lm:epsilon and zero}.
	
	\gap
	
	The following lemma provides a key characterization of a
	d-stationary point of problem \eqref{eq:empirical composite saa}.  Specifically, (\ref{eq:d-stationary characterization}) characterizes
	such a point as an optimal solution of a (nonconvex) minimization problem defined by the given point, which is equivalent to finitely many convex programs
	(\ref{eq:indexed d-stationarity}) as demonstrated in the proof.
	
	\begin{lemma}\label{lemma:d-stat}\rm
		The point $\bar{x} \in X$ is d-stationary for problem \eqref{eq:empirical composite saa} if and only if
		\begin{equation}\label{eq:d-stationary characterization}
		\bar{x} \, \in \, \displaystyle\operatornamewithlimits{argmin}_{x\in X}\;
		R_{N;\bar{x}}(x,\bar{x})
		\end{equation}
		Thus, for all $x \in X$,
		\begin{equation} \label{eq:d-stationary consequence}
		R_{N;\bar{x}}(x,\bar{x}) \, \geq \, R_{N;\bar{x}}(\bar{x},\bar{x}) \, = \, {\cal M}_N(\bar{x}).
		\end{equation}
	\end{lemma}
	
	\begin{proof}
		It is known from \cite[Lemma 5]{CuiPangSen18} that $\bar{x} \in X$ is d-stationary for problem \eqref{eq:empirical composite saa}
		if and only if $\bar{x}$ solves the problem
		\begin{equation} \label{eq:indexed d-stationarity}
		\begin{array}{ll}
		\displaystyle\operatornamewithlimits{minimize}_{x\in X} &
		\widehat{\mathcal{M}}_{N;J_1,J_2}(x,\bar{x}) \,\triangleq \\ [0.1in]	
		& \displaystyle\frac{1}{N}\displaystyle\sum_{n=1}^N \left[
		\begin{array}{l}
		h^\uparrow \left( \, f(x;\xi^{\,n}) - \left[ \, g_{j_{2,n}}(\bar{x};\xi^{\,n}) + \nabla_x g_{j_{2,n}}(\bar{x};\xi^{\,n})^{\top}(x-\bar{x}) \, \right];\zbld^{n}\, \right)
		+ \\ [0.1in]
		h^\downarrow\left( \, \left[ \, f_{j_{1,n}}(\bar{x};\xi^{\,n})  + \nabla_x f_{j_{1,n}}(\bar{x};\xi^{\,n})^{\top}(x-\bar{x}) \, \right] - g(x;\xi^{\,n});\zbld^{n}\,\right)\,
		\end{array} \right]
		\end{array}
		\end{equation}
		\[
		\begin{array}{lll}
		\mbox{for any } (J_1, J_2) & \in & \mathcal{A}(\bar{x};\xi^n) \,\triangleq \, \displaystyle{
			\prod_{n=1}^N
		} \, \mathcal{A}_{f}(\bar{x};\xi^{\,n}) \, \times \, \mathcal{A}_{g}(\bar{x};\xi^{\,n}) \\ [0.25in]
		& = & \left\{ \, (j_{1,n}, j_{2,n})_{n=1}^N \, \mid \,
		j_{1,n} \in \mathcal{A}_{f}(\bar{x};\xi^{\,n}), \ j_{2,n} \, \in \, \mathcal{A}_{g}(\bar{x};\xi^{\,n}) \epc \forall \, n \, \right\}.
		\end{array}
		\]
		Therefore, if the condition \eqref{eq:d-stationary characterization} holds, then for any $x\in X$ and any pair $(J_1,J_2)$ satisfying the above inclusion,
		\[
		\widehat{\mathcal{M}}_{N;J_1,J_2}(\bar{x},\bar{x}) = R_{N;\bar{x}}^{\,\uparrow}(\bar{x},\bar{x}) +
		R_{N;\bar{x}}^{\,\downarrow}(\bar{x},\bar{x}) \, \leq \, R_{N;\bar{x}}^{\,\uparrow}(x,\bar{x}) + R_{N;\bar{x}}^{\,\downarrow}(x,\bar{x})
		\, \leq \, \widehat{\mathcal{M}}_{N;J_1,J_2}(x,\bar{x}),
		\]
		showing that $\bar{x}$ is a d-stationary point for  problem \eqref{eq:empirical composite saa}.
		Conversely, if $\bar{x}$ is a d-stationary point, then for all $(J_1,J_2) \in  \mathcal{A}(\bar{x};\xi^n)$,
		\[
		R_{N;\bar{x}}^{\,\uparrow}(\bar{x},\bar{x}) +
		R_{N;\bar{x}}^{\,\downarrow}(\bar{x},\bar{x}) = \widehat{\mathcal{M}}_{N;J_1,J_2}(\bar{x},\bar{x}) \leq \widehat{\mathcal{M}}_{N;J_1,J_2}(x,\bar{x}),
		\quad \forall\; x\in X,
		\]
		which yields
		\[
		R_{N;\bar{x}}^{\,\uparrow}(\bar{x},\bar{x}) +
		R_{N;\bar{x}}^{\,\downarrow}(\bar{x},\bar{x}) \leq  \displaystyle\min_{(J_1, J_2)  \in  \mathcal{A}(\bar{x};\xi^n)}\widehat{\mathcal{M}}_{N;J_1,J_2}(x,\bar{x})
		= \, R_{N;\bar{x}}^{\,\uparrow}(x,\bar{x}) + R_{N;\bar{x}}^{\,\downarrow}(x,\bar{x}), \epc \forall\; x\in X.
		\]
		This completes the proof of this lemma.
	\end{proof}
	
	Notice that each minimization problem (\ref{eq:indexed d-stationarity}) is a convex program in $x$, confirming that d-stationarity of (\ref{eq:empirical composite saa})
	can be characterized by finitely many convex programs.   This is in contrast to d-stationarity of the population problem (\ref{eq:population composite saa}) which does not
	seem to have a convex programming characterization.  The discussion here extends to the
	minimization problems in the following definition of composite $\varepsilon$-strong d-stationary points that is motivated by the above lemma.
	
	\begin{definition} \label{def:epsilon-stationarity} \rm
		Let $\varepsilon>0$ be a given scalar.  The point $\bar{x}\in \mathbb{R}^p$ is called a {\sl composite $\varepsilon$-strong d-stationary point} of
		problem \eqref{eq:empirical composite saa} if
		\[
		\bar{x} \, \in \, \displaystyle\operatornamewithlimits{argmin}_{x\in X}\;
		R_{N;\bar{x};\varepsilon}(x,\bar{x}).
		\]
	\end{definition}
	
	\begin{remark}\label{remark:strong}\rm
		We remark that the above definition	of the composite $\varepsilon$-strong d-stationarity at $\bar{x}$ is equivalent to
		\begin{equation}\label{defn: new strong d}
		\mathcal{M}_N(\bar{x}) \, \leq \, R_{N;\bar{x};\varepsilon}(x,\bar{x}),
		\end{equation}
		which reduces to (\ref{eq:d-stationary consequence}) when $\varepsilon = 0$.  This is because
		\[\begin{array}{ll}
		R_{N;\bar{x};\varepsilon}(\bar{x},\bar{x}) \, = \, R_{N;\bar{x};\varepsilon}^{\,\uparrow}(\bar{x},\bar{x}) +
		R_{N;\bar{x};\varepsilon}^{\,\downarrow}(\bar{x},\bar{x}) \\[0.15in]
		= \, \displaystyle\frac{1}{N}\sum_{n=1}^N \left[\, h^\uparrow \left( \, f(\bar{x};\xi^n) -
		\, \displaystyle\max_{j \in {\cal A}_{g;\varepsilon}(\bar{x}; \,\xi)} \,g_j(\bar{x};\xi^n);\,\zbld_n \, \right) + h^\downarrow \left(\,
		\displaystyle\max_{j \in {\cal A}_{f;\varepsilon}(\bar{x}; \,\xi^n)}  f_j(\bar{x};\xi^n) - g(\bar{x}, \xi^n);\,\zbld_n  \right)\,\right]\\[0.15in]
		= \, \displaystyle\frac{1}{N}\sum_{n=1}^N \left[\, h^\uparrow \left( \, f(\bar{x};\xi^n) -
		\, \displaystyle\max_{j \in {\cal A}_{g}(\bar{x}; \,\xi)} \,g_j(\bar{x};\xi^n);\,\zbld_n \, \right) + h^\downarrow \left(\,
		\displaystyle\max_{j \in {\cal A}_{f}(\bar{x}; \,\xi^n)}  f_j(\bar{x};\xi^n) - g(\bar{x}, \xi^n);\,\zbld_n  \right)\,\right] \\[0.15in]
		= \, \displaystyle\frac{1}{N}\sum_{n=1}^N \left[\, h^\uparrow \left( \, f(\bar{x};\xi^n) -
		\, g(\bar{x};\xi^n);\,\zbld_n \, \right) + h^\downarrow \left(\,
		f(\bar{x};\xi^n) - g(\bar{x}, \xi^n);\,\zbld_n  \right)\,\right] = \mathcal{M}_N(\bar{x}).
		\end{array}
		\]
		Following similar notation and arguments as in the proof of Lemma \ref{lemma:d-stat} which pertains to $\varepsilon = 0$,
		we can alternatively write \eqref{defn: new strong d} as
		\begin{equation*}
		\calM_N(\bar{x}) \leq \widehat{\mathcal{M}}_{N;J_1,J_2}(x,\bar{x}), 
		\quad \forall \; x\in X, \epc \forall\;(J_1, J_2) \, \in \, \displaystyle{
			\prod_{n=1}^N
		} \, \mathcal{A}_{f;\varepsilon}(\bar{x};\xi^{\,n}) \, \times \, \mathcal{A}_{g;\varepsilon}(\bar{x};\xi^{\,n}),
		%
		\end{equation*}
		the latter being the definition in \cite{LuZhouSun18} of an $\varepsilon$-strong d-stationarity for the program \eqref{opt:dc}.
		Therefore, our definition of composite $\varepsilon$-strong d-stationarity for the composite difference-max program \eqref{eq:empirical composite saa}
		is a generalization of $\varepsilon$-strong d-stationarity for a structured difference-of-convex program introduced in the cited reference.
		\hfill $\Box$
	\end{remark}
	
	Comparing Lemma \ref{lemma:d-stat} and Definition \ref{def:epsilon-stationarity},
	one can obviously see that the composite $\varepsilon$-strong d-stationarity implies the d-stationarity of that point since the former concept
	needs to satisfy additional conditions given by the indices in the  $\varepsilon$-argmax set.  In fact, the latter property is a necessary condition for the local
	optimality of the vector $\bar{x}$, while the former is necessary only for the global optimality of $\bar{x}$.  Further
	connections of a composite $\varepsilon$-strong d-stationary solution and a d-stationary solution are presented in Proposition~\ref{pr:epsilon and zero} .  First we establish a lemma that allows us prove one such connection.
	
	\begin{lemma} \label{lm:epsilon and zero} \rm
		For every pair $(\bar{x},\xi) \in X \times \Omega$, 
		a scalar $\bar{\varepsilon} > 0$ 
		exists such that for all $\varepsilon \in [ \, 0, \, \bar{\varepsilon} \, ]$, we have
		$\mathcal{A}_{f;\varepsilon}(\bar{x};\xi) = \mathcal{A}_{f}(\bar{x};\xi)$ and
		$\mathcal{A}_{g;\varepsilon}(\bar{x};\xi) = \mathcal{A}_{g}(\bar{x};\xi)$. 
	\end{lemma}
	
	\noindent {\bf Proof.}  We may assume without loss of generality that 
	neither elements of $\{ f_j(\bar{x};\xi) \}_{j=1}^{k_f}$ are all equal nor the same for $\{ g_j(\bar{x};\xi) \}_{j=1}^{k_f}$.
	Arrange the elements in the families
	$\{ f_j(\bar{x};\xi) \}_{j=1}^{k_f}$ and $\{ g_j(\bar{x};\xi) \}_{j=1}^{k_g}$ in a non-increasing order as follows:
	\[ \begin{array}{l}
	f_{[1]}(\bar{x};\xi) \, = \, \cdots \, = \, f_{[s_f]}(\bar{x};\xi) \, > \, f_{[s_f+1]}(\bar{x};\xi)
	\, \geq \, \cdots \, \geq \, f_{[k_f]}(\bar{x}^N,\xi) \\ [0.1in]
	g_{[1]}(\bar{x};\xi) \, = \, \cdots \, = \, g_{[s_g]}(\bar{x};\xi) \, > \, g_{[s_g+1]}(\bar{x};\xi)
	\, \geq \, \cdots \, \geq \, g_{[k_g]}(\bar{x};\xi) ,
	\end{array} \]
	where the integer $s_f \in \{ 1, \cdots, k_f-1 \}$ and similarly for the integer $s_g$.  Let
	\begin{equation} \label{eq:key epsilon}
	\bar{\varepsilon} \, \triangleq \, \thalf \,
	\left\{ \, f_{[1]}(\bar{x};\xi) - f_{[s_f+1]}(\bar{x};\xi), \, g_{[1]}(\bar{x};\xi) - g_{[s_g+1]}(\bar{x};\xi) \, \right\}.
	\end{equation}
	Let $\varepsilon \in [ \, 0, \, \bar{\varepsilon} \, ]$ and $j \in \mathcal{A}_{f;\varepsilon}(\bar{x};\xi)$.  Suppose
	$f_j(\bar{x};\xi) < f_{[1]}(\bar{x};\xi)$.  Then we must have $f_j(\bar{x};\xi) \leq f_{[s_f+1]}(\bar{x};\xi)$.
	Hence,
	\[ \begin{array}{lll}
	f_{[s_f+1]}(\bar{x};\xi) & \geq &
	f_j(\bar{x};\xi) \, \geq \, f_{[1]}(\bar{x};\xi) - \varepsilon \\ [5pt]
	& \geq & f_{[1]}(\bar{x};\xi) - \thalf \, \left( \, f_{[1]}(\bar{x};\xi) - f_{[s_f+1]}(\bar{x};\xi) \, \right),
	\end{array}
	\]
	which yields $f_{[s_f+1]}(\bar{x};\xi) \geq f_{[1]}(\bar{x};\xi)$.  This is a contradiction.  Thus
	$\mathcal{A}_{f;\varepsilon}(\bar{x};\xi) = \mathcal{A}_f(\bar{x};\xi)$.  Similarly, we can prove
	$\mathcal{A}_{g;\varepsilon}(\bar{x};\xi) = \mathcal{A}_g(\bar{x};\xi)$.
	\hfill $\Box$
	
	\gap
	
	An easy application of the above lemma immediately yields the following result.
	
	\begin{proposition} \label{pr:epsilon and zero} \rm
		For every positive integer $N$, if $\bar{x}^N$ is a d-stationary point of problem \eqref{eq:empirical composite saa}
		corresponding to a given family of ralizations $\{ \, \xi^{\, n} \, \}_{n=1}^N \subset \Xi$, then a scalar
		$\bar{\varepsilon}_N$ exists such that $\bar{x}^N$ is a composite $\varepsilon$-strong d-stationary point of the same problem
		for any $\varepsilon \in [ \, 0, \, \bar{\varepsilon}_N \, ]$.  \hfill $\Box$
	\end{proposition}
	
	When $m(\bullet;\xi)$ is piecewise affine,
	the equivalence of composite $\varepsilon$-strong d-stationarity and d-stationarity for small $\varepsilon > 0$ can be
	augmented by a locally minimizing property.   Indeed in this case, by results in \cite{CuiPangHongChang18}, we know that a d-stationary
	point must be locally minimizing; thus the equivalence between d-stationarity, composite $\varepsilon$-strong d-stationary, and locally minimizing.
	The diagram below illustrates these relationships for the problem \eqref{eq:empirical composite saa}.
	\begin{figure}[H]
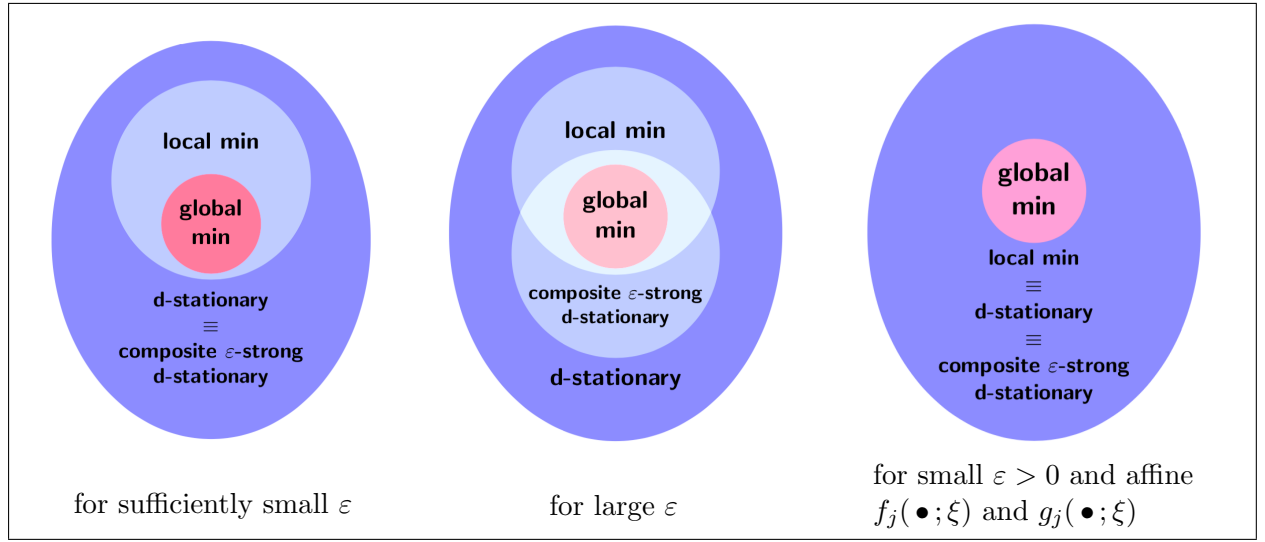

		\begin{center}
			\fbox{
				\begin{minipage}{.30\textwidth}
					\begin{center}
						\includegraphics[width = 0.9\textwidth]{diagram1.pdf}
						\vskip 0.15in
						for sufficiently small $\varepsilon$
					\end{center}
				\end{minipage}
				\begin{minipage}{0.33\textwidth}
					\begin{center}
						\includegraphics[width = 0.9\textwidth]{diagram2.pdf}
						\vskip 0.15in
						for large $\varepsilon$
					\end{center}
				\end{minipage}
				\begin{minipage}{0.33\textwidth}
					\begin{center}
						\includegraphics[width = 0.9\textwidth]{diagram3.pdf}
						
						\begin{tabular}{l}
							for small $\varepsilon > 0$ and affine \\
							$f_j(\,\bullet\,;\xi)$ and $g_j(\,\bullet\,;\xi)$
						\end{tabular}
					\end{center}
				\end{minipage}
			}
		\end{center}
		\caption{\footnotesize Diagram of the relationship between global/local minimizers and (composite $\varepsilon$-strong) d-stationary points.}
	\end{figure}
	
	To close this section, we point out that the computation of a d-stationary point of a  difference-max optimization problem
	can be accomplished by an enhancement \cite{PangRazaviyaynAlvarado16} of the original
	difference-of-convex algorithm (DCA) \cite{LeThiPham05} that makes use of an arbitrary $\varepsilon > 0$.
	The subsequent reference \cite{LuZhouSun18} shows that the so-computed d-stationary solution is actually $\varepsilon$-strong d-stationary.
	The more recent reference \cite{CuiPangSen18} further extends these references to a composite difference-max problem of which
	(\ref{eq:empirical composite saa}) is a special case.
	Thus the analysis in the next section about a d-stationary solution of (\ref{eq:empirical composite saa}) is computationally meaningful.
	This is in contrast to the analysis of minimizers of the problems (\ref{eq:empirical composite saa}) and (\ref{eq:population composite saa})
	that is in general detached from computational tractability.
	
	\section{Consistency of D-stationary Solutions} \label{sec: consistency}
	
	We establish in this section the convergence as $N$ tends to infinity
	of  composite $\varepsilon$-strong d-stationary solutions of (\ref{eq:empirical composite saa}) to a d-stationary solution of the population
	problem (\ref{eq:population composite saa}).  Adding to the uniform Lipschitz continuity
	(\ref{eq:Lipschitz fg}) of the functions $\left\{ f_j \right\}_{j=1}^{k_f}$ and $\left\{ g_j \right\}_{j=1}^{k_g}$,
	we impose the following assumptions.
	
	
	\begin{assumption}\label{ass for ULLN}\rm
		(a1) Both $\mbox{Lip}_f(\xi)$ and $\mbox{Lip}_g(\xi)$ in the inequalities (\ref{eq:Lipschitz fg}) are square integrable and
		$c_0 > 0$ exists such that for all $\xi$ in the probability-one subset $\Xi^{\, 1}$ of $\Xi$,
		$\max\left( \, \mbox{Lip}_f(\xi), \, \mbox{Lip}_g(\xi) \, \right) \leq c_0$.
		
		\gap
		
		(a2) There exist square integrable functions $\mbox{Lip}_{\nabla f}(\xi)$ and $\mbox{Lip}_{\nabla g}(\xi)$ and a probability-one
		subset $\Xi^{\, 1}_{\nabla}$ of $\Xi$ such that for $\xi\in \Xi^{\, 1}_{\nabla}$ and for any $x$ and $y$ in $X$,
		$$
		\left\{\begin{array}{ll}
		\left\|\, \nabla_x f_j(x;\xi) - \nabla_x f_j(y;\xi) \,\right\|_2 \, \leq \, \mbox{Lip}_{\nabla f}(\xi) \,
		\left\|\, x - y \, \right\|_2, \quad \forall\; j = 1, \ldots, k_f,\\[0.15in]
		\left\|\, \nabla_x g_j(x;\xi) - \nabla_x g_j(y;\xi) \, \right\|_2 \, \leq \, \mbox{Lip}_{\nabla g}(\xi) \,
		\left\|\, x - y \, \right\|_2, \quad \forall\; j = 1, \ldots, k_g.
		\end{array}\right.
		$$
		(a3) There exist square integrable functions $C_f(\xi)$ and $C_g(\xi)$ and a probability-one subset $\Xi^{\, 2}_{\nabla}$
		of $\Xi$ such that for all $\xi \in \Xi^{\, 2}_{\nabla}$,
		\[
		\left\{\begin{array}{ll}
		\displaystyle{
			\sup_{x \in X}
		} \, \left\| \, \nabla_x f_j(x;\xi) \,\right\|_2 \, \leq \, C_f(\xi), \quad \forall\; j = 1, \ldots, k_f, \\[0.15in]
		\displaystyle{
			\sup_{x \in X}
		} \, \left\| \, \nabla_x g_j(x;\xi) \, \right\|_2 \, \leq \, C_g(\xi),  \quad \forall\; j = 1, \ldots, k_g.
		\end{array}\right.
		\]
		(b) There exist a square integrable function $\mbox{Lip}_{h}(\zbld)$ and a probability-one subset $\wh{\cal Z}$ of ${\cal Z}$
		such that for all $\zbld \in \wh{\cal Z}$ and for any $t_1$ and $t_2\in \mathbb{R}$,
		\[
		\left|\, h(t_1;\zbld) - h(t_2;\zbld) \,\right| \, \leq \, \mbox{Lip}_h(\zbld) \left|\, t_1 - t_2 \,\right|.
		\]
		We let $\wh{\Xi} \, \triangleq \, \Xi^{\, 1} \, \cap \, \Xi^{\, 1}_{\nabla} \, \cap \, \Xi^{\, 2}_{\nabla}$
		and $\wh{\Omega} \, \triangleq \, \wh{\Xi} \times \wh{\cal Z}$.  Note that $\IP( \wh{\Omega} ) = 1$.
	\end{assumption}
	
	Notice that
	Assumptions~(a2) and (a3) in \ref{ass for ULLN} imply that
	\[
	\E_{\wt{\omega}}\left[ \, \displaystyle{
		\sup_{x \in X}
	} \, \left| \, h^{\updownarrow} \circ ( \, m(x;\wt{\xi});\wt{\zbld} \, ) \, \right| \, \right] \, < \, +\infty.
	\]
	We begin with several lemmas that are essential to the proof of our main result.
	The first one is the classical uniform law of large numbers and its implication on the continuous convergence of random functions.
	
	\begin{lemma}\rm (c.f. \cite[Lemma 3.10]{van2000empirical})\label{lemma: ULLN}
		Let the bivariate function $\mathcal{L} : \mathbb{R}^p \times \Omega \to \mathbb{R}$ be
		such that ${\cal L}(\bullet,\omega)$ is continuous on $\mathbb{R}^p$ for almost all $\omega\in \Omega$.
		Let $X\subseteq\mathbb{R}^p$ be a compact set.  Suppose that
		$\E_{\wt{\omega}} \left[\,\displaystyle\sup_{x \in X}\,|\,\mathcal{L}(x;\wt{\omega})\,|\,\right] < +\infty$.
		Then
		\[
		\displaystyle\lim_{N \rightarrow \infty}\, \sup_{x \in X}\,\left|\,\frac{1}{N}\sum_{n=1}^N\mathcal{L}(x; \omega^{\, n}) -
		\E_\omega \left[\,\mathcal{L}(x; \omega)\,\right]\,\right| \rightarrow 0 \quad \mbox{almost surely}\,.
		\]		
		Moreover, if $\E_{\wt{\omega}} \, {\cal L}(\bullet,\wt{\omega})$
		is continuous on an open set containing $X$, then for any $x \in X$ and any sequence $\left\{x^{N}\right\}\subset X$ converging to $x$,
		it holds that
		\[
		\displaystyle\lim_{N \rightarrow \infty}\,\left|\,\frac{1}{N}\sum_{n = 1}^N\mathcal{L}\left(x^{N} ; \omega^{\, n} \right) -
		\E_{\wt{\omega}} \, [\,\mathcal{L}(x;\wt{\omega})\,]\,\right| = 0 \epc \mbox{almost surely} \, .
		\]
	\end{lemma}
	
	%
	
	
	\begin{lemma}\label{lemma: ULLN 1}\rm
		Suppose that Assumption \ref{ass for ULLN} holds.  Let $X \subseteq \mathbb{R}^{\, p}$ be a compact set. Then for any $\bar{x}\in \mathbb{R}^{\, p}$
		and any $\varepsilon > 0$,
		\begin{equation} \label{eq:uniform LLN L}
		\displaystyle{
			\lim_{N \to +\infty}
		} \, \max\left( \, \begin{array}{l}
		\displaystyle{
			\sup_{x, y \in X} 
		} \, \displaystyle{
		} \, \left| \, R^{\,\uparrow}_{N;\bar{x};\varepsilon}(x,y) - R^{\,\uparrow}_{\bar{x};\varepsilon}(x,y)\, \right|, \\ [0.15in]
		\displaystyle{
			\sup_{x, y \in X} 
		} \, \displaystyle{
		} \, \left| \, R^{\,\downarrow}_{N;\bar{x};\varepsilon^{\prime}}(x,y) - R^{\,\downarrow}_{\bar{x};\varepsilon}(x,y)\, \right|
		\end{array} \, \right) \, = \, 0, \epc \mbox{almost surely}.
		\end{equation}
	\end{lemma}
	
	\begin{proof}
		To prove this lemma, it suffices to check that
		\begin{equation}\label{lemma: ULLN eq 1}
		\left\{\begin{array}{ll}
		\E_{\wt{\omega}}\left[ \, \displaystyle{
			\sup_{x, y \in X}
		} \,
		h^\uparrow \left( \, f(y;\wt{\xi}) - \displaystyle{
			\max_{j \in {\cal A}_{g;\varepsilon}(\bar{x};\wt{\xi})}
		} \, \left[\, g_j(x;\wt{\xi}) + (y-x)^{\top} \nabla_x g_j(x;\wt{\xi})\,\right];\wt{\zbld} \,\right) \, \right]
		< +\infty, \\[0.25in]
		\E_{\wt{\omega}}\left[ \, \displaystyle{
			\sup_{x, y \in X}
		} \,
		h^\downarrow \left( \, \displaystyle{
			\max_{j \in {\cal A}_{f;\varepsilon}(\bar{x};\wt{\xi})}
		} \, \left[ \, f_j(x;\wt{\xi}) + (y-x)^{\top} \nabla_x f_j(x;\wt{\xi})\,\right]- g(y;\wt{\xi});\wt{\zbld} \,\right)\right]
		\, < \, +\infty
		\end{array}\right.
		\end{equation}
		and then apply Lemma~\ref{lemma: ULLN}.  By Assumption \ref{ass for ULLN}~(a2) and (a3), we have that for all pairs $( \xi,\zbld ) \in \wh{\Omega}$,
		\[
		\begin{array}{l}
		h^\uparrow \left(\, f(y;\xi) -  \displaystyle\max_{j \in {\cal A}_{g;\varepsilon}(\bar{x}; \,\xi)} \left[\, g_j(x;\xi) + (y-x)^{\top} \nabla_x g_j(x; \xi)\,\right];\,\zbld \,\right)
		\\[0.2in]
		\leq \, h^\uparrow \left(\, f(\bar{x};\xi) + \, \mbox{Lip}_f(\xi) \, \| \, y-\bar{x} \, \|_2 -\, \displaystyle{
			\max_{j \in {\cal A}_{g;\varepsilon}(\bar{x}; \,\xi)}
		} \, g_j(y; \xi) + \frac{\mbox{Lip}_{\nabla g}(\xi)\| \, y-x \, \|_2^2}{2};\,\zbld \,\right) \\[0.2in]
		\leq \, h^\uparrow \left(\, f(\bar{x};\xi) + \, \mbox{Lip}_f(\xi) \, \| \, y-\bar{x} \, \|_2 -\, \displaystyle{
			\max_{j \in {\cal A}_{g;\varepsilon}(\bar{x}; \,\xi)}
		} \, g_j(\bar{x}; \xi) +\mbox{Lip}_g(\xi)\| \, y-\bar{x} \, \|_2+ \frac{\mbox{Lip}_{\nabla g}(\xi)\| \, y-x\, \|_2^2}{2};\,\zbld \,\right) \\[0.2in]
		\leq \,  h^\uparrow \left(\, f(\bar{x};\xi) - g(\bar{x};\xi);\zbld \, \right) + \, \mbox{Lip}_h(\zbld) \, \left[ \, \mbox{Lip}_f(\xi) \, \| \, y-\bar{x} \, \|_2 +
		\mbox{Lip}_g(\xi)\|\, y-\bar{x} \, \|_2 +  \displaystyle\frac{\mbox{Lip}_{\nabla g}(\xi)\|y-x\|_2^2}{2}\, \right] \\ [0.2in]
		= \,  h^\uparrow \left(\, m(\bar{x} ; \xi);\zbld \, \right) +  \,  \, \mbox{Lip}_h(\zbld) \, \left[ \, \mbox{Lip}_f(\xi) \, \| \, y-\bar{x} \, \|_2 +
		\mbox{Lip}_g(\xi)\|\,y-\bar{x} \, \|_2 +  \displaystyle\frac{\mbox{Lip}_{\nabla g}(\xi)\| \, y-x \, \|_2^2}{2} \, \right].
		\end{array}
		\]
		By Assumption \ref{ass for ULLN} (b) and setting $B = \text{Diam}(X)$, we further obtain that
		\[
		\begin{array}{l}
		\E_{\wt{\omega}}\left[\, \displaystyle\sup_{x,y\in X}
		h^\uparrow \left(\, f(y;\wt{\xi}) -  \displaystyle{
			\max_{j \in {\cal A}_{g;\varepsilon}(\bar{x};\wt{\xi})}
		} \, \left[ \, g_j(x;\wt{\xi})+(y-x)^{\top} \nabla_x g_j(x;\wt{\xi})\,\right];\,\wt{\zbld} \,\right) \, \right]
		\\ [0.2in]
		\leq \,
		\E_{\wt{\omega}}\left[\, \displaystyle\sup_{x\in X}h^\uparrow \left(\, m(\bar{x};\wt{\xi});\wt{\zbld} \,\right) \, \right]  +
		B \, \E_{\wt{\omega}} \left[\, \mbox{Lip}_h(\wt{\zbld}) \,\left( \mbox{Lip}_f(\wt{\xi})\, + \mbox{Lip}_g(\wt{\xi}) \, +
		\displaystyle{
			\frac{B \, \mbox{Lip}_{\nabla g}(\wt{\xi})}{2}
		} \, \right) \, \right] \\ [0.2in]
		< \, + \infty.
		\end{array}
		\]
		This string of inequalities is enough to yield
		the first inequality in \eqref{lemma: ULLN eq 1}. The second inequality in \eqref{lemma: ULLN eq 1} can be derived based on similar arguments
		and we omit the details here.
	\end{proof}
	
	From this point on, we will be working with infinite sequences $\{ \omega^n \}_{n=1}^{\infty}$ of random realizations of the random variable
	pairs $( \wt{\xi},\wt{\zbld} )$.  For this purpose, we let $\Omega^{\infty}$ denote the $\infty$-fold Cartesian product of the sample space
	$\Omega$.  Let $\mathcal{F}^{\infty}$ denote the sigma-algebra generated by subsets of $\Omega^{\infty}$, and let $\IP_{\infty}$
	be the corresponding probability measure defined on this sigma-algebra.  Let $\E_{\infty}$ be the expectation operator induced
	by $\IP_{\infty}$.  Throughout the analysis, we fix the probability tuple
	$\left( \Omega^\infty, \calF^\infty, \IP_\infty, \E_{\infty} \right)$.
	We say that an event $E \in \calF^\infty$ happens ``almost surely" if $\IP_\infty(E) = 1$.  Without loss of generality, we assume
	that the probability-one set $\wh{\Omega} \triangleq \wh{\Xi} \times {\cal Z}^{\, 1}$ is such that the
	limit (\ref{eq:uniform LLN L})
	in Lemma~\ref{lemma: ULLN 1} holds for all families $\{ \omega^n \}_{n=1}^{\infty} \subset \wh{\Omega}^{\, \infty}$.
	In the rest of the paper, for any such family of random realizations, we let, for each $N$, $x^{N;\varepsilon}(\omega^N)$ be a composite
	$\varepsilon$-strong d-stationary point of (\ref{eq:empirical composite saa}) corresponding to a given scalar $\varepsilon \geq 0$.
	(The case $\varepsilon = 0$ refers to a d-stationary point.)
	We will write $x^N$ for $x^{N;\varepsilon}(\omega^N)$ if the context is clear.
	
	\gap
	The following lemma is the key step to establish our main result of
	this section.
	
	\begin{lemma}\label{le: strong d} \rm
		Suppose that Assumption~\ref{ass for ULLN} holds.  Let $\varepsilon > 0$
		be given and let $\{ \omega^n \}_{n=1}^{\infty} \subset \wh{\Omega}$ be arbitrary.
		If the sequence
		$\{x^{N;\varepsilon}(\omega^N)\}$ converges to 
		$x^\infty$, then $x^\infty$ solves the nonconvex optimization problem
		\[
		\displaystyle\operatornamewithlimits{minimize}_{x\in X}\;
		R_{x^\infty;\varepsilon^{\prime}}(x,x^\infty)
		\]
		for every $\varepsilon^{\, \prime} \in [ 0,\varepsilon )$.
		In particular, $x^\infty$ is also a minimizer
		of $R_{x^\infty}(\bullet,x^\infty)$ on $X$.
	\end{lemma}
	
	\begin{proof}  Write $x^N \equiv x^{N;\varepsilon}(\omega^N)$ for simplicity.
		Since $x^N$ converges to 
		$x^\infty$, then for sufficiently large $N$, the following inclusions hold
		for all $\varepsilon^{\, \prime} \in [ 0,\varepsilon )$ and all $\xi \in \wh{\Xi}$,
		\[
		{\cal A}_{f; \varepsilon^{\, \prime}}(x^{\infty};\xi)
		\, \subseteq \, {\cal A}_{f;\varepsilon}(x^{N};\xi)
		\epc \mbox{and} \epc
		{\cal A}_{g; \varepsilon^{\, \prime}}(x^{\infty};\xi)
		\, \subseteq \, {\cal A}_{g;\varepsilon}(x^{N};\xi),
		\]
		by Lemma \ref{lm:set inclusion}.  Furthermore, since $x^N$ is a composite
		$\varepsilon$-strong d-stationary point of (\ref{eq:empirical composite saa}), it follows from
		\eqref{defn: new strong d} that for any $x \in X$,
		\begin{equation*}
		\begin{array}{l}
		\calM_N(x^N) \, \leq \, \displaystyle\frac{1}{N}\sum_{n = 1}^N h^\uparrow\left(f(x;\xi^{\,n}) - \displaystyle{
			\max_{j \in {\cal A}_{g;\varepsilon}(x^{N};\, \xi^{\,n})}
		} \, \left[\, g_j(x^N; \xi^{\,n}) +\nabla_x g_j\left(x^{N}; \xi^{\,n}\right)^{\top} (x - x^{N})\,\right]; \zbld^{\, n}\right) \\[0.25in]
		\hspace{0.7in} + \displaystyle{
			\frac{1}{N}
		} \, \displaystyle{
			\sum_{n = 1}^N
		} \,h^\downarrow \left( \displaystyle{
			\max_{j \in {\cal A}_{f;\varepsilon}(x^{N};\, \xi^{\,n})}
		} \, \left[\, f_j(x^N;\xi^{\,n})+\nabla_x f_j\left(x^{N}; \xi^{\,n}\right)^{\top} (x - x^{N})\,\right] - g(x; \xi^{\,n}) ; \zbld^{\, n}\right) \\[0.25in]
		\leq \, R_{N;x^\infty; \varepsilon^{\prime}}(x;x^N) \\ [0.15in]
		= \, \left[ \, R_{N;x^\infty; \varepsilon^{\prime}}(x,x^N) - R_{x^\infty; \varepsilon^{\prime}}(x,x^N) \, \right] +
		\left[ \, R_{x^\infty; \varepsilon^{\prime}}(x,x^N) - R_{x^\infty; \varepsilon^{\prime}}(x,x^{\infty}) \, \right] + R_{x^\infty; \varepsilon^{\prime}}(x,x^{\infty}) \\ [0.15in]
		\leq \, \left[ \, \displaystyle{
			\sup_{x^{\, \prime}, y \, \in \, X}
		} \, \left| \, R_{N;x^\infty; \varepsilon^{\prime}}(x^{\, \prime},y) - R_{x^\infty; \varepsilon^{\prime}}(x^{\, \prime},y) \,\right| \, \right] +
		\left[ \, \displaystyle{
			\sup_{x^{\, \prime} \in X}
		} \,  \left|\, R_{x^\infty; \varepsilon^{\prime}}(x^{\, \prime},x^N) - R_{x^\infty; \varepsilon^{\prime}}(x^{\, \prime},x^\infty) \,\right| \, \right]
		\\ [0.25in] \epc
		\epc + \, R_{x^\infty; \varepsilon^{\prime}}(x,x^\infty).
		\end{array}
		\end{equation*}
		Observe that
		\begin{equation*}
		\begin{array}{l}
		\displaystyle\sup_{x\in X} \left| \, R_{x^\infty;\varepsilon^{\prime}}^{\uparrow}(x,x^N) - R_{x^\infty;\varepsilon^{\prime}}^{\uparrow}(x,x^\infty) \,\right| \\[0.2in]
		\leq \, \E_{\wt{\omega}}\left[ \, \mbox{Lip}_h(\wt{\zbld}) \left| \, g(x^\infty;\wt{\xi}) - g(x^N;\wt{\xi}) + \displaystyle{
			\max_{j \in {\cal A}_{g;\varepsilon^{\prime}}(x^\infty;\wt{\xi})}
		} \, \left[ \, \nabla_x g_j(x^\infty;\wt{\xi})^\top \, x^\infty -  \nabla_x g_j(x^N;\wt{\xi})^\top\, x^N \,\right] \, \right| \, \right] \\[0.2in]
		\epc + \, \displaystyle{
			\sup_{x\in X}
		} \, \left[ \, \E_{\wt{\omega}} \, \mbox{Lip}_h(\wt{\zbld}) \, \|\, x \,\|_2 \, \max_{j \in {\cal A}_{g;\varepsilon^{\prime}}(x^\infty;\wt{\xi})} \,
		\left\|\, \nabla_x g_j(x^N;\wt{\xi}) - \nabla_x g_j(x^\infty;\wt{\xi}) \,\right\|_2\right]\\[0.2in]
		\leq \, \E_{\wt{\omega}}\left[\mbox{Lip}_h(\wt{\zbld}) \left|\, g(x^\infty;\wt{\xi}) - g(x^N;\wt{\xi}) + \displaystyle{
			\max_{j \in {\cal A}_{g;\varepsilon^{\prime}}(x^\infty; \,\xi)}
		} \, \left[ \, \nabla_x g_j(x^\infty;\wt{\xi})^\top \, x^\infty -  \nabla_x g_j(x^N;\wt{\xi})^\top\, x^N \,\right] \,\right| \, \right]\\[0.2in]
		\epc + \, \text{Diam}(X) \, \left[ \, \E_{\wt{\omega}} \, \mbox{Lip}_h(\wt{\zbld}) \, \displaystyle{
			\max_{j \in {\cal A}_{g;\varepsilon^{\prime}}(x^\infty; \,\xi)}
		} \, \left\|\, \nabla_x g_j(x^N;\wt{\xi}) - \nabla_x g_j(x^\infty;\wt{\xi}) \,\right\|_2\right].
		\end{array}
		\end{equation*}
		By the dominating convergence theorem and the continuity of both $g(\bullet;\xi)$ and $\nabla_x g_j(\bullet;\xi)$ from Assumption \ref{ass for ULLN},
		it follows that the last sum goes to $0$ as $N\to \infty$.
		Similarly, we can derive
		$\displaystyle\lim_{N\to \infty}\sup_{x\in X} \left| \, R_{x^\infty;\varepsilon^{\prime}}^\downarrow(x,x^N) - R_{x^\infty;\varepsilon^{\prime}}^\downarrow(x,x^\infty) \,\right| = 0$.
		It then follows from Lemma \ref{lemma: ULLN 1} that for all $x \in X$,
		\[
		R_{x^\infty;\varepsilon^{\prime}}(x^\infty,x^\infty) \, = \
		\calM(x^\infty) \, =  \, \lim_{N \rightarrow \infty} \calM_N(x^N)  \, \leq \, R_{x^\infty;\varepsilon^{\prime}}(x,x^\infty),
		\]
		which is the first conclusion of this lemma. The second conclusion can be obtained by noting that $R_{x^\infty;\varepsilon^{\prime}}(x,x^\infty) \leq R_{x^\infty}(x,x^\infty)$
	\end{proof}
	
	\begin{lemma}\label{lemma:bound for r}\rm
		Suppose that Assumption \ref{ass for ULLN} holds.
		Then for all $\omega \in \wh{\Omega}$, any $\varepsilon>0$, and all $\bar{x} \in X$,
		\[
		\left\{\begin{array}{ll}
		\left|\, r^{\uparrow}_{\bar{x};\varepsilon}(x,\bar{x};\omega) - r_{\bar{x};\varepsilon}^{\uparrow}(y,\bar{x};\omega) \,\right|
		\, \leq \, \mbox{Lip}_h(\zbld) \left[\, \text{Lip}_f(\xi) + C_g(\xi) \,\right] \, \| \, x - y \, \|_2 \\[0.1in]
		\left|\, r^{\downarrow}_{\bar{x};\varepsilon}(x,\bar{x};\omega) - r_{\bar{x};\varepsilon}^{\uparrow}(y,\bar{x};\omega) \, \right|
		\, \leq \, \mbox{Lip}_h(\zbld) \left[\, C_f(\xi) + \text{Lip}_g(\xi) \,\right] \, \| \, x - y \, \|_2
		\end{array}\right. , \quad \forall\; x,y \, \in \, \mathbb{R}^p.
		\]
	\end{lemma}
	\begin{proof}
		This can be easily seen by the following string of inequalities
		\[
		\begin{array}{rl}
		& \left| \, r^{\uparrow}_{\bar{x};\varepsilon}(x,\bar{x}; \omega) - r^{\uparrow}_{\bar{x};\varepsilon}(y,\bar{x};\omega)\, \right| \\[0.15in]
		\leq  & \mbox{Lip}_h(\zbld) \left( \, \left| \, f(x;\xi) - f(y; \xi) \,\right| +
		\displaystyle\max_{j \in {\cal A}_{g;\varepsilon}(\bar{x};\xi)}\left| \,(x - y)^{\top} \nabla_x g_j(\bar{x};\xi)\,\right| \,\right)\\[0.2in]
		\leq  & \mbox{Lip}_h(\zbld) \left[\, \text{Lip}_f(\xi) + C_g(\xi) \,\right] \,\| \, x - \, y \|_2
		\end{array}\]
		and similar ones for $r^{\downarrow}_{\bar{x};\varepsilon}(\bullet,\bar{x};\omega)$. 	
	\end{proof}
	
	Let $\mathcal{D}$ denote the set of directional stationary solution of \eqref{eq:population composite saa}, i.e.,
	\[
	\mathcal{D}\,\triangleq\, \left\{\, \bar{x} \in X\, \mid \mathcal{M}^{\, \prime}(\bar{x};x - \bar{x}) \geq 0, \epc \forall\; x\in X   \,\right\}.
	\]
	For any $x^{\, \prime} \in \mathbb{R}^n$, we also let $\mbox{dist}(x^{\, \prime},\mathcal{D}) \triangleq \displaystyle{
		\inf_{x \in {\cal D}}
	} \, \| \, x - x^{\, \prime} \|$, where $\| \bullet \|$ denotes the Euclidean norm of vectors.
	We are now ready to present the main convergence result, which shows that the limit of the empirical composite $\varepsilon$-strong d-stationary points
	is a d-stationary point of the population risk under mild conditions.
	
	\begin{theorem} \label{th:main}\rm
		Suppose that Assumption~\ref{ass for ULLN} holds.
		Let $\varepsilon > 0$ be given.
		Thus
		\begin{equation} \label{eq:distance to zero}
		\IP_{\infty}\left( \{ \omega^n \}_{n=1}^{\infty} \subset \wh{\Omega} \, \left| \,
		\displaystyle{
			\lim_{N \to \infty}
		} \, \mbox{dist}(x^{N;\varepsilon}(\omega^N),\mathcal{D}) \, = \, 0 \, \right. \right) \, = \, 1.
		\end{equation}
		In particular, if $\{x^{N;\varepsilon}(\omega^N)\}$ converges to $x^\infty$ almost surely, then $x^{\infty} \in {\cal D}$.
	\end{theorem}
	\begin{proof}  Suppose that (\ref{eq:distance to zero}) fails to hold.  Then there exists an event set
		$\mathcal{E}$ with positive probability such that for any family $\{ \omega^n \}_{n=1}^{\infty}$ in $\mathcal{E}$, we have $\displaystyle{
			\liminf_{N \to \infty}
		} \mbox{dist}(x^{N;\varepsilon}(\omega^N),\mathcal{D}) \, > \, 0 $.
		Let $\{ \omega^n \}_{n=1}^{\infty}$ be any such
		family.
		Since $X$ is compact, by passing to a subsequence
		if necessary, we may assume without loss of generality that
		the entire sequence $\left\{x^{N;\varepsilon}(\omega^N)\right\}$ converges to a point $x^\infty$.
		By 
		Lemma \ref{le: strong d}, we may deduce that $x^\infty$ is an optimal solution of
		$\displaystyle\operatornamewithlimits{minimize}_{y\in X}\;R_{x^\infty}(y,x^\infty)$. 
		Hence, we have that for any $x \in X$,
		\[
		\begin{array}{ll}
		\left(R_{x^\infty}^{\,\uparrow}(\,\bullet\,, x^\infty)\right)^\prime(x^\infty; x-x^\infty)  +
		\left(R_{x^\infty}^{\,\downarrow}(\,\bullet\,, x^\infty)\right)^\prime(x^\infty; x-x^\infty) \\ [0.15in]
		= \, \E_{\omega}\left[(h^\uparrow)^\prime(\,\bullet\,; \zbld)\left(m(x^\infty; \xi); \,f(\bullet;\xi)^{\prime}(x^\infty;x-x^\infty) -
		\displaystyle\max_{j \in {\cal A}_g(x^\infty; \,\xi)}  \nabla_x g_j(x^\infty; \xi)^\top (x-x^\infty) \, \right)\right]\\ [0.2in]
		+ \, \E_{\omega}\left[(h^\downarrow)^\prime(\,\bullet\,; \zbld)\left(m(x^\infty; \xi); \,\displaystyle\max_{j \in {\cal A}_f(x^\infty;\xi)}
		\nabla_x f_j(x^\infty, \xi)^\top (x-x^\infty) -g(\bullet;\xi)^{\prime}(x^\infty; x-x^\infty) \,\right)\,\right]\\ [0.2in]
		= \, \mathcal{M}^\prime(x^\infty; x- x^\infty) \geq 0,
		\end{array}\]
		where the equality is obtained by exchanging the directional derivative and the expectation based on \cite[Theorem 7.44]{ShapiroDR09}
		and Lemma \ref{lemma:bound for r}.
	\end{proof}
	
	Combining Theorem~\ref{th:main} with Proposition~\ref{pr:epsilon and zero}, we  obtain sufficient conditions for the consistency of the d-stationary points.  Before
	stating this result, we note that the $\bar{\varepsilon}$ in the latter proposition depends on the sample set $\{\xi^{\,n}\}_{n=1}^N$.
	In what follows, we provide a sufficient condition that guarantees the existence of a uniform $\bar{\varepsilon}$ that is independent of the samples
	so that the proposition can be applied to the sampled d-stationary points.  This condition is a sort of ``sufficient separation''
	between the component functions in the two pointwise maximum functions $f(\bullet;\xi)$ and $g(\bullet;\xi)$ at a given point $\bar{x}$.
	Specifically, we say that the (pointwise) {\sl sufficient separation condition} holds at $\bar{x} \in X$ if
	there exist positive constants $\delta$ and $c$ and a probability-one set $\Xi^{\, ss}_{\bar{x}}$ such that
	for all $\xi \in \Xi^{\, ss}_{\bar{x}}$,
	\[
	\begin{array}{rll}
	\displaystyle{
		\inf_{x \in \mathbb{B}_{\delta}(\bar{x})}
	} \, \left[ \, \displaystyle{
		\max_{j \in {\cal A}_f(x;\xi)}
	} \, f_j(x;\xi) - \displaystyle{
		\max_{j \notin {\cal A}_f(x;\xi)}
	} \, f_j(x;\xi) \, \right] & \geq & c \\ [0.2in]
	\displaystyle{
		\inf_{x \in \mathbb{B}_{\delta}(\bar{x})}
	} \, \left[ \, \displaystyle{
		\max_{j \in {\cal A}_g(x;\xi)}
	} \, g_j(x;\xi) - \displaystyle{
		\max_{j \notin {\cal A}_g(x;\xi)}
	} \, g_j(x;\xi) \, \right] & \geq & c.
	\end{array} 
	\]
	We first establish a lemma that establishes the equality of various index sets for points near any given point $\bar{x}$ satisfying this
	condition.
	
	\begin{lemma} \label{lm:sufficient separation} \rm
		Suppose that Assumption~\ref{ass for ULLN} 
		holds.  If $\bar{x}$ satisfies the sufficient separation condition with the associated probability-one set $\Xi^{\, ss}_{\bar{x}}$,
		then there exist positive constants $\bar{\varepsilon}$ and
		$\bar{x}$ such that
		$\mathcal{A}_{f;\varepsilon}(x;\xi) = \mathcal{A}_{f}(\bar{x};\xi)$ and
		$\mathcal{A}_{g;\varepsilon}(x;\xi) = \mathcal{A}_{g}(\bar{x};\xi)$ for
		all $\varepsilon \in [ \, 0,\bar{\varepsilon} \, ]$, all
		$x \in \mathbb{B}_{\bar{\delta}}(\bar{x})$, and all $\xi \in \wh{\Xi} \, \cap \, \Xi^{\, ss}_{\bar{x}}$.
	\end{lemma}
	
	\begin{proof}  To simplify the notation somewhat, we assume in the proof below that the two probability-one sets
		$\wh{\Xi}$ and $\Xi^{\, ss}_{\bar{x}}$ coincide.
		Let scalars $\bar{\varepsilon} \in ( \, 0,\, c/2 \, ]$ and $\bar{\delta} \in \left( \, 0, \min\left( \, \delta, \, \displaystyle{
			\frac{\bar{\varepsilon}}{4c_0}
		} \, \right) \, \right)$  be arbitrary.
		By Lemma~\ref{lm:set inclusion},
		for all $\varepsilon^{\, \prime} \in [ 0, \bar{\varepsilon}/2 ]$, all $\xi \in \wh{\Xi}$,
		and all pairs $x^1$ and $x^2$ in $X$ satisfying $\| \, x^1 - x^2 \, \|_2 \leq \bar{\delta}$, we have
		${\cal A}_{f;\varepsilon^{\, \prime}}(x^1;\xi) \subseteq {\cal A}_{f;\bar{\varepsilon}}(x^2; \xi)$ and
		${\cal A}_{g;\varepsilon^{\, \prime}}(x^1;\xi) \subseteq {\cal A}_{g;\bar{\varepsilon}}(x^2; \xi)$.
		In particular,
		with $\varepsilon^{\, \prime} = 0$, we have $ \mathcal{A}_{f}(\bar{x};\xi)\subseteq \mathcal{A}_{f;\bar{\varepsilon}}(x;\xi)$ and
		$\mathcal{A}_{g}(\bar{x};\xi) \subseteq \mathcal{A}_{g;
			\bar{\varepsilon}}(x;\xi)$ for all $x\in \mathbb{B}_{\bar{\delta}}(\bar{x})$ and all $\xi\in \wh{\Xi}$.
		We claim that the reverse inclusions hold.  Indeed,
		we derive from the proof of Proposition~\ref{pr:epsilon and zero} that
		$\mathcal{A}_{f;\varepsilon}(x;\xi) = \mathcal{A}_{f}(x;\xi)$ and
		$\mathcal{A}_{g;\varepsilon}(x;\xi) = \mathcal{A}_{g}(x;\xi)$
		for all $\varepsilon \in [ \, 0,\, c/2 \, ]$,  all $x \in \mathbb{B}_{\bar{\delta}}(\bar{x})$, 
		and all $\xi \in \wh{\Xi}$.
		Then for any $x\in \mathbb{B}_{\bar{\delta}}(\bar{x})$, if $j \in \mathcal{A}_{f;\bar{\varepsilon}}(x;\xi)$, we have $j \in \mathcal{A}_{f}(x;\xi)$; thus
		$f_j(x;\xi) \geq \displaystyle{
			\max_{1 \leq i \leq k_f}
		} \, f_i(x;\xi)$.  This implies that
		\[
		\begin{aligned}
		c_0 \, \|\bar{x} - x\|_2 +f_j(\bar{x}, \xi) \geq f_j(x, \xi) &
		\geq \max_{1 \leq i \leq k_f}f_i(\bar{x}, \xi) + \max_{1 \leq i \leq k_f}f_i(x, \xi) - \max_{1 \leq i \leq k_f}f_i(\bar{x}, \xi) \\[0.1in]
		& \geq \max_{1 \leq i \leq k_f}f_i(\bar{x}, \xi)- c_0 \, \|\bar{x} - x\|_2,
		\end{aligned}
		\]
		which further yields
		\[
		f_j(\bar{x}, \xi)  \geq \max_{1 \leq i \leq k_f}f_i(\bar{x}, \xi) - 2c_0 \, \|\bar{x} - x\|_2.
		\]
		We thus obtain $j \in \mathcal{A}_{f;\bar{\varepsilon}}(\bar{x};\xi) = \mathcal{A}_{f}(\bar{x};\xi)$.
		Therefore, ${\cal A}_{f;\bar{\varepsilon}}(x; \xi) = \mathcal{A}_{f}(\bar{x};\xi) = {\cal A}_{f;\varepsilon^{\, \prime}}(\bar{x}; \xi)$
		for all $\varepsilon^{\, \prime} \in [ \, 0,\bar{\varepsilon} \, ]$, all $x \in \mathbb{B}_{\bar{\delta}}(\bar{x})$ 
		and all $\xi \in \wh{\Xi}$.  Similarly we can prove the corresponding conclusion for $g$.
	\end{proof}
	
	
	Relying on Lemma~\ref{lm:sufficient separation}, we have the following corollary of Theorem~\ref{th:main} about the d-stationarity
	of convergent sequence of d-stationarity points of the empirical problems.
	
	\begin{corollary} \label{co:sufficient separation} \rm
		Suppose that Assumptions~\ref{ass for ULLN} holds.  Let $\{ \omega^n \}_{n=1}^{\infty} \subset \wh{\Omega}$ be arbitrary.
		For each positive integer $N$, let $x^N(\omega^N)$ be a d-stationary point of (\ref{eq:empirical composite saa})
		corresponding to $\{ \omega^n \}_{n=1}^N$.
		If the sequence $\{ x^N(\omega^N) \}$ converges to $x^\infty$ satisfying the
		sufficient separation condition, then $x^{\infty} \in {\cal D}$.
	\end{corollary}
	
	
	\begin{proof}  By Lemma~\ref{lm:sufficient separation}, it follows that for some scalar $\bar{\varepsilon} > 0$, it holds that
		for all $N$ sufficiently large, $\mathcal{A}_{f;\bar{\varepsilon}}(x^N(\omega^N);\xi^{n}) = \mathcal{A}_{f}(x^N(\omega^N);\xi^{n})$
		and $\mathcal{A}_{g;\bar{\varepsilon}}(x^N(\omega^N);\xi^{n}) = \mathcal{A}_{g}(x^N(\omega^N);\xi^{n})$.
		Therefore, $x^N(\omega^N)$ is a composite $\bar{\varepsilon}$-strong d-stationary point of (\ref{eq:empirical composite saa}) for all $N$ sufficiently large.
		The desired conclusion follows readily from Theorem \ref{th:main}.
	\end{proof}
	
	\begin{remark} \rm
		It is possible to state a probabilistic conclusion of Corollary~\ref{co:sufficient separation} similar to that in Theorem~\ref{th:main}.
		For this to hold, we need to strengthen the sufficient separation condition to all d-stationary solutions in ${\cal D}$; more importantly,
		the same constants $c$ and $\delta$ have to hold uniformly for all such solutions.  We omit the details and leave the Corollary in its
		pointwise form as stated above.
	\end{remark}
	
	\section{Asymptotic Distribution of the Stationary Values}
	
	In this and and the next section, we will work with sequences of composite $\varepsilon$-strong d-stationary solution of (\ref{eq:empirical composite saa})
	for an arbitrary fixed $\varepsilon > 0$.
	Our goal in this section is to derive an asymptotic distribution of the sequence of stationary values $\{ {\cal M}_N(x^N) \}$, where
	for each $N$, $x^N$ is a composite $\varepsilon$-strong d-stationary solution of (\ref{eq:empirical composite saa}),
	under the following piecewise affine assumption:
	
	\begin{assumption}\label{assu: PA}\rm
		The function $m(\bullet;\xi)$ is a piecewise affine function, i.e., each $f_j(\bullet;\xi)$	and $g_j(\bullet\,;\xi)$ are affine functions.
	\end{assumption}
	
	An important consequence of this special structure is the following lemma.
	
	\begin{lemma}\label{lemma:equiv epsilon}\rm
		Suppose that Assumption \ref{ass for ULLN}~(a1) and Assumption \ref{assu: PA} hold.
		Then for any $\varepsilon>0$, there exists $\delta>0$ such that for any $\varepsilon^{\prime} \in \left[ 0,\frac{\varepsilon}{2} \right]$,
		any $x$ and $\bar{x}$ satisfying $\| x - \bar{x} \|_2 \leq \delta$, and all $\omega \in \wh{\Omega}$,
		\begin{equation} \label{eq:equalities in r}
		r^{\updownarrow}_{\bar{x};\varepsilon^{\prime}}(x,\bar{x};\omega) \, = \, r^{\updownarrow}_{x;\varepsilon}(x,x;\omega) \, = \, r^{\updownarrow}_x(x,x;\omega).
		\end{equation}
		Hence, for any family $\{ \omega^n \}_{n=1}^{\infty} \subset \wh{\Omega}$
		\[
		R_{N;\bar{x};\varepsilon^{\prime}}( x,\bar{x}) \, = \, R_{N;x;\varepsilon}(x,x) \, = \, {\cal M}_N(x)
		\epc \mbox{and} \epc
		R_{\bar{x};\varepsilon^{\prime}}(x,\bar{x}) = R_{x;\varepsilon}(x,x) \, = \, {\cal M}(x).
		\]
	\end{lemma}
	
	\begin{proof}
		It follows from Lemma \ref{lm:set inclusion} that there exists a positive scalar $\delta$ such that for any
		$\varepsilon^{\prime} \in [ \, 0, \frac{\varepsilon}{2} \, ]$ and any $x$ and $\bar{x}$ satisfying $\| x - \bar{x}\|_2 \leq \delta$,
		and any $\xi$ in the probability-one set $\wh{\Xi}$,
		\[
		{\cal A}_{f}(x;\xi) \, \subseteq {\cal A}_{f; \varepsilon^{\prime}}(\bar{x};\xi) \, \subseteq \, {\cal A}_{f;\varepsilon}(x;\xi)
		\epc \mbox{and} \epc {\cal A}_{g}(x;\xi_n) \subseteq
		{\cal A}_{g;\varepsilon^{\prime}}(\bar{x};\xi) \, \subseteq \, {\cal A}_{g;\varepsilon}(x;\xi)
		\]
		Noticing that when $m(\bullet;\xi)$ is a piecewise affine function, we have
		$r^{\updownarrow}_{\bar{x};\varepsilon}(x,x^1;\omega) = r^{\updownarrow}_{\bar{x};\varepsilon}(x,x^2;\omega)$ for any
		$x$, $x^1$, $x^2$, and $\bar{x}$ in $X$, any $\varepsilon \geq 0$, and any $\omega \in \wh{\Omega}$.
		Therefore, for any $\varepsilon^{\prime} \in \left[ 0,\frac{\varepsilon}{2} \right]$,
		and any $x$ and $\bar{x}$ satisfying $\| x - \bar{x} \|_2 \leq \delta$, we derive for any $\omega \in \wh{\Omega}$,
		\[
		\begin{array}{lll}
		r^{\updownarrow}_{\bar{x};\varepsilon^{\prime}}(x,\bar{x};\omega) \, = \, r^{\updownarrow}_{\bar{x};\varepsilon^{\prime}}(x,x;\omega)
		& \leq & r^{\updownarrow}_x(x,x;\omega) \\ [0.1in]
		& = & r^{\updownarrow}_{x;\varepsilon}(x,x;\omega) \, \leq \, r^{\updownarrow}_{\bar{x};\varepsilon^{\prime}}(x,x;\omega) \, = \,
		r^{\updownarrow}_{\bar{x};\varepsilon^{\prime}}(x,\bar{x};\omega).
		\end{array}
		\]
		Consequently, equalities hold throughout, establishing the equalities in (\ref{eq:equalities in r}).
	\end{proof}
	
	An interesting consequence of Lemma~\ref{lemma:equiv epsilon} is that if $x^{\infty}$
	is as described in Lemma~\ref{le: strong d},
	then $x^{\infty}$ is a local minimizer of the population level problem
	(\ref{eq:population composite saa}).  This observation enables us to establish the
	following consistency result of local minima.
	
	\begin{corollary}\label{cor: local min consistency} \rm
		Suppose that Assumption \ref{ass for ULLN}~(a1) and Assumption \ref{assu: PA} hold.
		If $\{x^{N;\varepsilon}(\omega^N)\}$ converges to $x^\infty$ almost surely, then
		$x^{\infty}$ is a local minimizer of the population level
		problem (\ref{eq:population composite saa}).
	\end{corollary}
	
	\begin{proof}
		Under the given assumptions, we know if $\{x^{N;\varepsilon}(\omega^N)\}$ converges
		to $x^\infty$ almost surely, then $x^\infty \in \calD$.  By
		Lemma~\ref{lemma:equiv epsilon}, as long as $\| x - x^\infty \|_2 \leq \delta$,
		we have $R_{x^\infty;\varepsilon^{\prime}}(x,x^\infty) = R_{x;\varepsilon}(x,x) =
		{\cal M}(x)$.  Since $x^\infty \in \displaystyle{
			\operatornamewithlimits{\mbox{argmin}}_{x \in X}
		} \, R_{x^\infty;\varepsilon^{\prime}}(x,x^\infty)$, we may conclude that
		$x^\infty$ minimizes ${\cal M}(x)$ locally on $X$.
	\end{proof}
	
	Besides being instrumental in establishing the consistency of a convergent sequence
	of composite $\varepsilon$-strong d-stationary solutions
	of (\ref{eq:empirical composite saa}), Lemma~\ref{le: strong d}, along with
	Lemma~\ref{lemma:equiv epsilon}, enables us to derive the following theorem
	that provides the asymptotic distribution of the stationary values $\calM_N(x^N)$
	for such a sequence $\{ x^N \}$.
	In what follows, we use the notation $\xrightarrow{d}$ to denote the
	convergence in distribution, and $\mathcal{N}(\mu, \sigma^2)$ denotes the
	normal distribution with mean $\mu$ and variance $\sigma^2$.
	In addition, we use $\mbox{Var}[\bullet]$ to represent the variance of a random
	variable.  We recall the objective function
	${\cal L}(x;\omega) = h(m(x;\xi);\zbld)$ of the population
	problem (\ref{eq:population composite saa}).
	
	\begin{theorem}\label{weak convergence}\rm
		Suppose Assumptions \ref{ass for ULLN} and \ref{assu: PA} hold.
		Let $\{ x^{N;\varepsilon} \}$ be a composite $\varepsilon$-strong d-stationary point of (\ref{eq:empirical composite saa}) corresponding to a family
		$\{ \omega^n \}_{n=1}^{\infty} \subset \wh{\Omega}$.
		If $x^{N;\varepsilon}$ converges to $x^\infty$ almost surely and $\calL(x^\infty;\bullet)$ is square integrable, then
		\[
		\sqrt{N} \, \left[ \, \calM_N(x^{N;\varepsilon}) - \calM(x^{\, \infty}) \, \right] \, \xrightarrow{d} \, \inf_{x \in S} \; \mathbb{G}_x,
		\]
		where $\mathbb{G}_x$ follows $\mathcal{N}\left(\,0, \mbox{Var}\left[ {\cal L}(x;\wt{\omega}) \right]\,\right)$
		and $S \triangleq \displaystyle{
			\operatornamewithlimits{\mbox{argmin}}_{x \in \mathbb{B}_{\delta}(x^{\infty})}
		} \, {\cal M}(x)$ where $\delta$ is such that Lemma~\ref{lemma:equiv epsilon} holds.
		In particular, if $S = \{x^\infty\}$, then
		\[
		\sqrt{N} \left(\calM_N(x^{N;\varepsilon}) - \calM(x^{\, \infty})\right) \, \xrightarrow{d} \, \mathcal{N}(0, \mbox{Var}\left[ \calL(x^\infty;\widetilde{\omega})\right]).
		\]
	\end{theorem}
	\begin{proof}
		As $x^{N;\varepsilon}$ converges to $x^\infty$ almost surely, it follows from Lemma \ref{lemma:equiv epsilon} that for all such sufficiently large $N$
		and any $\varepsilon^{\prime} \in \left[ \, 0, \frac{\varepsilon}{2} \, \right]$,
		\[
		\begin{array}{l}
		\sqrt{N} \, \left[ \, \calM_N(x^{N;\varepsilon}) - \calM(x^{\, \infty}) \, \right] \,  = \,
		\sqrt{N} \, \left[ \, R_{N;x^{N;\varepsilon}}(x^{N;\varepsilon},x^{N;\varepsilon}) - R_{x^{\, \infty}; \varepsilon^{\prime}}(x^{\, \infty}, x^{\, \infty}) \, \right]
		\\ [0.1in]
		= \, \sqrt{N} \, \left[ \, R_{N;x^{\, \infty}; \varepsilon^{\prime}}(x^N, x^{\, \infty}) - R_{x^{\, \infty}; \varepsilon^{\prime}}(x^{\, \infty}, x^{\, \infty}) \, \right],
		\end{array}
		\]
		almost surely.
		Notice that for any $\varepsilon^{\prime} \in  \left[ \, 0, \frac{\varepsilon}{2} \, \right]$,
		\[
		R_{N;x^{\, \infty}; \varepsilon^{\prime}}(x^N, x^{\, \infty}) \, = \,
		R_{N;x^{N;\varepsilon};\varepsilon}(x^N, x^N) \leq R_{N;x^{N;\varepsilon};\varepsilon}(x,x^{N;\varepsilon})
		\, \leq \, R_{N;x^\infty;\varepsilon^{\prime}}(x, x^N), \quad \forall \;x\in X,
		\]
		almost surely.
		This implies that $x^{N;\varepsilon} \in \displaystyle\operatornamewithlimits{argmin}_{x\in X}\,R_{N;x^{\, \infty}; \varepsilon^{\prime}}(x, x^{\, \infty})$
		almost surely.  We also know that $x^\infty \in \displaystyle\operatornamewithlimits{argmin}_{x\in X}\, R_{x^{\, \infty}; \varepsilon^{\prime}}(x, x^{\, \infty})$.
		It follows from Lemma \ref{lemma:bound for r} that  there exists a square integrable function $C(\omega)$ such that for all $\omega\in \widehat{\Omega}$,
		\[
		\left|{r}_{x^\infty; \varepsilon^{\prime}}(x^1,x^\infty;\omega) - {r}_{x^\infty;\varepsilon^{\prime}}(x^2, x^\infty;\omega)\right| \, \leq \,
		C(\omega) \, \| \, x^1 - x^2 \, \|_2,
		\]
		which shows that \cite[Condition (A2), page 164]{ShapiroDR09} holds.  In addition, since
		${r}_{x^\infty; \varepsilon^{\prime}}(x^\infty,x^\infty;\widetilde{\omega}) = \calL(x^\infty;\widetilde{\omega})$ is square integrable,
		\cite[Condition (A1), page 164]{ShapiroDR09} is satisfied.
		By applying \cite[Theorem 5.7]{ShapiroDR09} and restricting to the almost sure set, we can derive that
		\[
		\sqrt{N} \, \left[ \, \calM_N(x^N) - \calM(x^{\, \infty}) \, \right]  \, \xrightarrow{d} \, \inf_{x \in {\cal S}} \; \mathbb{G}_x,
		\]
		where $\mathbb{G}_x$ follows $\mathcal{N}\left(\,0, \mbox{Var}\left[ {r}_{x^\infty; \varepsilon^{\prime}}(x, x^\infty;\omega) \right]\,\right)$
		and $S$ is the set of  minimizers of $\displaystyle\operatornamewithlimits{minimize}_{x \in B_\delta(x^\infty)}R_{x^{\, \infty}; \varepsilon^{\prime}}(x, x^{\, \infty})$.
		Again by leveraging Lemma \ref{lemma:equiv epsilon}, we have $R_{x^{\, \infty}; \varepsilon^{\prime}}(x, x^{\, \infty}) = \calM(x)$ and
		${r}_{x^\infty; \varepsilon^{\prime}}(x, x^\infty;\omega) = \calL(x, \omega)$ almost surely for all $\|x - x^\infty\|_2 \leq \delta$.
		Then $\mbox{Var}\left[ {r}_{x^\infty; \varepsilon^{\prime}}(x, x^\infty;\omega) \right] = \mbox{Var}\left[ \calL(x, \omega) \right]$.
		Thus the first conclusion follows.  The second conclusion is obvious.
	\end{proof}
	
	\begin{remark} \rm
		If $S= \{x^{\, \infty}\}$, then we can use
		\[
		\wh{V}_N(x^N) \, \triangleq \, \displaystyle{
			\frac{1}{N}
		} \, \displaystyle{
			\sum_{n = 1}^{N}
		} \, \left[ \, \calL(x^N;\omega^n)- \displaystyle{
			\frac{1}{N}
		} \, \displaystyle{
			\sum_{n^{\prime} = 1}^N
		} \, \calL(x^N;\omega^{n^{\prime}}) \, \right]^2
		\]
		to estimate $\mbox{Var}\left[ \calL(x^\infty; \widetilde{\omega})\right]$.
		Consistency of this estimator can be demonstrated by showing the uniform convergence of $\wh{V}_N(x)$ to $\mbox{Var}\left[ \calL(x; \widetilde{\omega})\right]$
		over $x \in X$ and the continuity of $\mbox{Var}\left[ \calL(x; \widetilde{\omega})\right]$ at $x^\infty$. Then by Slutsky theorem, we can show that
		\[
		\sqrt{N \, \wh{V}_N(x^N)} \, \left( \, \calM_N(x^N) - \calM(x^{\, \infty}) \, \right) \, \xrightarrow{d} \, \mathcal{N}(0,1).
		\]
	\end{remark}
	
	\section{Convergence Rate of the Stationary Points}
	
	Throughout this section, each member of the family of random variables $\{ \omega^n \}_{n=1}^{\infty}$ is assumed to be in the probability-one set
	$\wh{\Omega}$; we also fix a scalar $\varepsilon > 0$.  For each $N$, we write $x^N$ as the shorthand for a
	composite $\varepsilon$ strong d-stationary solution $x^{N;\varepsilon}(\omega^N)$ of
	(\ref{eq:empirical composite saa}).  Assuming that $\{ x^N \}$ converges to $x^{\, \infty} \in {\cal D}$ almost surely,
	we aim to show, under the setting of Theorem~\ref{th:main} and some additional assumptions,
	the existence of a sequence of positive scalars
	$\{ \rho_N \}_{N=1}^{\infty}$ such that the sequence
	$\left\{ \, \rho_N \, \| \, x^N - x^{\infty} \, \|_2 \, \right\}$ is bounded in probability;
	that is to say, for every $\varepsilon > 0$, there exist a scalar $C_{\varepsilon} > 0$ and a positive integer $N_{\varepsilon}$ such that
	$\| \, x^N - x^{\infty} \, \| = \mbox{O}_{\IP_{\infty}}(\rho_N^{-1})$, using the big-O notation in probability theory
	\cite[Section~2.2]{van2000empirical}.  In what follows, we say that a random variable $w_N$ is $\Gamma_{\IP}(1)$ if both $w_N$ and
	$w_N^{-1}$ are $\mbox{O}_{\IP}(1)$.   Besides the almost sure convergence of $\{ x^N \}$ to $x^{\infty}$, we further assume:
	
	%
	
	\begin{assumption}\label{assu:rate}\rm
		(b1) There exist a positive scalar $q$  and a random variable $w_N = \Gamma_{\IP}(1)$ such that for all $N$ sufficiently large,
		\begin{equation*}
		{R}_{ x^{\,N};\varepsilon}( x^N, x^{\, N})- {R}_{x^{\,N};\varepsilon}( x^{\, \infty}, x^{\, N})  \geq w_N \, \|x^N - x^\infty\|_2^{\,q},
		\end{equation*}
		almost surely.
		
		\gap
		
		(b2) There exist positive scalars $\alpha < q$, $c > 0$ and $\delta > 0$ such that for all $N$ sufficiently large, there exists a function
		$\Phi_N$ for which $w \rightarrow w^{-\alpha}\,\Phi_N(w)$ is non-increasing on $(0, \delta]$ and 	
		\[
		\E\left[ \, \displaystyle{
			\sup_{x\in \mathbb{B}_{\delta}(x^{\, \infty})}
		} \, \sqrt{N} \, \left| \,R_{N;x;\varepsilon}( x^{\, \infty}, x) - R_{x;\varepsilon}( x^{\, \infty}, x) - R_{N;x;\varepsilon}(x, x) +
		R_{x;\varepsilon}(x,x) \, \right| \, \right] \,	\leq c \, \Phi_N(\delta).
		\]
		where the expectation is taken over the samples  $\left\{(\xi^1,\zbld^1), \ldots, (\xi^N,\zbld^N)\right\}$.
		
		\vskip 0.05in
		
		(b3) A sequence of positive scalars $\{\rho_N\}$ converging to $\infty$ exists such that
		$\rho_N^q \,\Phi_N\left(\rho_N^{-1}\right) \leq \sqrt{N}$.
	\end{assumption}
	
	The rate result below does not require Assumption~ \ref{ass for ULLN}.
	
	\begin{theorem}\label{thm:rate}\rm
		Assume the setting of this section, including the above Assumption~\ref{assu:rate}.
		It holds that $\|x^{\, N}-x^\infty\|_2 = \mbox{O}_{\IP_{\infty}}(\rho^{-1}_N).$
	\end{theorem}
	
	\begin{proof}
		From Lemma \ref{lemma:d-stat} that $x^N\in \displaystyle\operatornamewithlimits{argmin}_{x\in X}\, R_{N;x^N;\varepsilon}(x, x^N)$ for any $N\geq 1$.
		We have
		\begin{equation}\label{eq in 7.2}
		\begin{array}{ll}
		0 & \leq \, R_{N;x^N;\varepsilon}(x^{\infty},x^N) - R_{N;x^N;\varepsilon}(x^N,x^N) \\ [0.15in]
		& = \, \left[ \, R_{N;x^N;\varepsilon}(x^{\infty},x^N) - R_{x^N;\varepsilon}(x^{\infty},x^N) \, \right] -
		\left[ \, R_{N;x^N;\varepsilon}(x^N,x^N) - R_{x^N;\varepsilon}(x^N,x^N) \, \right]	+ \\ [0.15in]
		& \hspace{0.18in} \left[ \, R_{x^N;\varepsilon}(x^{\infty},x^N) - R_{x^N;\varepsilon}(x^N,x^N) \, \right].
		\end{array}
		\end{equation}
		For any positive integer $j$ and the given positive scalar $\delta$ in (b2), we define a set $S_{N,\,j}$ as
		\begin{equation*}
		S_{N,\,j} \triangleq \left\{x \in X \; | \; 2^{\, j} < \rho_N\, \|x - x^{\, \infty}\|_2 \leq \min(\,2^{\, j+1}, \delta \rho_N \,)  \right\}.
		\end{equation*}
		If $x^{\, N} \in S_{N,\,j}$, restricting to the almost sure set in Assumption \ref{assu:rate}~(b1) if necessary, we have
		\[ \begin{array}{l}
		\underbrace{ \displaystyle\sup_{x\in \mathbb{B}_{\delta}(x^{\, \infty})}\left| \, R_{N;x;\varepsilon}(x^{\, \infty},x) - R_{x;\varepsilon}(x^{\, \infty},x) -
			R_{N;x;\varepsilon}(x,x) + R_{x;\varepsilon}(x,x) \, \right|}_{\mbox{denoted RHS$_{N,j}$}} \\[0.5in]
		\geq \, R_{x^N;\varepsilon}(x^N,x^N) -  R_{x^N;\varepsilon}(x^{\, \infty},x^N)
		\, \geq \, w_N \, \|x^N - x^{\, \infty}\|_2^q \, \geq \, w_N \, \left( \, 2^{\,j}\rho_N^{-1} \, \right)^{q}
		\end{array}
		\]
		where the two inequalities are by Assumption \ref{assu:rate}~(b1) and (b2), respectively.
		In the rest of the proof, the probabilities are all $\IP_{\infty}$.  For simplicity, we drop the subscript $\infty$.
		Thus for some constant $K_1$,
		\[\begin{array}{l}
		\IP\left(\, x^{\, N} \in S_{N,\,j}, \; \displaystyle w_N \geq K_1\,\right)
		\, \leq \, \IP\left(\mbox{RHS$_{N,j}$}
		\geq \, K_1 \, \left(\, 2^{\,j}\rho_N^{-1}\right)^{q} \, \right) \\ [0.15in]
		\leq \, \E\left[K^{-1}_1(2^{-j}\rho_N)^q  \;\mbox{RHS$_{N,j}$}
		\, \right] \epc \mbox{by Markov inequality} \\ [0.15in]
		\leq \, \displaystyle\frac{c\,\Phi_N\left(2^{\, j} \, \rho_N^{-1}\right) \rho_N^{q}}{K_1\sqrt{N}2^{\,j\,q}} \leq \displaystyle\frac{c\,2^{\,(\alpha-q)\,j}\, \Phi_N\left(\rho_N^{-1}\right) \rho_N^q}{K_1\sqrt{N}}
		\epc \mbox{by Assumption \ref{assu:rate}~(b2)}.
		\end{array}
		\]
		Therefore, given any positive integer $M$, we have that for all $N$ sufficiently large,
		\[
		\begin{array}{l}
		\IP\left(\, \rho_N\,\|x^{\,N} - x^{\,\infty}\|_2 > 2^{\,M}\,\right) \, \leq \, \IP\left(\, w_N < K_1\,\right) + \\ [0.1in]
		\hspace{0.5in} \IP\left(\, \rho_N\,\|x^{\,N} - x^{\,\infty}\|_2 > 2^{\,M}, \,\|x^{\,N} - x^{\,\infty}\|_2 \leq \delta, \, w_N \geq K_1\,\right)
		+ \IP\left(\, \|x^{\,N} - x^{\,\infty}\|_2 > \delta\,\right)
		\\ [0.1in]
		\leq \, \displaystyle{
			\sum_{j \geq M}
		} \, \IP\left(\, x^{\, N} \in S_{N,\,j}, w_N \geq K_1 \,\right) + \IP\left(\, \|x^{\,N} - x^{\,\infty}\|_2 > \delta\,\right) +\IP\left(\, w_N < K_1\,\right) \\ [0.2in]
		\leq \, \displaystyle\frac{c\,\rho_N^q \Phi_N\left(\rho_N^{-1}\right)}{K_1\sqrt{N}}  \,\sum_{j \geq M} 2^{\,(\alpha-q)\,j} + \IP\left(\, \|x^{\,N} - x^{\,\infty}\|_2 > \delta\,\right) + \IP\left(\, w_N<  K_1\,\right) \\ [0.2in]
		\leq \, \displaystyle{
			\frac{c}{K_1}
		} \, \displaystyle{
			\sum_{j \geq M}
		} \, 2^{\, (\alpha-q)j} + \IP\left(\, \|x^{\,N} - x^{\,\infty}\|_2 > \delta\,\right) + \IP\left(\, w_N< K_1\,\right) \epc \mbox{by Assumption \ref{assu:rate}~(b3)}.
		\end{array}
		\]
		One can thus make $\IP\left(\,\rho_N\,\|x^{\,N} - x^{\,\infty}\|_2 > 2^{\,M}\,\right)$
		arbitrarily small by choosing $M$ and $N$ sufficiently large and $K_1$ sufficiently small accordingly.
	\end{proof}
	
	Next, we provide sufficient conditions for Assumption~(b1) to hold.
	
	\begin{proposition}\label{pr:strong local}\rm
		Suppose that Assumption \ref{assu: PA} holds.
		Then assumption (b1) holds with $q =2$ if for some $\varepsilon^{\prime} \in \left[ \, 0, \frac{\varepsilon}{2} \right]$,
		$R_{x^\infty;\varepsilon^{\prime}}(\,\bullet\,,x^\infty)$ is locally strongly convex at
		$x^\infty$, i.e., there exist positive scalars $\delta$ and $c$ such that,
		\[
		R_{x^\infty;\varepsilon^{\prime}}(x,x^\infty)  - R_{x^\infty;\varepsilon^{\prime}}(x^\infty,x^\infty) \, \geq c\,\|x - x^\infty\|_2^2,
		\quad \forall\; x \, \in \, \mathbb{B}_\delta(x^\infty).
		\]	
	\end{proposition}
	
	\begin{proof}
		It follows from Lemma \ref{lemma:equiv epsilon} that for all $N$ sufficiently large.
		\[
		R_{ x^{\,\infty};\varepsilon^{\prime}}( x^N, x^{\, \infty}) \, = \, {R}_{ x^{\,N};\varepsilon}( x^N, x^{\, N})
		\epc \mbox{and} \epc
		R_{N; x^{\,\infty};\varepsilon^{\prime}}( x^N, x^{\, \infty}) = {R}_{N; x^{\,N};\varepsilon}( x^N, x^{\, N}) \epc \mbox{almost surely}
		\]
		Thus we can show that
		\[
		\begin{aligned}
		R_{ x^{\,N};\varepsilon}( x^N, x^{\, N})- {R}_{x^{\,N};\varepsilon}( x^{\, \infty}, x^{\, N})
		\geq &  \;	{R}_{ x^{\,\infty};\varepsilon^{\prime}}( x^N, x^{\, \infty}) - 	{R}_{ x^{\,\infty};\varepsilon^{\prime}}( x^N, x^{\, N}) \\[0.1in]
		= &  \;	{R}_{ x^{\,\infty};\varepsilon^{\prime}}( x^N, x^{\, \infty}) - 	{R}_{ x^{\,\infty};\varepsilon^{\prime}}( x^{\, \infty}, x^{\, \infty})
		\geq  c \, \|x^{\, N}- x^{\, \infty}\|_2^{2},
		\end{aligned}
		\]
		almost surely, where the last inequality is obtained by the assumed local strong convexity of $R_{x^\infty;\varepsilon^{\prime}}(\,\bullet\,,x^\infty)$ at $x^\infty$.
	\end{proof}
	
	\begin{remark}\rm
		By Theorem \ref{th:main},
		$x^\infty$ is a minimizer of  $R_{x^\infty;\varepsilon^{\prime}}(\bullet,x^\infty)$ for any $\varepsilon^{\prime} \in \left[ \, 0, \frac{\varepsilon}{2} \right]$.
		Thus the assumption in Proposition~\ref{pr:strong local} is a mild strengthening of this minimizing property of $x^{\infty}$.
	\end{remark}
	
	If each $f_j(\,\bullet\,;\xi)$ and $g_j(\,\bullet\,; \xi)$ are affine functions, based on the proof of Proposition~\ref{pr:strong local},
	we can show that in the equation \eqref{eq in 7.2},
	\[
	\begin{aligned}
	&\left[ R_{N;x^N;\varepsilon}(x^{\infty},x^N) - R_{x^N;\varepsilon}(x^{\infty},x^N) \, \right] - \left[ \, R_{N;x^N;\varepsilon}(x^N,x^N) - R_{x^N;\varepsilon}(x^N,x^N) \right ]\\[0.1in]
	\leq & \left[ {R}_{N; x^{\,\infty};\varepsilon^{\prime}}(x^{\, \infty}, x^{\, \infty}) - R_{x^N;\varepsilon}(x^{\infty},x^N) \, \right] - \left[ \, {R}_{N; x^{\,\infty};\varepsilon^{\prime}}( x^N, x^{\, \infty}) - {R}_{ x^{\,\infty};\varepsilon^{\prime}}( x^N, x^{\, \infty}) \right ] \\[0.1in]
	\leq & \left[ {R}_{N; x^{\,\infty};\varepsilon^{\prime}}(x^{\, \infty}, x^{\, \infty}) - R_{x^\infty;2\varepsilon}(x^{\infty},x^\infty) \, \right] - \left[ \, {R}_{N; x^{\,\infty};\varepsilon^{\prime}}( x^N, x^{\, \infty}) - {R}_{ x^{\,\infty};\varepsilon^{\prime}}( x^N, x^{\, \infty}) \right ] \\[0.1in]
	= & \left[ {R}_{N; x^{\,\infty};\varepsilon^{\prime}}(x^{\, \infty}, x^{\, \infty}) - R_{x^\infty;\varepsilon^{\prime}}(x^{\infty},x^\infty) \, \right] - \left[ \, {R}_{N; x^{\,\infty};\varepsilon^{\prime}}( x^N, x^{\, \infty}) - {R}_{ x^{\,\infty};\varepsilon^{\prime}}( x^N, x^{\, \infty}) \right],
	\end{aligned}
	\]
	almost surely,
	for all $N$ sufficiently large and any $\varepsilon^\prime \in [0, \frac{\varepsilon}{2}]$.  Again, the almost sure set does not depend on $\varepsilon$
	and parameters $x^N$ and $x^\infty$.  We can thus replace Assumption \ref{assu:rate}~(b2) by the following one so that Theorem \ref{thm:rate} still holds.
	
	\vskip 0.05in
	
	(b2$^{\, \prime}$) Assume that each $f_j(\,\bullet\,;\xi)$ and $g_j(\,\bullet\,; \xi)$ are affine functions. There exist positive scalars $\alpha < q$, $c > 0$ and $\delta > 0$
	such that for all $N$ sufficiently large, there exists a function $\Phi_N$ for which $w \rightarrow w^{-\alpha}\,\Phi_N(w)$ is non-increasing on $(0, \delta]$ and 	
	\begin{equation*}
	\begin{array}{ll}
	\E\left[\displaystyle\sup_{x \in \mathbb{B}_{\delta}(x^{\, \infty})}\sqrt{N}\left| \,{R}_{N; x^\infty;\varepsilon^{\prime}}( x^{\, \infty}, x^{\, \infty}) - {R}_{x^\infty;\varepsilon^{\prime}}( x^{\, \infty}, x^{\, \infty}) -{R}_{N;x^\infty;\varepsilon^{\prime}}(x, x^{\, \infty}) + {R}_{x^\infty;\varepsilon^{\prime}}(x, x^{\, \infty}) \,\right| \, \right]	\\[0.3in]
	\leq   c\, \Phi_N(\delta),
	\end{array}
	\end{equation*}
	for all $\varepsilon^{\prime} \in \left( \, 0, \frac{\varepsilon}{2} \, \right]$, where the expectation is taken over the samples
	$\left\{(\xi^1,\zbld^1), \ldots, (\xi^N,\zbld^N)\right\}$.
	
	\gap
	
	The following corollary does not require a proof.
	
	\begin{corollary}\label{coro: b3'} \rm
		Assume the setting of this section and Assumptions \ref{assu:rate}~(b1), (b2$^{\, \prime}$), and (b3) hold.
		It holds that $\| x^N - x^\infty\|_2 = \mbox{O}_{\IP_{\infty}}(\rho_N^{-1})$.
	\end{corollary}
	
	%
	
	
	An advantage of assuming (b2$^{\, \prime}$) is that we can derive a sufficient condition for it to hold.
	This condition is based on the concept of bracketing number in asymptotic statistics \cite{van1998asymptotic}
	to measure the size of some function class $\calF$.
	We mainly consider the bracketing number relative to the $L_2(\mathbb{P})$-norm.  Given two functions $\ell$ and $u$, the bracket $[\ell, u]$ is the set of all
	functions $f$ with $\ell \leq f \leq u$.  A $\sigma$-bracket in $L_2(\mathbb{P})$ is a bracket $[\ell,u]$ with $\| \ell - u \|_2 \leq \sigma$.
	The bracketing number $\calN_{[\;]}(\sigma, \calF, L_2(\mathbb{P}))$ is the minimum number of $\sigma$-brackets needed to cover $\calF$.
	For the bracketing number relative to $\ell_2$ norm in Euclidean space, the definition can be adapted similarly.  In the following, we cite
	an important lemma, without proof, that is useful to obtain the bound in Assumption~(b2$^{\, \prime}$).
	
	\begin{lemma}\rm (c.f. \cite[Corollay 19.35]{van2000empirical})\label{lemma: entropy bound}
		For any class $\calF$ of measurable functions $f: \Omega \mapsto \mathbb{R}$ with envelope function
		$F \triangleq \sup_{f \in \calF}\left|f\right|$, there exists a positive scalar $K_1$ such that
		\[
		\sqrt{N}\,\E\left[\,\sup_{f \in \calF}\left| \, \frac{1}{N}\sum_{n=1}^N f(\omega_n) - \E\left[f(\omega)\right] \, \right| \; \right]
		\, \leq \, K_1 \, \int_{0}^{\|F\|_2} \sqrt{\log \calN_{[\;]}(\sigma, \calF, L_2(\mathbb{P}))}d\sigma.
		\]
	\end{lemma}
	
	\begin{proposition}\label{entropy for lip}\rm
		If Assumption \ref{ass for ULLN} holds,
		then Assumption (b2$^{\, \prime}$) holds with $\Phi_N(\delta) = \delta$.
	\end{proposition}
	
	\begin{proof}
		For any $\varepsilon' \in [0, \frac{\varepsilon}{2}]$, consider the functional class
		\[\calF \,\triangleq\, \left\{\,{r}_{x^\infty; \varepsilon^{\prime}}(x, x^\infty;\omega) - {r}_{x^\infty; \varepsilon^{\prime}}(x^{\, \infty}, x^{\, \infty};\omega)  \; \mid \;  x\in \mathbb{B}_{\delta}(x^{\, \infty}) \,\right\}.\]
		It follows from Lemma \ref{lemma:bound for r} that there exists a square integrable function $C(\omega)$ such that
		\begin{equation}\label{lip of difference}
		\begin{aligned}
		|\, {r}_{x^\infty; \varepsilon^{\prime}}(x, x^\infty;\omega) - {r}_{x^\infty; \varepsilon^{\prime}}(x^{\, \infty}, x^{\, \infty};\omega) \,| \leq
		C(\omega)\,\|\,x-x^{\, \infty}\, \|_2 \leq C(\omega) \delta.
		\end{aligned}
		\end{equation}
		Then $\calF$ is contained in the bracket $\left[-\delta\,C(\omega), \delta\, C(\omega) \right]$ and $\delta\, C(\omega)$ is the envelope function
		of $\calF$.  Below we establish the upper bound for $\calN_{[\;]}\left(\sigma, \calF, L_2(\mathbb{P})\right)$, i.e., the bracketing number of $\calF$.
		
		\gap
		
		For any $x \in \mathbb{B}_{\delta}(x^\infty)$, the bracketing number of $\sigma$-brackets to cover this compact set is of order $\left(\frac{\delta}{\sigma}\right)^{p}$.
		Denote this set of brackets as $\calG$. Then there exists a bracket $\left[x_1, x_2\right] \in \calG$ satisfying
		$ \|x_1-x_2\|_2 \leq \sigma$ such that $x_1 \leq x \leq x_2$ (pointwise comparison). Based on \eqref{lip of difference}, we further have
		\[
		\begin{aligned}	 	
		- C(\omega) \|x_1-x_2\|_2 \leq & r_{x^{\, \infty}; \varepsilon^{\prime}}(x, x^{\, \infty};\omega) -
		r_{x^\infty; \varepsilon^{\prime}}(x^{\, \infty}, x^{\, \infty};\omega) \triangleq t(x, x^{\, \infty}; \omega) 	 	\leq  C(\omega)\|x_1-x_2\|_2.
		\end{aligned}
		\]
		This means that any $t(x, x^{\, \infty}; \omega) \in \calF$ can be covered by a bracket
		$\left[-C(\omega)\|x_1-x_2\|_2, \, C(\omega)\|x_1-x_2\|_2\right]$ of $L_2(\IP)$-size of $2\sigma ||C(\omega)||_2$. Since $x$ can be arbitrarily chosen,
		this implies that there exists a constant $k$ such that
		$$
		\calN_{[\;]}\left(2\sigma\|C(\omega)\|_2, \calF, L_2(\mathbb{P})\right) \leq k \left(\frac{\delta}{\sigma}\right)^{p}, \epc
		\mbox{for every $0 \leq \sigma \leq \frac{\delta}{2}$}.
		$$
When $\sigma > \displaystyle{
\frac{\delta}{2}
}$, the left-hand side in the above inequality is $1$. It then follows from Lemma \ref{lemma: entropy bound} that
\begin{equation*}
\begin{array}{l}
\E\left[\displaystyle\sup_{\|x - x^{\, \infty}\|\leq \delta}\sqrt{N}\left| \,{R}_{N; x^\infty;\varepsilon^{\prime}}( x^{\, \infty}, x^{\, \infty}) -
R_{x^\infty;\varepsilon^{\prime}}( x^{\, \infty}, x^{\, \infty}) -{R}_{N;x^\infty;\varepsilon^{\prime}}(x, x^{\, \infty}) +
R_{x^\infty;\varepsilon^{\prime}}(x, x^{\, \infty}) \,\right| \, \right] \\[0.2in]
\leq \, K_1\int_{0}^{\delta\, \|C(\omega) \|_2}\sqrt{\log\calN_{[\;]}\left(\sigma, \calF, L_2(\mathbb{P})\right)}\,\mbox{d}\sigma
= 2K_1\|C(\omega) \|_2\int_{0}^{\delta/2\, }\sqrt{\log\calN_{[\;]}\left(2\sigma\|C(\omega) \|_2, \calF, L_2(\mathbb{P})\right)}\,\mbox{d}\sigma
\\ [0.2in]
\leq \, K_2 \int_{0}^{\delta/2} \sqrt{\log \left(\frac{k\delta}{\sigma}\right)}\,\mbox{d}\sigma \, \leq \, K \, \delta
\end{array}
\end{equation*}
for some constants $K_1, K_2$ and $K$.
\end{proof}
	
By combining Propositions \ref{pr:strong local} and \ref{entropy for lip}, we obtain our final theorem for the convergence rate of $x^N$ to $x^{\infty}$.
	
	\begin{theorem}\rm \label{thm: final theorem for rate}
		If Assumptions in Propositions \ref{pr:strong local} and \ref{entropy for lip} hold,
		then $\|x^N - x^\infty \|_2 = \mbox{O}_{\IP_{\infty}}(\frac{1}{\sqrt{N}})$.
	\end{theorem}
	
	\begin{proof}
		By Propositions \ref{pr:strong local} and \ref{entropy for lip}, we know that
		Assumption \ref{assu:rate}~(b2) holds with $q= 2$ and Assumption (b2$^{\, \prime}$) holds with $\Phi_N(\delta) = \delta$.
		In oder to make Assumption \ref{assu:rate}~(b3) hold, it is suffice to find a sequence $\rho_N$ such that $\rho_N^2 \rho_N^{-1} \leq \sqrt{N}$.
		It is clear that $\rho_N$ can be chosen as $\sqrt{N}$. Therefore we obtain our conclusion based on Corollary \ref{coro: b3'}.
	\end{proof}
	
	
\section{Application: Noisy Amplitude-based Phase Retrieval Problem}
	
In this section, we use the nonconvex nonsmooth phase retrieval problem as an example
to illustrate that the C-stationary points and d-stationary points are
distinguishable even for the population risk minimization problems.  More importantly,
we can apply our established theory in the previous sections to this problem to demonstrate that every
computed d-stationary point converges to a global minimizer of the population problem at the rate of $\displaystyle \frac{1}{\sqrt{N}}$.
	
\gap
	
Phase retrieval, as described in the growing literature such as
\cite{candes2015phase,Shechtman15}, is a topical problem whose aim is to
recover a nonzero signal $\bar{x}\in \mathbb{R}^p$ from phaseless measurements.
We consider
\[
z_n \, = \, | \, \bar{x}^\top \xi^n \, | + \varepsilon_n,
\]
where $\left\{\varepsilon_n\right\}_{i =n}^N$ are independent and identically
distributed samples of a random error $\wt{\varepsilon}$ that has mean $0$ and variance $\sigma^2$.
We assume $\varepsilon_n$ is independent of $\xi_n$, for $n= 1, \cdots, N$.  In practice, we can obtain the
estimation of $\bar{x}$ by solving the following amplitude-based empirical minimization problem:
\begin{equation} \label{eq:phase empirical}
\displaystyle{
\operatornamewithlimits{\mbox{minimize}}_{x \in X}
} \quad \displaystyle{
\frac{1}{N}
} \, \displaystyle{
\sum_{n = 1}^{N}
} \, ( \, z_n - |x^{\top}\xi^n| \, )^2, 
\end{equation}
which corresponds to the population problem
\begin{equation} \label{eq:phase population}
\displaystyle{
\operatornamewithlimits{\mbox{minimize}}_{x \in X}
} \quad \calM(x) \, = \, \E_{\wt{\omega}} \, \left[ \, \wt{\zbld} - | \, x^{\top}\wt{\xi} \, |  \,\right]^2, 
\end{equation}
where $\wt{\zbld} = |\bar{x}^{\top}\wt{\xi}| + \wt{\varepsilon}$.  In this analysis,
we assume  $\wt{\xi}=\displaystyle\frac{\widetilde{\zeta}}{\| \widetilde{\zeta} \|_2}$ and $\widetilde{\zeta}$ follows the
standard $p$-dimensional multivariate Gaussian distribution.  In addition,
$\E_{\wt{\varepsilon}}\,[\,\wt{\varepsilon}\,] = 0$ and $\text{Var}_{\wt{\varepsilon}}\, [\,\wt{\varepsilon}\,] = \sigma^2$.
We further assume that $X$ is a convex compact set strictly containing $\mathbb{B}_{\|\bar{x}\|}(0_p)$.
The two problems (\ref{eq:phase population}) and (\ref{eq:phase empirical}) are special cases of the piecewise affine regression problem.

\gap
	
Before proceeding to the analysis of the problem (\ref{eq:phase population}), we need to say a few words about the set $X$ which was assumed
to be a convex compact set in our preceding analysis.  Such boundedness plays an important role in the previous analysis and ensures all points of
interest, that is, the stationary solutions of the population and empirical problems, are bounded.  In turn, the latter boundedness facilitates
the analysis, enabling us to bypass the technical issues associated with unboundedness and focus on the statistical analysis.
The boundedness of $X$ is unfortunately
inconsistent with the normal setting of the phase retrieval problems which has $X$ equal to the entire space, i.e., these problems are unconstrained.
In order to reconcile the gap between the common (unconstrained) setting of the problems and the constrained setting of the analysis, we assume
throughout the analysis below that the set $X$ is a compact ball centered at the origin with a radius sufficiently large so that $X$ contains
in its interior all the stationary points of
(\ref{eq:phase population}) given in Proposition~\ref{pr:phase stationary} and of the empirical problems (\ref{eq:phase empirical}) for all $N$.
Although a deeper analysis may allow us to show that such a precautious setting is unnecessary, we will work with this simplifying assumption
throughout the following analysis to avoid the technical complications of unboundedness and the possible existence of stationary solutions
lying on the boundary of $X$.

\gap

Another remark to be made about the problems (\ref{eq:phase empirical}) and (\ref{eq:phase population})
is that these two problems here are different
from the least-square formulation of solving quadratic equations and variations of such a formulation.  Specifically, the objective function
of the optimization formulation of such equations is
$\E_{\wt{\omega}}\left[ \, ( \, \bar{x}^{\top}\wt{\xi} \, )^2- ( \, x^{\top}\wt{\xi} \, )^2  \,\right]^2$; see e.g., the two references cited above.
The recent references \cite{duchi2018solving,Dima18} employ the objective $\E_{\wt{\omega}}\left| \, ( \, \bar{x}^{\top}\wt{\xi} \, )^2 - ( \, x^{\top}\wt{\xi} \, )^2  \,\right|$
which is also different from ours.  Nevertheless, the references such as \cite{wang2018solving,ma2019optimization} has used the same formulation as ours in studying
the phase problem but the results of its analysis are not as sharp as ours.
One major advantage of the piecewise affine objective
$\wt{\zbld} - \left| x^{\top}\wt{\xi} \,\right|$ employed in our formulations
(\ref{eq:phase population}) and (\ref{eq:phase empirical}) is that the resulting
objective in the empirical problem (\ref{eq:phase empirical}) is the composite of
a convex quadratic function with a piecewise affine function, thus is a
{\sl piecewise linear-quadratic} (PLQ) function in $x$.  This is in contrast to	
$\displaystyle{		
\sum_{n = 1}^{N}
} \, \left| \, z_n^2 - ( \, x^{\top}\xi^n \, )^2 \, \right|$, which is a
{\sl piecewise quadratic} (as opposed to piecewise linear-quadratic) function
in $x$, and also to the objective $\displaystyle{
\sum_{n = 1}^{N}
} \, \left( \, z_n^2 - ( \, x^{\top}\xi^n \, )^2 \, \right)^2$, which is a
quartic (multivariate) polynomial, thus smooth, function of $x$.  See the
reference \cite{CuiPangHongChang18} for a comprehensive study of a (finite-dimensional)
PLQ optimization problem; in particular, many favorable properties that are not shared
by objectives of other kinds, including the piecewise quadratic and non-quadratic ones
are presented therein.  
Our	contributions to the problems (\ref{eq:phase population}) and (\ref{eq:phase empirical}) are
summarized below:
	
\gap
	
{\bf (i)} The origin $x = 0$ is a Clarke stationary solution of the empirical problem
(\ref{eq:phase empirical}) for every $N$ and also
a Clarke stationary solution of the population problem (\ref{eq:phase population});
yet $x = 0$ is not a directional stationary solution, thus not	
a local minimizer, of either problem; (note: the origin is a stationary solution
of the other two objective functions, which is excluded by
our PLQ objective); these results are also valid when $\wt{\zeta}$ is
not normalized.  Moreover, we show that all the stationary solutions of the population problem (\ref{eq:phase population})
except $\pm \bar{x}$ are saddle points.  We further demonstrate that $\calM(x)$ is locally strong convex near $\pm \bar{x}$.
All these results are seemingly new in the existing literature.
	
	
	\gap
	
{\bf (ii)} By applying our developed theory, we demonstrate that every defined
$\varepsilon$-strong d-stationary point of the empirical problem
(\ref{eq:phase empirical}) converges to one of true signals $\pm \bar{x}$ at the rate of
$\frac{1}{\sqrt{N}}$. Compared with existing literature such as \cite{ma2019optimization}, which rely heavily on a particular algorithm with spectral initialization,
to the best of our knowledge, this is the first theoretical analysis that provides the statistical
guarantee of the global convergence to true signals for the amplitude-based phase retrieval problem \eqref{eq:phase population}.
	
\gap
	
{\bf (iii)} We consider a normalized random variable $\wt{\xi}$ so that
the resulting variable $\wt{\xi}$ is uniformly bounded; this boundedness is required by our
asymptotic analysis.  Presently, it is not clear if a rigorous
asymptotic theory can be developed for a coupled nonconvex nondiffrentiable problem
such as the phase problem here without requiring boundedness of the underlying
randomness. 	
	
\gap
	
{\bf (iv)} An algorithm described in \cite{CuiPangSen18} can be applied to
numerically verify the obtained results of statistical consistency of the
d-stationary solutions of the empirical problems and shed lights on the convergence	
of such solutions and their objective values for this phase retrieval problem.
Here we point out that the algorithm in the cited reference does not require
any special treatments or assumptions on the initialization, which are needed for
most existing literature of phase retrieval problems such as \cite{candes2015phase} or \cite{ma2019optimization}.
While the exception is \cite{chen2019gradient} for the quartic-based phase retrieval problem, they still require
the initial point of the proposed algorithm to satisfy certain conditions with high probability to demonstrate its global
convergence, see~\cite[Theorem 2 \& 3]{chen2019gradient}.
Thus combining our established theory and the corresponding algorithm in \cite{CuiPangSen18},
we have fill the gap between practical computation and theoretical analysis of the amplitude-based phase retrieval problem
with the above choice of the random variable $\wt{\xi}$.
	
	\gap
	
	Before the derivation, we point out two facts about $\wt{\xi}$ and refer to
	\cite[Chapter~4]{bryc2012normal} for more properties of this random vector.
	\begin{itemize}
		\item[(F1)] The random vector $\wt{\xi}$ follows a uniform distribution on the unit
		sphere in $\mathbb{R}^p$; $\wt{\xi}$ and $\| \widetilde{\zeta} \|_2$ are
		independent \cite[Theorem 4.1.2]{bryc2012normal}.
		\item[(F2)]  $\wt{\xi}$ is invariant over any orthogonal transformation.
	\end{itemize}
	With $\wt{\xi}$ as stated, we have
	\[
	\begin{array}{lll}
	\calM(x) = \E_{\wt{\omega}} \, \left[ \, \wt{\zbld} - | \, x^{\top}\wt{\xi} \, |  \,\right]^2 & = &
	\E_{\wt{\omega}} \, \left[ \, |\, \bar{x}^{\top}\wt{\xi} \, | -
	| \, x^{\top}\wt{\xi} \, |  \,\right]^2 + \sigma^2 \\ [0.15in]
	& = & \E_{\wt{\xi}} \, \left[ \, \wt{\xi}^{\, \top}\bar{x}\bar{x}^{\top}\wt{\xi}
	\,\right] +	\E_{\wt{\xi}} \left[ \, \wt{\xi}^{\, \top}xx^{\top}\wt{\xi}  \,\right]
	-2\E_{\xi} \, \left[ \ \left| \ \wt{\xi}^{\, \top}\bar{x}x^{\top}\wt{\xi} \ \right|
	\,\right] + \sigma^2 \\ [0.15in]	
	& = & \E_{\wt{\xi}} \, \left[ \, \wt{\xi}^{\, \top}(\bar{x}\bar{x}^{\top}+xx^{\top})
	\wt{\xi} \,\right] - \E_{\wt{\xi}}\left[ \, \left| \, \wt{\xi}^{\, \top}(\bar{x}x^{\top}
	+ x\bar{x}^{\top})\wt{\xi} \, \right| \,\right] + \sigma^2.	
	\end{array}
	\]	
	Based on the first equality, it is clear that $\pm \bar{x}$ are global minimizers of $\calM(x)$. Define the matrices $M_1(x) \triangleq ( \bar{x}\bar{x}^{\top} + xx^{\top} )$ and	
	$M_2(x) \triangleq \bar{x}x^{\top} + x\bar{x}^{\top}$.  Clearly both matrices $M_1(x)$
	and $M_2(x)$ are of rank at most 2.  Let $\lambda_{\pm}(M_i(x))$ together with
	$p - 2$ zeros be the eigenvalues of the matrix $M_i(x)$ for $i = 1, 2$.	
	By some linear algebraic manipulations, we can show
	\[	
	\lambda_{\pm}(M_1(x)) \, = \, \frac{||x||^2_2 + ||\bar{x}||_2^2 \pm
		\displaystyle{\sqrt{(||x||_2^2-||\bar{x}||_2^2)^2 + 4(\bar{x}^{\top}x)^2}}}{2}
	\]
	and
	\[
	\lambda_{\pm}(M_2(x)) \, = \, \bar{x}^{\top}x \pm \| x \|_2 \| \bar{x} \|_2.
	\]
	By using eigenvalue decomposition and (F2), we derive
	\begin{equation} \label{eq:expression phase}
	\calM(x) = \E_{\wt{v}} \, \left[ \, \lambda_+(M_1(x))\wt{v}_1^{\, 2} +
	\lambda_-(M_1(x)) \wt{v}_2^{\, 2} \, \right] -
	\E_{\wt{v}} \, \left[ \, \left| \, \lambda_+(M_2(x))\wt{v}_3^{\, 2} +
	\lambda_-(M_2(x))\wt{v}_4^{\, 2} \, \right| \,\right] + \sigma^2,
	\end{equation}
	where $\wt{v}_1$ and $\wt{v}_2$ being two coordinates of a uniform distribution on the
	unit sphere, and similarly for $\wt{v}_3$ and $\wt{v}_4$. These random variables
	are not necessarily
	independent. Denote $w_1$ and $w_2$ as the corresponding eigenvectors of
	$\lambda_{\pm}(M_1(x))$ and $w_3$ and $w_4$ as the corresponding eigenvectors for
	$\lambda_{\pm}(M_2(x))$, respectively.  So $\wt{v}_i = w_i^\top \widetilde\xi =
	\displaystyle\frac{w_i^\top \widetilde{\zeta}}{\|\,\widetilde{\zeta}\, \|_2}$,
	for $i = 1, \cdots, 4$.  Then by independence between $\widetilde\xi$ and
	$\| \, \widetilde{\zeta}\,\|_2$, we can show that 	
	\[
	\E_{\wt{v_i}}\,[\, \wt{v}^{\, 2}_i \,] = \frac{\E_{\wt{\zeta}} \,
		[ \,(\,w_i^\top \widetilde{\zeta} \,)^2 \,]}{\E_{\wt{\zeta}} \,
		[ \, \| \, \widetilde{\zeta} \,\|^2_2 \,]} = \frac{1}{p}.
	\]
	Similarly, we can also show
	\[
	\begin{array}{l}
	\E_{\wt{v}} \, \left[ \, \left| \, \lambda_+(M_2(x))\wt{v}_3^{\, 2} +
	\lambda_-(M_2(x))\wt{v}_4^{\, 2} \, \right| \,\right] \\ [0.1in]
	= \, \E_{(\widetilde\xi, \,\wt{\zeta})} \, \left[ \displaystyle{
		\frac{\, \left| \, \lambda_+(M_2(x))(w_3^\top \wt{\zeta} \,)^{\, 2} +
			\lambda_-(M_2(x))(w_4^\top \wt{\zeta}\,)^{\, 2} \, \right|}{
			\left\| \, \widetilde{\zeta} \, \right\|^2_2}
	} \,\right] \\ [0.35in]
	= \, \displaystyle{
		\frac{\E_{\widetilde\xi} \, \left[ \, \left| \,
			\lambda_+(M_2(x))(w_3^\top \wt{\zeta} \,)^{\, 2} +
			\lambda_-(M_2(x))(w_4^\top \wt{\zeta}\,)^{\, 2} \, \right| \,\right]}
		{\E_{\wt{\zeta}} \left[ \, \| \,\widetilde{\zeta} \,\|^2_2 \, \right]}
	} \\ [0.25in]
	= \, \displaystyle{
		\frac{1}{p}
	} \, \left\{ \, \E_{\wt{u}} \, \left[ \, \left| \, \lambda_+(M_2(x))\wt{u}_3^{\, 2} +
	\lambda_-(M_2(x))\wt{u}_4^{\, 2} \, \right| \, \right] \, \right\},
	\end{array}
	\]	
	where $\wt{u}_3$ and $\wt{u}_4$ are mutually independent Gaussian random variables.
	Based on this, we can further simplify $\calM(x)$ as
	\begin{equation} \label{eq:expression phase}
	\begin{array}{lll}
	\calM(x) & = & \displaystyle{
		\frac{1}{p}
	} \left[ \, \lambda_+(M_1(x)) + \lambda_-(M_1(x))\,\right] -
	\E_{\wt{u}}\left[ \, \left| \, \lambda_+(M_2(x)) \wt{u}_3^{\, 2} +
	\lambda_-(M_2(x))\wt{u}_4^{\, 2} \, \right| \, \right] + \sigma^2
	\\ [0.15in]
	& = & \displaystyle{
		\frac{1}{p}
	} \, \left( \, \| \, x \, \|_2^2 + \| \, \bar{x} \, \|_2^2 \, \right)
	- \displaystyle{
		\frac{1}{p}
	} \, \E_{\wt{u}}\left[ \, \left| \, \bar{x}^{\top}x \,
	\left( \, \wt{u}_3^{\, 2} + \wt{u}_4^{\, 2} \, \right) +
	\| x \|_2 \| \bar{x} \|_2 \left( \, \wt{u}_3^{\, 2} - \wt{u}_4^{\, 2} \, \right) \,
	\right| \, \right] + \sigma^2.
	\end{array}
	\end{equation}
	When $x = \pm\bar{x}$, we have $M_1(x) = 2\bar{x}\bar{x}^{\top}$ and
	$M_2(x) =\pm 2\bar{x}\bar{x}^{\top}$, thus $\calM(\pm\bar{x}) = \sigma^2$. We next demonstrate that $x = 0$ is a Clarke
	stationary point of both ${\cal M}(x)$ and ${\cal M}_N(x)$.
	
	\begin{proposition} \rm
		Let $\widetilde{\zeta}$ follow the standard $p$-dimensional multivariate Gaussian
		distribution and
		$\wt{\xi} = \displaystyle{
			\frac{\widetilde{\zeta}}{\| \widetilde{\zeta} \|_2}
		}$.
		Then $x = 0$ is a Clarke stationary point of both ${\cal M}$ and ${\cal M}_N$.
	\end{proposition}
	
	\begin{proof} Since $0$ belongs to the interior of $X$, we can first verify
		that $\nabla \lambda_{+}(M_1(0)) + \nabla \lambda_{-}(M_1(0)) = 0$.
		Hence to show that $x = 0$ is a Clarke stationary point of ${\cal M}$, it suffices
		to show
		\begin{equation} \label{eq:zero critical}
		0 \, \in \, \left\{ \, \partial_C \, \E_{\wt{u}} \left[ \, \left| \,
		\underbrace{( \, \bar{x}^{\top}x + \| x \|_2 \, \| \bar{x} \|_2 \, )\wt{u}_3^{\, 2} +
			( \, \bar{x}^{\top}x - ||x||_2 \, ||\bar{x}||_2 \, )\wt{u}_4^{\, 2}}_{\mbox{
				denoted $e(x;\wt{u})$}} \, \right| \, \right] \, \right\}_{x=0} .
		\end{equation}
		Let $\wh{\cal M}(x;u) \triangleq | \, e(x;u) \, |$.  To evaluate
		$\partial_C \, {\cal M}(0)$, we employ the expression (\ref{eq:Clarke subdiff})
		by taking $\wh{x}^{\, k} \triangleq \displaystyle{
			\frac{\wh{x}}{k}
		}$, where $\wh{x}$ is a fixed nonzero vector satisfying $\bar{x}^{\top} \wh{x} = 0$.
		We have $\wh{\cal M}(\pm \wh{x}^{\, k};u) = \displaystyle{
			\frac{1}{k}
		} \, \| \, \wh{x} \, \|_2 \, \| \, \bar{x} \, \|_2 \, \left( \, u_3^2 - u_4^2 \, \right)$
		which is not equal to zero almost surely.  Hence,
		\[
		\nabla_x \wh{\cal M}(\pm \wh{x}^{\, k};u) \, = \, \mbox{sgn}\left( \, u_3^2 - u_4^2 \,
		\right) \, \left[ \, \bar{x} \, \left( \,  u_3^2 + u_4^2 \, \right) \pm \,
		\displaystyle{
			\frac{\| \, \bar{x} \, \|_2}{\| \, \wh{x} \, \|_2}
		} \, \wh{x} \, \left( \, u_3^2 - u_4^2 \, \right) \, \right],
		\]
		which is independent of $k$.  Consequently,
		\[ \begin{array}{lll}
		\nabla_x \E_{\wt{u}}\left\{ \, \displaystyle{
			\frac{1}{2}
		} \, \left[ \, \nabla_x \wh{\cal M}(\wh{x}^{\, k};\wt{u}) +
		\nabla_x \wh{\cal M}(-\wh{x}^{\, k};\wt{u}) \, \right] \, \right\}
		& = & \displaystyle{
			\frac{1}{2}
		} \, \E_{\wt{u}} \left[ \, \nabla_x \wh{\cal M}(\wh{x}^{\, k};\wt{u}) +
		\nabla_x \wh{\cal M}(-\wh{x}^{\, k};\wt{u}) \, \right] \\ [0.15in]
		& = &  \, \bar{x} \, \E_{\wt{u}} \left[ \, \mbox{sgn}\left( \, \wt{u}_3^{\, 2} -
		\wt{u}_4^{\, 2} \, \right) \, \left( \, \wt{u}_3^{\, 2} + \wt{u}_4^{\, 2} \, \right)
		\, \right] \, = \, 0,
		\end{array} \]
		where the last equality holds because the distribution of
		$\mbox{sgn}(\wt{u}_3^{\, 2} - \wt{u}_4^{\, 2})( \wt{u}_3^{\, 2} + \wt{u}_4^{\, 2})$
		is symmetric.  This is enough to establish (\ref{eq:zero critical}).  Thus $x = 0$ is
		a Clarke stationary point of the population objective ${\cal M}$ for the phase problem
		(\ref{eq:phase population}).  Omitting the details, we can similarly show that $x = 0$
		is a Clarke stationary point of the empirical objective ${\cal M}_N$ by verifying
		\[
		0 \, \in \, \partial_C \left\{ \, \displaystyle{
			\frac{1}{N}
		} \, \displaystyle{
			\sum_{n = 1}^{N}
		} \, ( \, z_n - |x^{\top}\xi^n| \, )^2 \,  \right\}_{x=0}
		\]
		using the same sequence of points $\left\{ \pm x^k \right\}$ as above.
	\end{proof}	
	
	Next, we show that $x = 0$ is not a d-stationary point of ${\cal M}$.
	Since $\wh{\cal M}(\bullet;u)$ is positively homogeneous, it follows that
	\[ \begin{array}{lll}
	\wh{\cal M}(\bullet;u)^{\prime}(0;v) \, = \, \wh{\cal M}(v;u) & = &
	\left| \, ( \, \bar{x}^{\top}v + \| v \|_2 \, \| \bar{x} \|_2 \, )u_3^2 +
	( \, \bar{x}^{\top}v - ||v||_2 \, ||\bar{x}||_2 \, )u_4^2 \, \right|,
	\epc \forall \, v \\ [0.2in]
	& = &  \left| \, \bar{x}^{\top}v \, ( \, u_3^2 + u_4^2 \, ) +
	||v||_2 \, ||\bar{x}||_2 \, \left( \, u_3^2 - u_4^2 \, \right) \, \right| \\ [0.2in]
	& = & 2 \| \, \bar{x} \, \|_2^2 \, u_3^2 \epc \mbox{for $v = \bar{x} \in X$}.
	\end{array} \]
	Hence
	\[
	{\cal M}^{\, \prime}(0;\bar{x}) \, = \, - \frac{1}{p}\,
	\E_{\wt{u}}\left[ \, \wh{\cal M}(\bullet;\wt{u} )^{\prime}(0;v) \, \right]
	\, = \, -\frac{2}{p} \, \| \, \bar{x} \, \|_2^2 \,
	\E_{\wt{u}}\left[ \, \wt{u}^{\, 2} \, \right] \, < \, 0.
	\]
	
	We next compute the full set of d-stationary points of the population problem
	(\ref{eq:phase population}).  For a given nonzero vector $x$,
	since $e(x;\bullet) \neq 0$ almost surely,
	we can derive from the expression (\ref{eq:expression phase}),
	\[
	\nabla \calM(x) \, = \, \frac{2}{p} \, \left[ \, 1 - \thalf \,
	\E_{\wt{u}}\left\{ \, \mbox{sgn}(e(x;\wt{u})) \, \displaystyle{
		\frac{\| \, \bar{x} \, \|_2}{\| \, x \, \|_2}
	} \, ( \, \wt{u}_3^{\, 2} - \wt{u}_4^{\, 2} \, ) \, \right\} \, \right] \, x -
	\frac{1}{p}\E_{\wt{u}}\left\{ \, \mbox{sgn}(e(x;\wt{u})) \, ( \, \wt{u}_3^2 +
	\wt{u}_4^2 \, ) \, \right\} \, \bar{x}.
	\]
	Based on this expression, we can establish the following result.
	
\begin{proposition} \label{pr:phase stationary} \rm
Let $\widetilde{\zeta}$ follow the standard $p$-dimensional multivariate Gaussian
distribution and
$\wt{\xi} = \displaystyle\frac{\widetilde{\zeta}}{\| \widetilde{\zeta} \|_2}$.
Then the stationary solutions of (\ref{eq:phase population}) either are $\pm \bar{x}$ or belong
to		
\[
\left\{ x \, \mid \, \bar{x}^{\top} x = 0 \mbox{ and }
\| \, x \, \|_2 = \displaystyle{
\frac{2}{\pi}
} \, \| \, \bar{x} \, \|_2 \right\}
\]
Moreover, there is only one suboptimal stationary value which is equal to
$\displaystyle\frac{1}{p} \,\left[ 1 - \displaystyle{
\frac{4}{\pi^2}
} \, \right] \| \, \bar{x} \, \|_2^2$.
\end{proposition}
	
	\begin{proof}  Since there is no stationary solution on the boundary of $X$, we can
		compute all stationary solutions by letting $x \neq 0$ satisfy $\nabla \calM(x) = 0$.
		Note that we have already showed that $0$ is not a d-stationary solution of $\calM$.
		If $\E_{\wt{u}}\left\{ \, \mbox{sgn}(e(x;\wt{u})) \, ( \, \wt{u}_3^{\, 2} +
		\wt{u}_4^{\, 2} \, ) \, \right\} \neq 0$, then for some nonzero scalar $\eta$,
		dependent on $x$, we have $x = \eta \, \bar{x}$.  Thus,
		\[ \begin{array}{lll}
		e(x,u) & = & \left[ \, \eta \, \left( \,  u_3^2 + u_4^2 \, \right) + | \, \eta \, |
		\, \left( \,  u_3^2 - u_4^2 \, \right) \, \right] \, \| \, \bar{x} \, \|_2^2
		\\ [0.1in]
		& = & \left\{ \begin{array}{ll}
		2 \, \eta \, u_3^2 \, \| \, \bar{x} \, \|_2^2 & \mbox{if $\eta > 0$} \\ [5pt]
		2 \, \eta \, u_4^2 \, \| \, \bar{x} \, \|_2^2 & \mbox{if $\eta < 0$}
		\end{array} \right. ,
		\end{array} \]
		which implies $\mbox{sgn}(e(x;u)) = \mbox{sgn}( \eta )$.  Hence, we have
		\[ \begin{array}{lll}
		0 \, = \, \nabla \calM(x) & = & \displaystyle\frac{1}{p} \,\bar{x} \, \left[ \,
		\left\{ \, 2 \, \eta - \mbox{sgn}(\eta) \, \displaystyle{
			\frac{\eta}{| \, \eta \, |}
		} \, \E_{\wt{u}}\left[ \, \wt{u}_3^{\, 2} - \wt{u}_4^{\, 2} \, \right] \, \right\}
		- \mbox{sgn}(\eta) \, \E_{\wt{u}}\left[ \, \wt{u}_3^{\, 2} + \wt{u}_4^{\, 2} \, \right]
		\, \right] \\ [0.2in]
		& = & \left\{ \begin{array}{ll}
		\displaystyle\frac{2}{p} \, \, \bar{x} \, \left[ \, \eta - \E_{\wt{u}} \left[ \,
		\wt{u}_3^{\, 2} \, \right] \, \right] & \mbox{if $\eta > 0$} \\ [0.2in]
		\displaystyle\frac{2}{p} \, \, \bar{x} \, \left[ \, \eta + \E_{\wt{u}} \left[ \,
		\wt{u}_4^{\, 2} \, \right] \, \right] & \mbox{if $\eta < 0$}
		\end{array} \right.,
		\end{array} \]
		which implies $\eta = \pm 1$.  Consequently, we have proved that if
		$\E_{\wt{u}}\left\{ \, \mbox{sgn}(e(x;\wt{u})) \, ( \, \wt{u}_3^{\, 2} +
		\wt{u}_4^{\, 2} \, ) \, \right\} \neq 0$,
		then $x = \pm \bar{x}$.  Suppose that
		$\E_{\wt{u}}\left\{ \, \mbox{sgn}(e(x;\wt{u})) \, ( \, \wt{u}_3^{\, 2} +
		\wt{u}_4^{\, 2} \, ) \, \right\} = 0$ and also $x$ is not proportional to $\pm \bar{x}$.
		Write $e(x;u) = z_+ \, u_3^2 - z_- \, u_4^2$, where both
		$z_{\pm} \triangleq \| \, x \, \|_2 \| \, \bar{x} \, \|_2 \pm \bar{x}^{\top} x$
		are nonnegative scalars. Suppose $\bar{x}^\top x >0$,  then $z_+ > z_- > 0$. By letting
		$\II( \bullet )$ be the indicator of a (random) event, we deduce
		\[
		\begin{array}{lll}
		0 & = & \E_{\wt{u}}\left\{ \, \mbox{sgn}(e(x;\wt{u})) \, ( \, \wt{u}_3^{\, 2} +
		\wt{u}_4^{\, 2} \, ) \, \right\} \\ [0.1in]
		& = & \E_{\wt{u}}\left\{ \, \II\left( \, z_+ \, \wt{u}_3^{\, 2} - z_- \, \wt{u}_4^{\, 2}
		\, > \, 0 \, \right) \, ( \, \wt{u}_3^{\, 2} + \wt{u}_4^{\, 2} \, ) \, \right\}
		- \E_{\wt{u}}\left\{ \, \II\left( \, z_+ \, \wt{u}_3^{\, 2} - z_- \, \wt{u}_4^{\, 2} \,
		< \, 0 \, \right) \, ( \, \wt{u}_3^{\, 2} + \wt{u}_4^{\, 2} \, ) \, \right\}
		\\ [0.1in]
		& = & \E_{\wt{u}}\left\{ \, \II\left( \, z_+ \, \wt{u}_3^{\, 2} - z_- \, \wt{u}_4^{\, 2}
		\, > \, 0 \, \right) \, ( \, \wt{u}_3^{\, 2} + \wt{u}_4^{\, 2} \, ) \, \right\}
		- \E_{\wt{u}}\left\{ \, \II\left( \, z_- \, \wt{u}_3^{\, 2} - z_+ \, \wt{u}_4^{\, 2}
		\, < \, 0 \, \right) \, ( \, \wt{u}_3^{\, 2} + \wt{u}_4^{\, 2} \, ) \, \right\}
		\\ [0.1in]
		& & + \, \E_{\wt{u}}\left\{ \, \II\left( \, \displaystyle{
			\frac{z_-}{z_+}
		} \, \wt{u}_3^{\, 2} \, \leq \, \wt{u}_4^{\, 2} \, \leq \, \displaystyle{
			\frac{z_+}{z_-}
		} \, u_3^2 \, \right) \, ( \, \wt{u}_3^{\, 2} + \wt{u}_4^{\, 2} \, ) \, \right\}
		\\ [0.1in]
		& = &  \E_{\, \wt{u}}\left\{ \, \II\left( \, \displaystyle{
			\frac{z_-}{z_+}
		} \, \wt{u}_3^{\, 2} \, \leq \, \wt{u}_4^{\, 2} \, \leq \, \displaystyle{
			\frac{z_+}{z_-}
		} \, u_3^2 \, \right) \, ( \, \wt{u}_3^{\, 2} + \wt{u}_4^{\, 2} \, ) \, \right\},
		\end{array} \]
		where the last equality holds because $\wt{u}_3$ and $\wt{u}_4$ are independent and
		have the same distribution.	Thus $ \II\left( \, \displaystyle{
			\frac{z_-}{z_+}
		} \, \wt{u}_3^{\, 2} \, \leq \, \wt{u}_4^{\, 2} \, \leq \, \displaystyle
		\frac{z_+}{z_-}
		\, \wt{u}_3^{\, 2} \, \right) \, = \, 0$ almost surely.  This implies $z_+ = z_-$,
		which is equivalent to $\bar{x}^\top x = 0$.  We thus get a contradiction.  Similarly,
		one can show that $\bar{x}^\top x < 0$ cannot hold.  Therefore, we get
		$\bar{x}^\top x = 0$.  Then
		$e(x;u) = \| \, \bar{x} \|_2 \| x \|_2 \left( \, u_3^2 - u_4^2 \, \right)$ and
		\[
		\begin{array}{lll}
		0 & = & \nabla \calM(x) \\ [0.1in]
		& = &  \displaystyle\frac{2}{p} \, \, \left[ \, 1 - \displaystyle\frac{1}{2} \,
		\E_{\wt{u}}\left\{ \, \mbox{sgn}( \wt{u}_3^{\, 2} - \wt{u}_4^{\, 2} ) \,
		\displaystyle{
			\frac{\| \, \bar{x} \, \|_2}{\| \, x \, \|_2}
		} \, ( \, \wt{u}_3^{\, 2} - \wt{u}_4^{\, 2} \, ) \, \right\} \, \right] \, x -
		\displaystyle\frac{1}{p}\,\underbrace{\E_{\wt{u}}\left\{ \, \mbox{sgn}( \wt{u}_3^{\, 2}
			- \wt{u}_4^{\, 2} ) \, ( \, \wt{u}_3^{\, 2} + \wt{u}_4^{\, 2} \, ) \, \right\}}_{
			\mbox{= 0 by symmetry}} \, \bar{x} \\ [0.1in]
		& = &  \displaystyle\frac{2}{p} \, \, \left[ \, 1 - \displaystyle\frac{1}{2} \,
		\displaystyle{
			\frac{\| \, \bar{x} \, \|_2}{\| \, x \, \|_2}
		} \, \E_{\wt{u}} \left[ \, | \, \wt{u}_3^{\, 2} - \wt{u}_4^{\, 2} \, | \, \right]
		\, \right] \, x \, = \, \frac{2}{p}\, \left[ \, 1 - \displaystyle{
			\frac{2}{\pi}
		} \, \displaystyle{
			\frac{\| \, \bar{x} \, \|_2}{\| \, x \, \|_2}
		} \, \right] \, x.
		\end{array} \]
		Thus $\| \, x \, \| =  \displaystyle{
			\frac{2}{\pi}
		} \, \| \, \bar{x} \, \|_2$ as desired.  The last assertion of the proposition
		follows readily by substituting the properties of a d-stationary point into the
		objective function $\calM(x)$ obtained in (\ref{eq:expression phase}).
	\end{proof}
	
	In what follows, we apply our established theory in the previous sections to this
	phase retrieval problem.  First we demonstrate that every suboptimal stationary
	solution of the problem \eqref{eq:phase population} is a saddle point, neither local
	minimizer or maximizer by the following proposition.
	Notice that based on  \cite[Lemma 5.6]{Dima18}, we can further write \eqref{eq:expression phase} as
				\begin{equation}\label{eq:phase population2}
				 \begin{array}{l}
		p \, \calM(x) \, = \,  \| \, x \, \|_2^2 + \| \, \bar{x} \, \|_2^2 + p \, \sigma^2 +
		2 \, \bar{x}^{\top} x\\ [0.1in]
		\epc - \, \displaystyle{
			\frac{4}{\pi}
		} \left[ \, 2x^\top \bar{x} \,
		\arctan{\underbrace{\sqrt{\frac{\| x \|_2 \| \bar{x} \|_2 + \bar{x}^{\top}x}{
						\| x \|_2 \| \bar{x} \|_2 - \bar{x}^{\top}x}}}_{\mbox{denoted $g(x)$}}} +
		\sqrt{\left(\bar{x}^{\top}x + \| x \|_2 \| \bar{x} \|_2\right)
			\left( \| x \|_2 \| \bar{x} \|_2 - \bar{x}^{\top}x \right)} \right].
		\end{array}
		\end{equation}

	\begin{proposition}\label{local is global for 11}\rm
		Let  $\widetilde{\zeta}$ follow the standard $p$-dimensional multivariate
		Gaussian distribution and
		$\xi = \displaystyle\frac{\widetilde{\zeta}}{\| \widetilde{\zeta} \|_2}$.
		Then any point in
		$\calD^{\, \prime} \,\triangleq \, \left\{ x \, \mid \, \bar{x}^{\top} x = 0 \mbox{ and }
		\| \, x \, \|_2 = \displaystyle{
			\frac{2}{\pi}
		} \, \| \, \bar{x} \, \|_2 \right\}$
		is a  saddle point of 	(\ref{eq:phase population}).
	\end{proposition}
	
	\begin{proof}
			Provided that $x$ is not zero and $x\neq \pm \bar{x}$, we can deduce from \eqref{eq:phase population2} that
		\begin{equation*}
		\begin{aligned}
		\nabla \calM(x) & = \frac{1}{p}\left( 2x + 2 \bar{x}- \frac{4}{\pi}\left[2\bar{x}
		\arctan{(g(x))} + \frac{2x^\top \bar{x}}{1 + g^2(x)} \nabla g(x) +
		\frac{\|\bar{x}\|^2 x - (x^\top \bar{x})\bar{x}}{\sqrt{\| x \|^2_2 \| \bar{x} \|^2_2
				- (x^\top \bar{x})^2}}\right]\right)
		\end{aligned}
		\end{equation*}
		and, letting $I_p$ denote the identity matrix of order $p$,
		\begin{equation*}
		\begin{array}{ll}
		p \, \nabla^2 \calM(x) \, = & 2 \, I_p -\displaystyle{
			\frac{4}{\pi}
		} \, \left[ \displaystyle{
			\frac{2\bar{x} \, \nabla g(x)^\top}{1 + g^2(x)}
		} + \displaystyle{
			\frac{(\|\bar{x} \|_2^2 \, I_p - \bar{x}\bar{x}^\top)\sqrt{
					\| x \|^2_2 \| \bar{x} \|^2_2 - (x^\top \bar{x})^2}}{\| x \|^2_2 \| \bar{x} \|^2_2
				- (x^\top \bar{x})^2}} \, \right] \\ [0.3in]
		& + \, \displaystyle{
			\frac{4}{\pi}
		} \, \left[ \, \displaystyle{
			\frac{\left[ \, \|\bar{x} \|_2^2 \, x - (\bar{x}^\top x) \bar{x})(\|\bar{x} \|_2^2 \, x
				- (\bar{x}^\top x) \, \bar{x} \, \right]^\top}{\left[ \, \| x \|^2_2 \| \bar{x} \|^2_2
				- (x^\top \bar{x})^2 \, \right]^\frac{3}{2}}} \, \right] \\ [0.4in]
		& - \, \displaystyle{
			\frac{4}{\pi}
		} \, \left[ \displaystyle{
			\frac{2\bar{x}\, (1+ g^2(x))\nabla g(x)^\top - 2 \bar{x}^\top x \,
				\nabla g^2(x) \nabla g(x)^\top }{( \, 1+g^2(x) \, )^2}
		} \, \right].
		\end{array}
		\end{equation*}
		Therefore, for any $x \in \mathcal{D}^{\, \prime}$, we have
		\begin{equation*}
		\begin{array}{lll}
		\nabla^2 \calM(x) & = & \displaystyle{
			\frac{1}{p}
		} \, \left( \, 2 \, I_p - 2I_p + \displaystyle{
			\frac{4}{\pi}
		} \, \left[ \, \displaystyle{
			\frac{\bar{x}\bar{x}^\top}{\| x \|_2 \| \bar{x} \|_2}
		} + \displaystyle{
			\frac{\| \bar{x}\|_2 x x^\top}{\|x \|_2^3}
		} \, \right] - \displaystyle{
			\frac{8}{\pi}
		} \, \bar{x} \, \nabla g(x)^\top \right) \\ [0.3in]
		& = & \displaystyle{
			\frac{1}{p}
		} \, \left( \, \displaystyle{
			\frac{4}{\pi}
		} \, \left[ \, \displaystyle{
			\frac{\bar{x}\bar{x}^\top}{\| x \|_2 \| \bar{x} \|_2} + \frac{\| \bar{x}\|_2 x x^\top}{
				\|x \|_2^3}
		} \, \right] - \displaystyle{
			\frac{8}{\pi}
		} \, \bar{x} \, \nabla g(x)^\top \right).
		\end{array}
		\end{equation*}
		By noting that $\bar{x}^\top \nabla  g(x) = \displaystyle\frac{\pi}{2}$ for any
		$x \in \calD^\prime$, the above equalities further yield
		\begin{equation*}
		\begin{aligned}
		\text{trace}\left( \nabla^2\, \calM(x)\right) & = \,
		\frac{1}{p}\left( 4 - \frac{8}{\pi} \nabla g(x)^\top\bar{x} \right) =
		\frac{1}{p}\left( 4 - 4 \right)= 0.
		\end{aligned}
		\end{equation*}
		It is easy to check that
		$\nabla \mathcal{M}(x)=\displaystyle\frac{2x}{p} \left(1-\frac{2\|\bar{x}\|_2}{
			\pi \|x\|_2}\right)$ for any $x$ orthogonal to $\bar{x}$. Thus,
		$\nabla\calM(x)$ is not constantly $0$ in the neighborhood of $x \in \calD^{\, \prime}$,
		which implies that there must exist a positive and a negative eigenvalues for the
		Hessian matrix $\nabla^2 \calM(x)$ for any $x\in \mathcal{D}^{\, \prime}$.
	\end{proof}
	
We remark that every d-stationary point of the empirical phase retrieval problem
\eqref{eq:phase empirical} is in fact its local minimizer
since the objective function is the composite of a convex function with a piecewise
linear function with a convex compact constraint \cite[Proposition 11]{CuiPangHongChang18}. Next, we will demonstrate
that every empirical $\varepsilon$-strong d-stationary point $x^N$ of phase retrieval problem \eqref{eq:phase empirical} converges
to $\calD_0 = \left\{\pm \bar{x}  \right\}$ at the rate of $\frac{1}{\sqrt{N}}$. As we know $\calD_0$ is the set of all global minimizers
of the problem \eqref{eq:phase population}. To show this, we need the following lemma.

\begin{lemma}\label{lm: strong local min}\rm
The population amplitude-based phase retrieval problem \eqref{eq:phase population} is locally strong convex at the nonzero vectors $\pm \bar{x}$.
\end{lemma}
	
\begin{proof}
We first demonstrate that the objective of the population problem $\calM(x)$  is locally strong convex at $ \bar{x}$.  This is equivalent to showing
that there exist positive scalars $\delta$ and $\gamma$ such that for any $x\in \mathbb{B}_{\delta}(\bar{x})$,
\[
\calM(x) - \calM(\bar{x}) \, \geq\, \frac{\gamma}{p} \,\| x - \bar{x}\|_2^2.
\]
Based on the expression of $\calM(x)$ in \eqref{eq:phase population2}, it suffices to show the following inequality for $x\in \mathbb{B}_{\delta}(\bar{x})$:
\begin{equation}\label{phase:proof1}
\begin{array}{ll}
 (1-\gamma)(\| \, x\|_2 -\|\,\bar{x}\, \|_2)^2        +
		2 (1+\gamma)\, \bar{x}^{\top} x  + 2(1-\gamma)\|x\|_2\|\bar{x}\|_2\\[0.15in]
	- \, \displaystyle{
			\frac{4}{\pi}
		} \left[ \, 2x^\top \bar{x} \,
		\arctan{\sqrt{\frac{\| x \|_2 \| \bar{x} \|_2 + \bar{x}^{\top}x}{
						\| x \|_2 \| \bar{x} \|_2 - \bar{x}^{\top}x}}} +
		\sqrt{\left(\bar{x}^{\top}x + \| x \|_2 \| \bar{x} \|_2\right)
			\left( \| x \|_2 \| \bar{x} \|_2 - \bar{x}^{\top}x \right)} \right] \geq 0.
			\end{array}
\end{equation}
To proceed, we denote by $\theta(x)$ the angle between $x$ and $\bar{x}$, i.e.,
\[
\cos \theta(x) = \frac{x^\top \bar{x}}{\|x\|_2 \|\bar{x}\|_2}.
\]
By shrinking the neighborhood $\mathbb{B}_{\delta}(\bar{x})$ if necessary, we may assume without loss of generality that $\theta(x)\in \left(-\displaystyle\frac{\pi}{2}, \, \displaystyle\frac{\pi}{2}\right)$.
Let $\gamma<1$ be arbitrary and \[\delta \, = \,\displaystyle\frac{1-\sin\left(\displaystyle\frac{\pi}{2}\gamma\right)}{1+\sin\left(\displaystyle\frac{\pi}{2}\gamma\right)}\;\|\bar{x}\|_2.\] Since $x\in \mathbb{B}_{\delta}(\bar{x})$, we have
\[
\begin{array}{rl}
\cos \theta(x) = \displaystyle\frac{x^\top \bar{x}}{\|x\|_2 \|\bar{x}\|_2}\geq \displaystyle\frac{\|x\|_2^2 + \|\bar{x}\|_2^2 - \delta}{2\|x\|_2\|\bar{x}\|_2} \geq \displaystyle\frac{(\|\bar{x}\|_2 - \delta)^2 + \|\bar{x}\|_2^2 - \delta^2}{2(\|\bar{x}\|_2 + \delta)\|\bar{x}\|_2} = 1-\frac{2\delta}{\|\bar{x}\|_2 + \delta} = \sin\left(\frac{\pi}{2}\gamma\right),
\end{array}
\]
which implies that $\theta(x)\in \left[-\displaystyle\frac{\pi}{2}(1-\gamma),\, \displaystyle\frac{\pi}{2}(1-\gamma)\right]$.
Direct computation shows that
\[\left\{\begin{array}{ll}
\arctan{\sqrt{\displaystyle\frac{\| x \|_2 \| \bar{x} \|_2 + \bar{x}^{\top}x}{
\| x \|_2 \| \bar{x} \|_2 - \bar{x}^{\top}x}}} = \arctan \sqrt{\displaystyle\frac{1+\cos \theta(x)}{1-\cos\theta(x)}} \, = \,
\arctan \left(\,\left|\,\cot \displaystyle\frac{\theta(x)}{2}\,\right|\,\right) = \displaystyle\frac{\pi}{2} - \frac{|\theta(x)|}{2}; \\[0.25in]
\sqrt{\left(\bar{x}^{\top}x + \| x \|_2 \| \bar{x} \|_2\right)
\left( \| x \|_2 \| \bar{x} \|_2 - \bar{x}^{\top}x \right)} = |\,\sin\theta(x)\,| \,\|x\|_2\|\bar{x}\|_2.
\end{array}\right.
\]
Therefore, the inequality \eqref{phase:proof1} is equivalent to
\[
(1-\gamma)(\| x\|_2 -\| \bar{x} \|_2)^2 + 2\| x\|_2\|\bar{x} \|_2\left(\underbrace{1-\gamma + \cos\theta(x)\left(\gamma -1 + \frac{2}{\pi}|\theta(x)|\,\right)
- \frac{2}{\pi}|\,\sin\theta(x)\,|\,}_{\mbox{denoted $q_{\gamma}\circ \theta(x)$}}\right)\geq 0.
\]
Notice that $q_{\gamma}(0)=0$ and
\[
q_\gamma^{\, \prime}(\theta) \, = \, \sin \theta \left(\,  1-\gamma -\frac{2\theta}{\pi} \,\right) \geq 0 \epc\mbox{if} \;\,
\theta\in \left(\, 0\,,\, \frac{\pi}{2}(1-\gamma) \,\right].
\]
Therefore, $q_{\gamma}(\theta)\geq 0$ for all $\theta\in \left[ \, 0, \displaystyle\frac{\pi}{2}(1-\gamma) \, \right]$. Since
$q_{\gamma}(\theta) = q_{\gamma}(-\theta)$, we further obtain that $q_{\gamma}(\theta)\geq 0$ for any
$\theta\in \left[-\displaystyle\frac{\pi}{2}(1-\gamma), \, \displaystyle\frac{\pi}{2}(1-\gamma)\right]$.
This proves the inequality \eqref{phase:proof1} for any $x\in \mathbb{B}_{\delta}(\bar{x})$.  Similarly one can show the local strong convexity
of $\calM$ near $-\bar{x}$.
\end{proof}
	
\begin{theorem}\label{thm: statistical accuracy}\rm
Let $x^N$ be a $\varepsilon$-strong d-stationary of phase retrieval problem \eqref{eq:phase empirical}.  Suppose there is no stationary solution on the
boundary of $X$ of (\ref{eq:phase population}), then
\[
\sqrt{N}\text{dist}(x^N, \calD_0) = O_{\IP_\infty}(1).
\]
\end{theorem}
\begin{proof}
		First, we check if Assumption \ref{ass for ULLN} holds. Under the setting of
		this phase retrieval problem, we know $f(x, \wt{\xi}) = 0$ and
		$g(x, \wt{\xi}) = \mbox{max} \left\{ \wt{\xi}^{\, \top} x, -\wt{\xi}^{\, \top} x
		\right\}$.  Then $\mbox{Lip}_f(\wt{\xi}) = 0$ and
		$\mbox{Lip}_g ( \wt{\xi} ) = \| \wt{\xi} \|_2 = 1$.  Assumption \ref{ass for ULLN} (a1)
		holds. It is clear that Assumption \ref{ass for ULLN} (a2) and (a3) hold because
		$\mbox{Lip}_{\nabla g}(\wt{\xi}) = 0 $ and $C_g(\wt{\xi}) = 1$.  In order to check
		Assumption \ref{ass for ULLN} (b), we can see
		\begin{equation}
		\begin{aligned}
		|h(t_1; \zbld) - h(t_2; \zbld)| = |t_1 + t_2 + 2 \zbld|\, |t_1 - t_2|.
		\end{aligned}
		\end{equation}
		Since we only consider $t_1 = \wt{\xi}^{\,\top} x_1$ and $\wt{\xi}^{\,\top} x_2$ for any $x_1, x_2 \in X$, we know $\mbox{Lip}_h (\zbld) = |t_1 + t_2 + 2 \zbld|$ is uniformly bounded. Thus Assumption \ref{ass for ULLN} holds. By Theorem \ref{th:main}, we know
		\[
		\IP_\infty \left(\lim_{N \rightarrow \infty}\text{dist}(x^N, \calD^\prime \cup \calD_0) = 0\right) =1
		\]
		Next, it is clear that Assumption \ref{assu: PA} holds as $g(x, \wt{\xi}\,)$ is a piecewise affine function. Then by Corollary \ref{cor: local min consistency}, suppose $x^N$ converges to $x^\infty$, as one of the elements in $\calD^\prime \cup \calD_0$, then $x^\infty$ must be a local minimizer of the problem \eqref{eq:phase population}. As we demonstrated in Proposition \ref{local is global for 11}, the set of all global minimizers is $\calD_0$, which is also the set of all local minimizers. Therefore, we can show that
		\[
		\IP_\infty \left(\lim_{N \rightarrow \infty}\text{dist}(x^N, \calD_0) = 0\right) =1.
		\]
		Next, we derive the convergence rate of $\text{dist}(x^N, \calD_0)$. It is enough to check if Assumption \ref{assu:rate}~(b1) holds. By Proposition \ref{pr:strong local}, we need to show there exist positive scalars $\delta$ and $c$ such that,
		\[
		R_{x^\infty;\varepsilon^{\prime}}(x,x^\infty)  - R_{x^\infty;\varepsilon^{\prime}}(x^\infty,x^\infty) \, \geq c\,\|x - x^\infty\|_2^2,
		\quad \forall\; x \, \in \, \mathbb{B}_\delta(x^\infty),
		\]
		where $x^\infty \in \calD_0$.
		By Lemma \ref{lemma:equiv epsilon}, it is equivalent to show
		\[
		\calM(x)  - \calM(x^\infty) \, \geq c\,\|x - x^\infty\|_2^2,
		\quad \forall\; x \, \in \, \mathbb{B}_\delta(x^\infty).
		\]
		This has been given by Lemma \ref{lm: strong local min}. Therefore, $ x^\infty \in \calD_0$ has the property of  local quadratic growth. By applying Theorem \ref{thm: final theorem for rate}, we can conclude the argument in the theorem that $\sqrt{N}\text{dist}(x^N, \calD_0) = O_{\IP_\infty}(1)$.

	\end{proof}
	
For the empirical phase retrieval problem \eqref{eq:phase empirical}, d-stationary
points can be obtained by the algorithm developed in \cite{CuiPangSen18}.  In what
follows, we report briefly the numerical results with the computational experiments
running this algorithm for solving (\ref{eq:phase empirical}) with various sample
sizes $N$.  Given the true signal $\bar{x}\in \mathbb{R}^{20}$ which we take to be
the vector of all ones, we generate samples $\{\xi_n\}_{n=1}^N$
from the uniform distribution on the sphere of a unit ball and compute the
corresponding  $z_n = |\bar{x}^\top \xi_n| + \varepsilon_n$ with $\varepsilon_n$ following
$\calN(0, 0.1)$. We first run the proposed algorithm in \cite{CuiPangSen18} with the
initial point in the set of all saddle points $\calD^{\, \prime}$.  Notice that many
developed algorithms in the existing literature requiring spectral initialization will fail in our numerical studies as the initial point is
orthogonal to the signal (\cite{ma2019optimization}).
We test the performance on various sample sizes ranging from 400 to 2000.
In the first figure below, it clearly shows that the computed empirical d-stationary
solutions are within the neighborhood of $\pm \bar{x}$.  Next, we compute the
$\ell_2$-distances between the computed empirical d-stationary solutions and
$\calD_0$ over 100 replications for various sample sizes. As we can see in the
second figure below, as the sample size $N$ increases, the $l_2$ error decreases
in the rate of nearly $\frac{1}{\sqrt{N}}$. This exactly matches our finding in
Theorem \ref{thm: statistical accuracy}. In addition, the objective values
$\calM_N(x_N)$ are around $0.01$, which is the specified noise level as
$\text{Var}[\varepsilon_n] = 0.01$ for $n = 1, \cdots, N$. Overall, our numerical findings are consistent with our developed theory.
\begin{figure}[H]
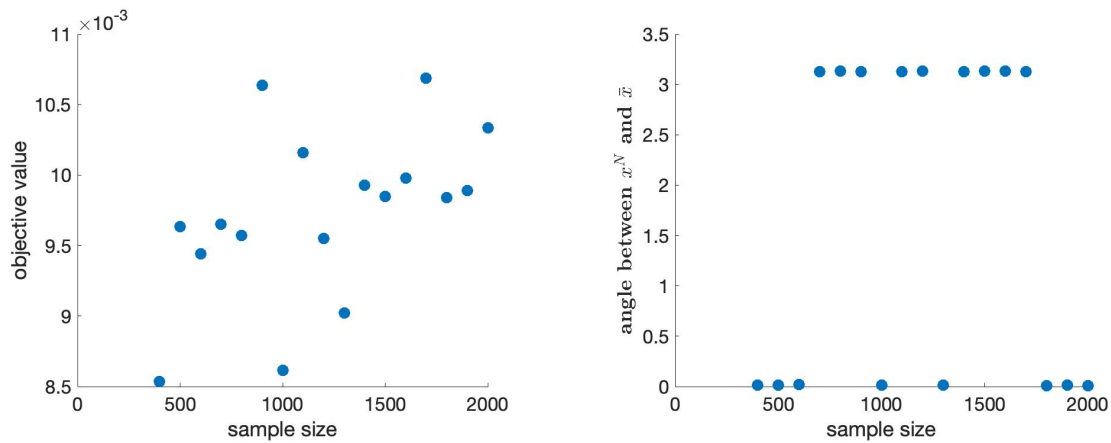

\begin{center}
\begin{minipage}{.45\textwidth}
\centering
\includegraphics[width=0.95\textwidth]{obj.pdf}
\end{minipage}
\epc	
\begin{minipage}{.45\textwidth}
\centering
\includegraphics[width=0.95\textwidth]{angle.pdf}
\end{minipage}
\end{center}
\caption{Results of the proposed algorithm in \cite{CuiPangSen18} on the phase retrieval problem \eqref{eq:phase empirical}. Left plot corresponds to the
stationary values of the computed d-stationary points. Right plot corresponds to angles between the computed d-stationary points and $\bar{x}$.
The initial points are all set to be in $\calD^\prime$. It is clear that the angles are close to either $0$ and $\pi$.}
\end{figure}
	
\begin{figure}[H]
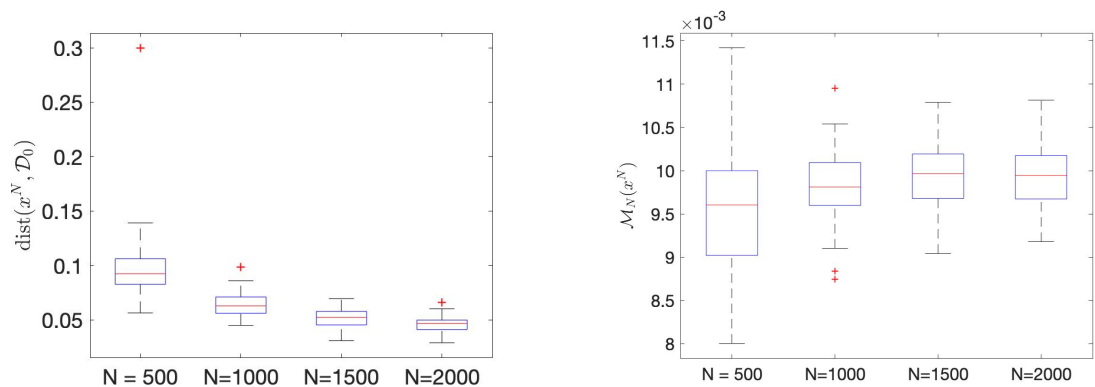

\begin{center}
\begin{minipage}{.45\textwidth}
\centering
\includegraphics[width=0.95\textwidth]{boxplotdist.pdf}
\end{minipage}
\epc
\begin{minipage}{.45\textwidth}
\centering
\includegraphics[width=0.95\textwidth]{boxplotobj.pdf}
\end{minipage}
\end{center}
\caption{Boxplots of $\ell_2$ errors between computed d-stationary solutions and $\pm \bar{x}$ and $\calM_N(x^N)$ for various sample sizes.
For each sample size, we repeat our simulation study for 100 times.}
\end{figure}
	
\section{Concluding Remarks}
	
Coupled nonconvex and nondifferentiable  statistical estimation problems present great challenges for both  rigorous computation and analysis.
Understanding and differentiating  properties of the computable solutions and establishing the asymptotics  of their statistical behaviors  are necessary tasks in
addressing such challenges. Our paper offers a first step in this direction by analyzing the relationship between a sharp kind of stationary solutions of the
empirical optimization problems and their population counterparts.  There remains much to be done, such as
the convergence rate and asymptotic distributions under relaxed assumptions and for general composite piecewise smooth estimation problems,
refined connections between solutions of various kinds of the empirical problems and their analogs in the  population formulations,
and finally understanding the desirable merits and undesirable drawbacks
of the  stationary points and values obtained from numerical optimization algorithms in nonconvex estimation processes.



\begin{thebibliography}{000}
		
		
		\bibitem{Attouch84}
		{\sc H.\ Attouch}.
		{\sl Variational Convergence for Functions and Operators}.
		Pitman Press, Boston (1984).
		
		\bibitem{Bagirov10}
		{\sc A.M.\ Bagirov, C.\ Clausen, and M.\ Kohler}.
		An algorithm for the estimation of a regression function by continuous piecewise linear functions.
		{\sl Computational Optimization and Applications} 45 (2010) 159--179.
		
		\bibitem{bartlett2002rademacher}
		{\sc P.\ Bartlet and S.\ Mendelson}.
		Rademacher and Gaussian complexities: Risk bounds and structural results.
		{\sl Journal of Machine Learning Research} 3 (2002) 463--482.
					
		\bibitem{bryc2012normal}
		{\sc W.\ Bryc}.
		{\sl The Normal Distribution: Characterizations with Applications.}
		Springer Science \& Business Media, 2012
		
		\bibitem{candes2015phase}
		{\sc E.J.\ Candes, X.\ Li, M.\ Soltanolkotabi}.
		Phase retrieval via Wirtinger flow: Theory and algorithms.
		{\sl IEEE Transactions on Information Theory} 61 (2015) 1985–-2007.
		
		\bibitem{chen2019gradient}
		{\sc Y.\ Chen, Y.\ Chi, J.\ Fan and C.\ Ma}.
		Gradient descent with random initialization: fast global convergence for nonconvex phase retrieval.
		{\sl Mathematical Programming} 176 (2019) 5--37.
		
		
		
		\bibitem{chernoff1954distribution}
		{\sc H.\ Chernoff}.
		On the distribution of the likelihood ratio.
		{\sl The Annals of Mathematical Statistics} (1954) 573--578.
		
		\bibitem{Chung54}
		{\sc K.\ Chung}.
		On a stochastic approximation method.
		{\sl  The Annals of Mathematical Statistics} (1954) 463--483.
		
		\bibitem{Clarke83}
		{\sc F.H.\ Clarke}.
		{\sl Optimization and Nonsmooth Analysis}.
		Classics in Applied Mathematics.
		SIAM, Volume 5, 1990.
		[Reprint from John Wiley (New York 1983).]
		
		\bibitem{CuiPangSen18}
		{\sc Y.\ Cui, J.S.\ Pang, and B.\ Sen}.
		Composite difference-max programs for modern statistical estimation problems.
		{\sl SIAM Journal on Optimization} 28 (2018) 3344--3374.
		
		\bibitem{CuiPangHongChang18}
		{\sc Y.\ Cui, T.H.\ Chang, M.\ Hong and J.S.\ Pang}.
		A study of piecewise linear-quadratic programs.
		{\sl arXiv:1709.05758} (2018).
		
		\bibitem{DentchevaPenevRucz17}
		{\sc D.\ Dentcheva, S.\ Penev, and A.\ Ruszczynski}.
		Statistical estimation of composite risk functionals and
		risk optimization problems.
		{\sl Annals of the Institute of Statistical Mathematics} 69 (2017) 737--760.
		
		\bibitem{Dima18}
		{\sc D.\ Davis, D.\ Drusvyatskiy and C.\ Paquette}.
		The nonsmooth landscape of phase retrieval.
		{\sl arXiv:1711.03247} (2018).
		
		
		\bibitem{duchi2018solving}
		{\sc J.\ Duchi and F.\ Ruan}.
		Solving (most) of a set of quadratic equalities: Composite optimization for robust phase retrieval.
		{\sl Information and Inference: A Journal of the IMA} 8 (2018) 471--529.
		
		
		
		
		\bibitem{DupacovaWets88}
		{\sc J.\ Dupa{\v c}ov{\'a} and R.J.-B.\ Wets}.
		Asymptotic behavior of statistical estimators and of optimal
		solutions of stochastic optimization problems.
		{\sl The Annals of Statistics} 16 (1988) 1517--1549.
		
		\bibitem{HannahDunson13}
		{\sc D.\ Dunson and L.A.\ Hannah}.
		Multivariate convex regression with adaptive partitioning.
		{\sl Journal of Machine Learning Research} 14 (2013) 3261--3294.
		
		\bibitem{FacchineiPang03}
		{\sc F.\ Facchinei and J.S.\ Pang}.
		{\sl Finite-dimensional Variational Inequalities and Complementarity Problems}.
		Springer, New York (2003).
		
		\bibitem{Ferguson96}
		{\sc T.S.\ Ferguson}.
		{\sl  A Course in Large Sample Theory}. Routledge (1996).
		
		\bibitem{ra1922mathematical}
		{\sc  R.A.\ Fisher}.
		On the mathematical foundations of theoretical statistics.
		{\sl Philosophical Transactions of the Royal Society A} 222 (1922) 594--604 .
		
		\bibitem{fisher1925theory}
		{\sc  R.A.\ Fisher}.
		Theory of statistical estimation.
		{\sl Mathematical Proceedings of the Cambridge Philosophical Society} 22 (1925)  700--725.
		
		\bibitem{geyer1994asymptotics}
		{\sc  C.J.\ Geyer}.
		On the asymptotics of constrained $M$-estimation.
		{\sl The Annals of Statistics} 22 (1994) 1993--2010.
		
		\bibitem{GlorotBordesBengio2011}
		{\sc X.\ Glorot, A.\ Bordes, and Y.\ Bengio}.
		Deep sparse rectifier neural networks.
		{\sl In  Proceedings of the 14th International Conference on Artificial Intelligence and Statistics}  (2011) 315--323.
		
		\bibitem{GurkanRobinson99}
		{\sc G.\ G{\"u}rkan, A.\ Yonca {\"O}zge and S.\ Robinson}.
		Sample-path solution of stochastic variational inequalities.
		{\sl Mathematical Programming} 84 (1999)  313--333.
		
		\bibitem{huber1967behavior}
		{\sc  P.J.\ Huber}.
		The behavior of maximum likelihood estimates under nonstandard conditions.
		{\sl Proceedings of the fifth Berkeley symposium on mathematical statistics and probability} (1967)
		Volume 1, Pages 221--233 University of California Press.
		
		\bibitem{JinJordan18}
		{\sc C.\ Jin, L.\ Liu, R.\ Ge and  M.\ Jordan}.
		On the local minima of the empirical risk.
		In {\sl Proceedings of Neural Information Processing Systems (NIPS)} 2018.
		
		\bibitem{king1987asymptotic}
		{\sc  A.J.\ King}.
		Asymptotic behaviour of solutions in stochastic optimization: nonsmooth analysis and the derivation of non-normal limit distributions.
		{\sl Ph.D.\ dissertation}, Department of Mathematics, University of Washington, Seattle (1993).
		
		\bibitem{KingRockafellar93}
		{\sc A.\ King and R.T.\ Rockafellar}.
		Asymptotic theory for solutions in statistical estimation and stochastic programming.
		{\sl Mathematics of Operations Research} 18 (1993) 148--162.
		
		
		\bibitem{lecam1970assumptions}
		{\sc  L.\ LeCam}.
		On the assumptions used to prove asymptotic normality of maximum likelihood estimates.
		{\sl The Annals of Mathematical Statistics} 41 (1970)  802--828.
		
		\bibitem{LeThiPham05}
		{\sc H.A.\ Le Thi and D.T.\ Pham}.
		The DC programming and DCA revised with DC models of real world nonconvex optimization problems.
		{\sl Annals of Operations Research} 133 (2005) 25--46.
		
		
		\bibitem{loh2011high}
		{\sc P.L.\ Loh and M.\ Wainwright}.
		High-dimensional regression with noisy and missing data: Provable guarantees with non-convexity.
		{\sl The Annals of Statistics} 40 (2012) 1637--1664.
		
		\bibitem{loh2015regularized}
		{\sc P.L.\ Loh and M.\ Wainwright}.
		Regularized M-estimators with nonconvexity: statistical and algorithmic theory for local optima.
		{\sl Journal of Machine Learning Research} 16 (2015) 559--616.
		
		\bibitem{loh2017statistical}
		{\sc P.L.\ Loh}.
		Statistical consistency and asymptotic normality for high-dimensional robust $ M $-estimators.
		{\sl The Annals of Statistics} 45 (2017) 866--896.
		
		\bibitem{LuZhouSun18}
		{\sc Z.\ Lu, Z.\ Zhou and Z.\ Sun}.
		Enhanced proximal DC algorithms with extrapolation for a class of structured nonsmooth DC minimization.
		{\sl Mathematical Programming} 176 (2019)  369--401.
		
		\bibitem{ma2019optimization}
		{\sc J.\ Ma, J.\  Xu and A. \ Maleki}.
		Optimization-based AMP for phase retrieval: the impact of initialization and $\ell_2$ regularization.
		{\sl  IEEE Transactions on Information Theory} 65 (2019) 3600--3629.
		
		\bibitem{MeiMontanari17}
		{\sc S.\ Mei, Y.\ Bai and A.\ Montanari}.
		The landscape of empirical risk for non-convex losses.
		{\sl The Annals of Statistics} 46 (2018) 2747--2774.
		
		\bibitem{Molchanov05}
		{\sc I.\ Molchanov}.
		{\sl Theory of Random Sets}. Vol. 19, no. 2.
		Springer, London (2005).
		
		\bibitem{robustSA09}
		{\sc A.\ Nemirovski, A.\ Juditsky, G.\  Lan and A.\ Shapiro}.
		Robust stochastic approximation approach to stochastic programming.
		{\sl SIAM Journal on  Optimization} 19 (2009) 1574--1609.
		
		\bibitem{NairHinton2010}
		{\sc V.\ Nair and G.E.\ Hinton}.
		Rectified linear units improve restricted Boltzmann machines.
		{\sl In Proceedings of the 27th International Conference on Machine Learning}  (2010) 807--814.
		
		
		\bibitem{PangRazaviyaynAlvarado16}
		{\sc J.S.\ Pang, M.\ Razaviyayn, and A.\ Alvarado}.
		Computing B-stationary points of nonsmooth DC programs.
		{\sl Mathematics of Operations Research} 42 (2016) 95--118.
		
		\bibitem{Polyak90}
		{\sc B.\ Polyak}.
		New stochastic approximation type procedures.
		{\sl Avtomatica i Telemekhanika} 7 (1990) 98--107.
		
		\bibitem{Polyak92}
		{\sc B.\ Polyak and A.\ Juditsky}.
		Acceleration of stochastic approximation by averaging.
		{\sl  SIAM Journal on  Control and Optimization} 30 (1992) 838--855.
		
		\bibitem{RobbinsMonro1951}
		{\sc H.\ Robbins and S.\ Monro}.
		A stochastic approximation method.
		{\sl Annals of  Mathematical Statistics} 22 (1951) 400--407.
		
		\bibitem{RockafellarConvex}
		{\sc R.T.\ Rockafellar}.
		{\sl Convex Analysis}.
		Princeton University Press, Princeton (1970).
		
		\bibitem{RockafellarRWets98}
		{\sc R.T.\ Rockafellar and R.J.-B.\ Wets}.
		{\sl Variational Analysis}. Springer, New York (1998).
		
		\bibitem{Royset19}
		{\sc J.\ Royset}.
		Approximations of semicontinuous functions
		with applications to stochastic optimization and statistical estimation.
		{\sl Mathmeatical Programming, Series A}, \url{https://doi.org/10.1007/s10107-019-01413-z}.
		
		\bibitem{RoysetWets18}
		{\sc J.\ Royset and R.J.-B.\ Wets}.
		Variational analysis of constrained M-estimators.
		http://arxiv.org/abs/1702.08109v4 (May 2018).
		
		\bibitem{Shechtman15}
		{\sc Y.\ Shechtman, Y.\ Eldar, O.\ Cohen, H.\ Chapman, J.\ Miao and M.\ Segev}.
		Phase retrieval with application to optical imaging: a contemporary overview.
		{\sl IEEE signal processing magazine} 32 (2015) 87--109.
		
		\bibitem{Scholtes02}
		{\sc S.\ Scholtes}.
		{\sl Introduction to Piecewise Differentiable Equations}.
		Springer Briefs in Optimization (2002).
		
		\bibitem{self1987asymptotic}
		{\sc  S.G.\ Self and K.Y.\ Liang}.
		Asymptotic properties of maximum likelihood estimators and likelihood ratio tests under nonstandard conditions.
		{\sl Journal of the American Statistical Association} 82 (1987)  605--610.
		
		\bibitem{Shapiro89}
		{\sc A.\ Shapiro}.
		Asymptotic properties of statistical estimators in stochastic programming.
		{\sl The Annals of Statistics} 17 (1989) 841--858.
		
		\bibitem{Shapiro03}
		{\sc A.\ Shapiro}.
		Monte Carlo sampling methods.
		In Rusczy{\'n}ski and Shapiro (eds.) Stochastic Programming, {\sl Handbooks in OR \& MS}.
		Volume 10.  Amsterdam: NorthHolland Publishing Company (2003).
		
		\bibitem{ShapiroDR09}
		{\sc A.\ Shapiro, D.\ Dentcheva, and A.\ Ruszczy{\'n}ski}.
		{\sl Lectures on Stochastic Programming: Modeling and Theory}.
		SIAM Publications (Philadelphia 2009).
		
		\bibitem{shapiro2007uniform}
		{\sc  A.\ Shapiro and H.\ Xu}.
		Uniform laws of large numbers for set-valued mappings and subdifferentials of random functions.
		{\sl Journal of Mathematical Analysis and Applications} 352 (2007)  1390--1399.
		
		
		\bibitem{van1998asymptotic}
		{\sc A.W.\ van der Vaart}.
		{\sl Asymptotic Statistics}
		Vol. 3, Cambridge University Press (1998).
		
		\bibitem{van1996weak}
		{\sc A.W.\ van der Vaart and J. \ Wellner}.
		{\sl Weak Convergence and Empirical Processes}
		Springer Series in Statistics (1996).
		
		
		\bibitem{van2000empirical}
		{\sc S.\ van de Geer}.
		{\sl Empirical Processes in M-estimation}
		vol. 6, Cambridge University Press (2000).
		
		\bibitem{Xu10}
		{\sc H.\ Xu}.
		Sample average approximation methods for a class of stochastic variational inequality problems.
		{\sl Asia-Pacific Journal of Operational Research} 27 (2010) 103--119.
		
		\bibitem{XuZhang09}
		{\sc  H.\ Xu and D.\ Zhang}.
		Smooth sample average approximation of stationary points in nonsmooth stochastic optimization and applications.
		{\sl Mathematical Programming} 119 (2009) 371--401.
		
		\bibitem{wald1949note}
		{\sc  A.\ Wald}.
		Note on the consistency of the maximum likelihood estimate.
		{\sl The Annals of Mathematical Statistics} 20 (1949) 595--601.
		
		\bibitem{wang2018solving}
		{\sc G.\ Wang, G.B.\ Giannakis, and Y.C.\ Eldar}.
		Solving systems of random quadratic equations via truncated amplitude flow.
		{\sl IEEE Transactions on Information Theory} 64 (2018) 773–794.
				
		\bibitem{Wets79}
		{\sc R.J.-B\ Wets}.
		A statistical approach to the solution of stochastic programs with (convex) simple recourse.
		Working Paper, University of Kentucky (1979).
		
		
		
		
	\end{thebibliography}
\end{document}